\documentclass[11pt,reqno,twoside]{article}

\usepackage[hang]{footmisc}
\usepackage{lipsum}

\usepackage{cmap} 

\usepackage[T1]{fontenc}
\usepackage[utf8]{inputenc}
\usepackage{graphicx}
\usepackage{array}
\usepackage{placeins}
\usepackage{enumerate}
\usepackage{algcompatible}
\usepackage[most]{tcolorbox}
\usepackage{mdframed}

\usepackage{verbatim}
\newcommand{\comments}[1]{}

\usepackage{soul}

\usepackage{setspace}

\let\counterwithin\relax  
\usepackage{lmodern} 

\usepackage{comment}

\usepackage{bm} 

\usepackage{bbold}

\usepackage{amsmath,amsbsy,amsgen,amscd,amsthm,amsfonts,amssymb}

\usepackage{float}       
\usepackage{algorithm}
\usepackage{algpseudocode}

\usepackage[centering,top=1.1in,bottom=1.4in,left=1in,right=1in]{geometry}

\usepackage[sf,bf,compact]{titlesec}

\newcommand{\mcG}{\mathcal{G}}
\newcommand{\mcA}{\mathcal{A}}
\newcommand{\mcC}{\mathcal{C}}
\newcommand{\mcJ}{\mathcal{J}}

\usepackage{enumitem}
\usepackage[dvipsnames]{xcolor}

\definecolor{dark-gray}{gray}{0.3}
\definecolor{dkgray}{rgb}{.4,.4,.4}
\definecolor{dkblue}{rgb}{0,0,.5}
\definecolor{medblue}{rgb}{0,0,.75}
\definecolor{rust}{rgb}{0.5,0.1,0.1}

\usepackage{url}
\usepackage[colorlinks=true]{hyperref}
\hypersetup{linkcolor=dkblue}    
\hypersetup{citecolor=rust}      
\hypersetup{urlcolor=rust}     

\usepackage[final]{microtype} 

\newtheoremstyle{myThm} 
    {\topsep}                    
    {\topsep}                    
    {\itshape}                   
    {}                           
    {\sffamily\bfseries}                   
    {.}                          
    {.5em}                       
    {}  

\newtheoremstyle{myRem} 
    {\topsep}                    
    {\topsep}                    
    {}                   
    {}                           
    {\sffamily}                   
    {.}                          
    {.5em}                       
    {}  

\newtheoremstyle{myDef} 
    {\topsep}                    
    {\topsep}                    
    {}                   
    {}                           
    {\sffamily\bfseries}                   
    {.}                          
    {.5em}                       
    {}  

\theoremstyle{myThm}
\newtheorem{theorem}{Theorem}[section]
\newtheorem{lemma}[theorem]{Lemma}
\newtheorem{proposition}[theorem]{Proposition}

\newtheorem{assumption}[theorem]{Assumption}
\newtheorem{definition}[theorem]{Definition}

  \newtheorem{examplex}[theorem]{Example}
 \newenvironment{example}
  {\pushQED{\qed}\examplex}
  {\popQED\endexamplex}

\theoremstyle{myRem}
\newtheorem{remarkx}[theorem]{Remark}
 \newenvironment{remark}
  {\pushQED{\qed}\remarkx}
  {\popQED\endremarkx}
\usepackage{fancyhdr}
\let\originalleft\left
\let\originalright\right
\renewcommand{\left}{\mathopen{}\mathclose\bgroup\originalleft}
\renewcommand{\right}{\aftergroup\egroup\originalright}

\usepackage{mathtools}
\mathtoolsset{centercolon}  

\providecommand{\mathbbm}{\mathbb} 

\newcommand{\R}{\mathbbm{R}}
\newcommand{\E}{\mathbbm{E}}

\newcommand{\mcD}{\mathcal{D}}
\newcommand{\mcM}{\mathcal{M}}
\newcommand{\mcB}{\mathcal{B}}

\newcommand{\bfF}{\mathbf{F}}
\newcommand{\bfU}{\mathbf{U}}
\newcommand{\bfW}{\mathbf{W}}
\newcommand{\bfA}{\mathbf{A}}
\newcommand{\bfM}{\mathbf{M}}
\newcommand{\bfB}{\mathbf{B}}
\newcommand{\bfC}{\mathbf{C}}
\newcommand{\bfD}{\mathbf{D}}
\newcommand{\bfX}{\mathbf{X}}
\newcommand{\bfu}{\mathbf{u}}
\newcommand{\bff}{\mathbf{f}}
\newcommand{\bfw}{\mathbf{w}}
\newcommand{\bfv}{\mathbf{v}}

\newcommand{\mcN}{\mathcal{N}}

\newcommand{\mcL}{\mathcal{L}}
\newcommand{\mcH}{\mathcal{H}}
\newcommand{\mcW}{\mathcal{W}}

\newcommand{\OVB}{\mathsf{OVB}}
\newcommand{\ErrorOVB}{\mathsf{Error\text{-}OVB}}
\newcommand{\Var}{\mathsf{Var}}
\newcommand{\ErrorVar}{\mathsf{Error\text{-}Var}}
\newcommand{\ErrorCompression}{\mathsf{Error\text{-}Compression}}
\newcommand{\ErrorTruncation}{\mathsf{Error\text{-}Truncation}}
\newcommand{\ErrorEstimation}{\mathsf{Error\text{-}Estimation}}
\newcommand{\OPS}{\mathsf{OPS}}
\newcommand{\supp}{\mathrm{supp}}

\newcommand{\poly}{\mathrm{poly}}
\newcommand{\polylog}{\mathrm{polylog}}

\definecolor{mygreen}{rgb}{0.13,0.55,0.13}

\newcommand{\M}{\mathcal{M}}

\usepackage[font = small, margin=30pt]{caption}

\usepackage[]{algorithm}
\usepackage{enumerate}

\usepackage{graphicx}

\usepackage{authblk}
\usepackage{chngcntr}
\usepackage{mathrsfs} 
\counterwithin{table}{section}
\counterwithin{algorithm}{section}

\usepackage[normalem]{ulem}

\usepackage{tikz}
\usetikzlibrary{shapes,arrows,decorations.pathmorphing,backgrounds,positioning,fit,petri}
\usetikzlibrary{calc}
\usetikzlibrary{matrix}
\usetikzlibrary{backgrounds}
\usetikzlibrary{shapes.geometric}
\usepackage{pgfplots}
\pgfplotsset{compat=newest}
\usepgfplotslibrary{groupplots}
\pgfdeclarelayer{background}
\pgfdeclarelayer{bBack}
\pgfsetlayers{bBack,background,main}
\usetikzlibrary{fit}

\title{Convergence Rates for Learning Pseudo-Differential Operators}

\author{Jiaheng Chen and Daniel Sanz-Alonso}

\date{University of Chicago}

\vspace{.5in}

\makeatletter\@addtoreset{section}{part}\makeatother%
\numberwithin{equation}{section}

\newcommand{\upperRomannumeral}[1]{\uppercase\expandafter{\romannumeral#1}}

\renewcommand{\hat}{\widehat}

\begin{document}
\maketitle 


\renewcommand{\thefootnote}{\fnsymbol{footnote}}

\vspace{-2em}
\abstract{This paper establishes convergence rates for learning elliptic pseudo-differential operators, a fundamental operator class in partial differential equations and mathematical physics. In a wavelet--Galerkin framework, we formulate learning over this class as a structured infinite-dimensional regression problem with multiscale sparsity.
Building on this structure, we propose a sparse, data- and computation-efficient estimator, which leverages a novel matrix compression scheme tailored to the learning task and a nested-support strategy to balance approximation and estimation errors. In addition to obtaining convergence rates for the estimator,
we show that the learned operator induces an efficient and stable Galerkin solver whose numerical error matches its statistical accuracy. Our results therefore contribute to  bringing together operator learning, data-driven solvers, and wavelet methods in scientific computing.}


\bigskip 

\section{Introduction}\label{sec:introduction}

Operator learning 
is an emerging paradigm at the intersection of scientific computing, partial differential equations (PDEs), and machine learning   \cite{li2020fourier,anandkumar2020neural,lu2021deepxde,lu2021learning,pathak2022fourcastnet,kovachki2023neural,li2024physics,boulle2024mathematical,kovachki2024operator,subedi2025operator}. In a supervised learning setting, the task can be formulated as follows: for an unknown operator $\mcA$, we observe noisy input--output pairs $\{(u_i,f_i)\}_{i=1}^{N}$ satisfying
\begin{align}\label{eq:intro_continuous_model}
f_i = \mcA u_i + w_i, \qquad 1 \le i \le N,
\end{align}
where $\{w_i\}_{i=1}^{N}$ models the noise. The goal is to estimate $\mcA$ accurately and efficiently from the data, under a prescribed error metric and modeling choices for the inputs, outputs, and noise.

This paper establishes convergence rates for learning elliptic pseudo-differential operators (PDOs), a broad and fundamental operator class in analysis and PDEs \cite{hormander2007analysis,shubin1987pseudodifferential,taylor2006pseudo,wong2014introduction}. PDOs encompass both differential and integral operators and naturally capture nonlocality and ill-conditioning; canonical examples include Green's operators, Dirichlet-to-Neumann maps, and boundary integral operators. We assume that the inputs $\{u_i\}_{i=1}^{N}$ and noise terms $\{w_i\}_{i=1}^{N}$ are Gaussian random functions (including Mat\'ern classes) with prescribed Sobolev smoothness, and we measure error in an operator norm between Sobolev spaces.

Motivated by the (quasi-)sparse representation of PDOs in wavelet coordinates and by wavelet--Galerkin methods in scientific computing \cite{beylkin1991fast,dahmen1997wavelet,schneider2013multiskalen,cohen2000wavelet,dahmen2006compression,cohen2001adaptive,cohen2004adaptive2,harbrecht2024multilevel}, we propose a sparse, data- and computation-efficient wavelet-based estimator $\widehat{\mcA}$, and show that, with high probability and up to poly-logarithmic factors,
\[
\|\widehat{\mcA}-\mcA\|_{H^{t}\to H^{-t'}} \;\lesssim\;  N^{-\frac{1}{2+\rho}}.
\]
The exponent $\rho\ge 0$ depends on the Sobolev indices $(t,t')$ in the error metric $\|\cdot\|_{H^{t}\to H^{-t'}}$, the order of the operator $\mcA$, the spatial dimension, the smoothness of the inputs and noise, and wavelet regularity/approximation order parameters. Our theory also covers the noiseless setting ($w_i\equiv 0$ in \eqref{eq:intro_continuous_model}), in which case the estimator achieves super-algebraic convergence. Finally, we show that $\widehat{\mcA}$ yields an efficient and stable wavelet--Galerkin solver whose numerical error matches its statistical accuracy, thereby bridging operator learning, data-driven PDE solvers, and wavelet methods in scientific computing.

\subsection{Main Contributions and Outline}
\begin{itemize}
\item 
\textbf{Learning PDOs in a wavelet--Galerkin framework.}
We formulate the study of learning PDOs in a wavelet--Galerkin framework, casting the operator learning task as the estimation of the operator’s bi-infinite wavelet matrix representation. By doing so, we reduce the continuous operator learning problem to a structured infinite-dimensional matrix regression problem.

\item \textbf{Structured infinite-dimensional regression with multiscale sparsity.}
We propose a sparse, data- and computation-efficient estimator for the wavelet--Galerkin discretization of the unknown operator by exploiting the (quasi-)sparse representation of PDOs in wavelet coordinates. Two key ingredients are:
\begin{enumerate}[label=(\roman*)]
    \item \emph{Learning-oriented matrix compression.}
    We introduce a new compression scheme tailored to the operator-learning error metric, identifying a target sparse support that captures the essential wavelet coefficients up to the desired accuracy. Our goal-oriented compression scheme enables sharper learning guarantees than classical  methods in scientific computing. 
    \item \emph{Nested-support regression.}
    To mitigate data-induced \emph{omitted-variable bias} when regressing on the target support, we fit each column on a carefully enlarged regression support and then restrict back to the target one. The enlargement is calibrated so that the omitted-variable bias is reduced to (at most) the level of the truncation and compression error, while the variance remains of the same order. This nested-support regression approach provides a principled way to balance data-induced bias against the approximation error from truncation and compression, and may be of independent interest for other structured learning problems. 
\end{enumerate}

\item \textbf{Convergence rates for learning elliptic PDOs.} Under general assumptions on the inputs, noise, and the choice of wavelets, our first main result, Theorem \ref{thm:main1}, establishes a nearly optimal high-probability error bound of order $N^{-\frac{1}{2+\rho}}$ (up to poly-logarithmic factors) for learning elliptic PDOs. The bound exhibits a \emph{parametric regime}, achieving the rate $N^{-1/2}$ (i.e., $\rho=0$) when the error metric is sufficiently weak; otherwise $\rho>0$, yielding a strictly \emph{nonparametric} rate. The estimator is optimally sparse in wavelet coordinates, containing only $\mathcal{O}(2^{Jn})$ nonzero entries, and can be computed in nearly $\mathcal{O}(N2^{Jn})$ time, where $2^{Jn}$ is the dimension of the learned Galerkin matrix and $N$ is the sample size; under the bias--variance optimal choice of the truncation level $J$ (as a function of $N$), this corresponds to $\mathcal{O}\big(N^{\frac{n}{(2+\rho)(t+t'-r)}}\big)$ nonzero coefficients and a total runtime of $\mathcal{O}\big(N^{1+\frac{n}{(2+\rho)(t+t'-r)}}\big)$, where $n$ is the spatial dimension and $r$ is the order of the PDO. The proof of Theorem~\ref{thm:main1} hinges on a sharp multiscale error analysis that controls truncation and compression errors, as well as omitted-variable bias and variance. A key technical contribution is to combine sharp probabilistic analysis with a refined characterization of the blockwise multiscale sparsity pattern induced by our new learning-oriented matrix compression scheme.  Our theory also covers the noiseless setting, where our estimator attains a super-algebraic convergence rate. As a corollary of our general results, we obtain in Example \ref{rem:example} what is, to our knowledge, the first explicit and nearly optimal tradeoff between statistical accuracy and computational cost for learning Green's functions of elliptic PDEs ---in general dimensions and for general differential order--- from generic noisy Gaussian data. We also sharpen existing accuracy–cost tradeoffs in the noiseless setting to a nearly optimal one.  

\item \textbf{Convergence rates for data-driven PDE solver.}
The learned sparse operator induces an efficient and stable wavelet--Galerkin solver. Our second main result, Theorem \ref{thm:main2}, shows that the solver’s numerical error inherits the statistical accuracy of the operator estimator (up to poly-logarithmic factors), and thus decays at an algebraic rate in the sample size $N$. Theorem \ref{thm:main2} therefore provides a principled bridge between operator learning, data-driven PDE solvers, and wavelet--Galerkin methods in scientific computing.
\end{itemize}

The paper is organized as follows. After discussing related work and setting notation in the rest of the introduction, we formalize our operator learning problem in Section \ref{sec:setup}. Then, in Section \ref{sec:waveletformulation}, we formulate operator learning in a wavelet--Galerkin framework. Section \ref{sec:define_estimator} defines our estimator using learning-oriented matrix compression and nested-support regression. Convergence rates for learning elliptic PDOs and for data-driven PDE solvers are established in Sections \ref{sec:convergencerate} and \ref{sec:stability}. Section \ref{sec:Conclusions} closes with conclusions and directions for future work.

\subsection{Related Work} \label{ssec:relatedwork}

\paragraph{Operator learning}
There is a vast and rapidly growing literature on operator learning, and a comprehensive overview is beyond the scope of this paper. Representative examples include neural-network-based approaches such as Fourier neural operators \cite{li2020fourier,pathak2022fourcastnet,kovachki2023neural}, DeepONets \cite{lu2021learning,lu2021deepxde,lanthaler2022error}, and physics-informed neural operators \cite{li2024physics,karniadakis2021physics}; kernel-based methods \cite{stepaniants2023learning,batlle2024kernel,jalalian2025data,long2024kernel,yang2025kernel,zhang2025minimax,kang2025optimal}; and random feature methods \cite{nelsen2021random,nelsen2024operator,liao2025cauchy,yu2025regularized}. Related lines of work include learning Green’s functions \cite{boulle2023learning,stepaniants2023learning,schafer2024sparse,boulle2022learning,boulle2023elliptic,wang2025operator}, operator learning for dynamical systems (e.g., Koopman operators and generators) \cite{brunton2022modern,kostic2022learning,kostic2023sharp,colbrook2024rigorous,kostic2024learning,llamazares2024data}, structured covariance and precision operator estimation \cite{al2025covariance,al2025covariance2,al2024optimal,chen2025precision}, and applications to inverse problems \cite{molinaro2023neural,gao2024adaptive,nelsen2025operator}. From a theoretical perspective, there has been substantial recent interest in approximation-theoretic foundations \cite{chen1995universal,lu2021learning,kovachki2021universal,lanthaler2022error,deng2022approximation,kovachki2023neural,lanthaler2023operator,lanthaler2023parametric,lanthaler2024operator,lanthaler2025nonlocality,de2025extension}, as well as in statistical and data-complexity analyses of operator learning problems \cite{mollenhauer2022learning,de2023convergence,jin2022minimax,kovachki2024data,liu2024deep,subedi2024controlling,reinhardt2024statistical,adcock2025towards}. Below, we focus on the works most closely related to the present paper, and we refer the reader to recent surveys for broader perspectives \cite{boulle2024mathematical,kovachki2024operator,subedi2025operator}.

The paper \cite{de2023convergence} studies linear operator learning from noisy data in the setting where the operator is diagonalizable in a known basis, thereby reducing the problem to eigenvalue learning, and establishes posterior contraction rates in a Bayesian framework. The work \cite{jin2022minimax} investigates the statistical limits of learning a Hilbert--Schmidt operator between two infinite-dimensional Sobolev reproducing kernel Hilbert spaces via ridge-regression-type estimators, deriving matching upper and lower bounds under the Hilbert--Schmidt norm; their analysis reduces to estimating an infinite-dimensional coefficient matrix with a polynomial-type entrywise decay profile. In contrast to \cite{de2023convergence,jin2022minimax}, we study the learning of elliptic PDOs under the \emph{operator norm} between Sobolev spaces. Such a worst-case guarantee, rather than an error measured in a more ``average'' sense, is directly compatible with stability and error analyses for data-driven PDE solvers and aligns naturally with classical wavelet--Galerkin methods. From a technical standpoint, our problem is intrinsically non-diagonal: the bi-infinite wavelet matrix of a PDO exhibits a multiscale (quasi-)sparse structure, rather than an entrywise polynomial decay profile, giving rise to a highly structured infinite-dimensional regression problem.

A related line of work \cite{boulle2023learning,boulle2023elliptic} considers the recovery of Green’s functions (solution operators) for elliptic PDEs via randomized numerical linear algebra, and obtains exponential convergence in the sample size by exploiting the low-rank structure on well-separated domains together with off-diagonal decay of the Green’s function. Their setting is noiseless (i.e., $w_i\equiv 0$ in \eqref{eq:intro_continuous_model}) and relies on carefully designed inputs aligned with the hierarchical low-rank structure of the Green's function, whereas our framework covers a unified class of integral and differential operators and allows noisy observations with generic Gaussian-process input data.  Moreover, their analysis is specialized to the three-dimensional case and their error bounds depend on the Hilbert--Schmidt norm of the solution operator, which can be infinite in higher dimensions since Green's functions may not be square-integrable. By contrast, our wavelet framework treats general dimensions in a unified way and works with general Sobolev-to-Sobolev operator norms.  More recently, \cite{reinhardt2024statistical} establishes general convergence guarantees for least-squares empirical risk minimizers over general operator classes in terms of their approximation properties and metric entropy bounds, via empirical process techniques. By comparison, we focus on a physically meaningful, PDE-inspired operator class and derive convergence rates through an explicit wavelet-based multiscale regression analysis that leverages the compressibility of PDOs in wavelet coordinates; moreover, our estimator is sparse and computationally efficient.

We also mention neural-network architectures inspired by PDOs, or designed to represent/approximate PDOs \cite{fan2019bcr,bubba2021deep,gupta2021multiwavelet,ying2022solving,shin2022pseudo,chen2024pseudo}. Our contribution is complementary: to the best of our knowledge, we provide the first rigorous statistical convergence rates for learning PDOs. While our estimator can be viewed through the lens of functional linear regression \cite{ramsay1991some,gupta2025optimal}, the key novelty is that the coefficient structure is governed by \emph{a priori} operator-analytic bounds on wavelet representations of PDOs, yielding a multiscale compressibility pattern that we exploit both statistically and computationally.

\paragraph{Wavelet--Garlekin methods}

Wavelet methods in scientific computing and numerical analysis have a long history; see, e.g., \cite{MadayPerrierRavel1991,beylkin1991fast,BacryMallatPapanicolaou1992,Daubechies1993DifferentPerspectives,dahmen1994wavelet,dahmen1993wavelet,dahmen1994multiscale,Harten1995,dahmen1997wavelet,cohen2001adaptive,cohen2002adaptive,harbrecht2006wavelet}. We refer the reader to the books and reviews \cite{Daubechies1993DifferentPerspectives,dahmen1997wavelet,cohen2000wavelet,Cohen2003,schneider2013multiskalen,stevenson2009adaptive}. More recently, related ideas have been deployed in other settings, such as multilevel approximation of Gaussian random fields \cite{herrmann2020multilevel,harbrecht2024multilevel,bachmayr2024multilevel}.
Three features of wavelets are particularly central in numerical analysis. First, wavelet systems characterize Sobolev (and Besov) spaces in the sense that, under suitable constructions, they form Riesz bases for a range of Sobolev spaces. This property yields simple \emph{diagonal preconditioners} for differential and integral operators discretized in wavelet bases via Galerkin methods \cite{dahmen1992multilevel,jaffard1992wavelet,dahmen1997wavelet}. Second, vanishing moments (cancellation property) allow wavelets to exploit operator smoothness, leading to \emph{sparse} or \emph{compressible} representations for broad classes of operators \cite{beylkin1991fast,dahmen1994wavelet,dahmen1993wavelet,dahmen1994multiscale,dahmen1997wavelet,stevenson2004compressibility}, in both standard and non-standard forms. Third, wavelets support adaptivity and nonlinear approximation, enabling optimal finite-term approximations and sparse representations of functions \cite{devore1992compression,devore1998nonlinear,cohen2001adaptive,cohen2002adaptive}. Together, these principles underpin fast, adaptive, multiscale numerical algorithms that leverage approximate sparsity. Beyond numerical analysis, wavelet ideas have had a lasting impact across applied mathematics and data science, including statistical estimation via thresholding \cite{donoho1994ideal,donoho1995adapting,donoho1996density,donoho2002noising}, data compression in image processing (e.g., JPEG 2000) \cite{cohen1992biorthogonal,Taubman2000EBCOT,SkodrasChristopoulosEbrahimi2001JPEG2000,TaubmanMarcellin2002JPEG2000Book}, compressed sensing \cite{donoho2006compressed,candes2008introduction}, and, more recently, connections to neural networks \cite{bruna2013invariant,mallat2016understanding}.

While much of the classical wavelet literature is devoted to numerical simulation and analysis, the present work studies operator learning through a wavelet lens, with the goal of establishing statistical convergence rates for learning an unknown operator from noisy input--output data pairs. Our estimator targets the bi-infinite wavelet matrix representation of the operator and crucially leverages diagonal preconditioning and the (quasi-)sparsity of PDOs in wavelet coordinates, while exhibiting several features that do not arise in classical wavelet–Galerkin PDE solvers. We emphasize the following distinctions from the classical wavelet--Galerkin methods. First, in contrast to the matrix compression schemes used in \cite{schneider2013multiskalen,dahmen2006compression,harbrecht2024multilevel}, we introduce a learning-oriented compression scheme that identifies the matrix entries to be estimated in a way tailored to the statistical learning task. Second, we develop a nested-support regression strategy to construct the estimator and mitigate the omitted-variable bias induced by the data ---an effect absent in the context of deterministic numerical discretizations. In addition, we incorporate a further symmetrization step that exploits the symmetry of the unknown operator to reduce variance. Third, our error analysis requires a refined characterization of the blockwise sparsity structure of the wavelet matrix: beyond the number of nonzero entries (nnz) at a global level, we quantify sparsity within each block of wavelet coefficients corresponding to a given pair of scales. This refinement is essential in operator learning because, under the operator norm, the variance of our estimator aggregates across scales in a way that is not captured by global nnz alone. Finally, in Section \ref{sec:stability} we discuss data-driven PDE solvers that use the learned sparse operator as a surrogate, further clarifying the connections to ---and distinctions from--- classical wavelet–Galerkin methods. Overall, this work contributes to bridging operator learning, data-driven PDE solvers, and wavelet–Galerkin methods.

\subsection{Notation}
Throughout this paper, $\mcM$ will denote a smooth, closed, and connected orientable Riemannian manifold of dimension $n$ immersed into Euclidean space $\R^D$ for some $D>n.$ 
The support of a real-valued function $\phi:\mcM\to \R$ is denoted by $\supp(\phi) := \overline{\{x\in\mcM : \phi(x)\ne 0\}}$, where the closure is taken in $\mcM$.  If for some subset $\mcM'\subset \mcM$ there exists a compact set $\mcM''$ such that $\mcM'\subset \mcM'' \subset \mcM$, we say that $\mcM'$ is compactly included in $\mcM$ and write $\mcM'\Subset\mcM$. The space of all smooth real-valued functions on $\mcM$ is denoted by $C^{\infty}(\mcM)$, and $C_0^{\infty}(\mcM)\subset C^{\infty}(\mcM) $ denotes the subspace of smooth functions $\phi$ whose support satisfies $\supp(\phi) \Subset \mcM$. The space $L^2(\mcM)$ consists of all square-integrable functions with respect to the intrinsic measure on $\mcM$, and its inner product is denoted by $(\cdot,\cdot)_{L^2(\mcM)}$. The Laplace–Beltrami operator on $\mcM$ is denoted by $\Delta_{\mcM}$. For any $s\in [0,\infty)$, we denote by $H^s(\mcM)$ the Sobolev space of order $s$ on $\mcM$, defined via the spectral decomposition of the Laplace–Beltrami operator; its norm is written $\|\cdot\|_{H^s(\mcM)}$. For $s>0$, we denote by $H^{-s}(\mcM)$ the dual space of $H^s(\mcM)$, and write $\langle \cdot,\cdot \rangle$ for the duality pairing between $H^s(\mcM)$ and $H^{-s}(\mcM)$. For notational simplicity, we often omit the dependence on $\mcM$ when referring to Sobolev spaces, 
and the dependence on $s$ when referring to duality pairings, whenever no ambiguity can arise. On occasion, we will work on an open bounded domain rather than a manifold, in which case we will make explicit the dependence of the function spaces on the domain. 

For a vector $x\in\R^m$, we denote by $\|x\|_2$ its Euclidean norm and by $\|x\|_{\max}$ its entry-wise maximum norm. For a matrix $B \in \R^{m_1\times m_2},$ we write $\|B\|$ for its spectral (operator) norm, $\|B\|_1$ its matrix $\ell_1$ norm (maximum absolute column sum), $\|B\|_{\infty}$ its matrix $\ell_{\infty}$ norm (maximum absolute row sum), and $\|B\|_{\max}$ for its entry-wise maximum norm. More generally, for a matrix $B\in \R^{\Lambda_{1}\times \Lambda_{2}}$ indexed by sets $\Lambda_{1}, \Lambda_2$ (possibly infinite), we interpret $B$ as a linear operator $\ell^2(\Lambda_{2})\to \ell^2(\Lambda_1)$ and let $\|B\|$ denote its operator norm.  For a symmetric positive semi-definite matrix, we denote by $\sigma_{\min}(B)$ and $\sigma_{\max}(B)$ its smallest and largest eigenvalues, respectively. For a finite set $\Lambda$, we write $|\Lambda|$ for its cardinality.

Given two positive sequences $\{a_k\}$ and $\{b_k\}$, we write $a_k \lesssim b_k$ if there exists a constant $c>0$, independent of $k$, such that $a_k \le c\, b_k$ for all $k$. If both $a_k \lesssim b_k$ and $b_k \lesssim a_k$ hold, we write $a_k \asymp b_k$. If the constant $c$ depends on some parameter $\tau$, we write $a_k \lesssim_{\tau} b_k,b_k\lesssim_{\tau} a_k$, and $a_k\asymp_{\tau} b_k$ to indicate this dependence.

\section{Problem Setup}\label{sec:setup}
 This section provides background on PDOs and formalizes our operator learning problem.
\subsection{Pseudo-Differential Operators}
Let $\mcD\subset \R^n$ be an open bounded domain. For $r \in \R,$ the class $\OPS^{r}(\mcD)$ of \emph{pseudo-differential operators} (PDOs) of order $r$ consists of all linear maps
\begin{align*}
\mcA: C_0^{\infty}(\mcD)\to C^{\infty}(\mcD),\quad (\mcA u)(x):=\frac{1}{(2\pi)^n}\int_{\R^n} e^{ix\cdot \xi} a(x,\xi)u^{\rm{ft}}(\xi)d\xi,
\end{align*}
where $u^{\rm{ft}}(\xi):=\int_{\mcD} u(x)e^{-ix\cdot \xi}dx$ denotes the Fourier transform of $u$, and $a(x,\xi)$ is a \emph{symbol} belonging to the H\"ormander class $\mathsf{S}^{r}(\mcD)$. The symbol class $\mathsf{S}^{r}(\mcD)$ consists of all functions $a\in C^{\infty}(\mcD\times \R^n) $ such that, for all $K \Subset \mcD$ and every pair of multi-indices $\alpha,\beta\in \mathbb{N}_0^n$, there exists a constant $C_{K,\alpha,\beta}>0$ with 
\begin{align*}
 |\partial^{\alpha}_{\xi} \partial^{\beta}_{x}a(x,\xi)|\le C_{K,\alpha,\beta} (1+|\xi|)^{r-|\alpha|},\quad \text{ for all }  x\in K, \, \xi\in \R^n.
\end{align*}
Here $\mathbb{N}_0^n$ denotes $n$-tuples of nonnegative integers, and 
$|\alpha| := \alpha_1+\cdots+\alpha_n.$ 
Intuitively, the order $r \in \R$ characterizes the regularity properties of operators in the class $\OPS^{r}(\mcD):$ positive order corresponds to differential operators, while negative order corresponds to integral (smoothing) operators.

For a manifold $\mcM$, the class $\OPS^{r}(\mcM)$ is defined in the usual way by lifting to $\mcM$ via local coordinates \cite{taylor2006pseudo}. Specifically, a linear operator $\mcB:C^{\infty}(\mcM)\to C^{\infty}(\mcM)$ is said to belong to $\OPS^{r}(\mcM)$ if, for any finite smooth partition of unity $\{\chi_i\in C_0^{\infty}(\mcM_i):i=1,\ldots,M\}$ with respect to an atlas $\{\mcM_i,\gamma_i\}_{i=1}^{M}$ of $\mcM$, all transported operators
\[
f\mapsto B_{i,i'}f := [(\mcB[\chi_i (f\circ \gamma_i^{-1} )])\chi_{i'}]\circ \gamma_{i'}
\]
belong to $\OPS^{r}\big(\gamma_i^{-1}(\mcM_i)\big)$ for all $i,i' = 1,\ldots,M$. $\OPS^{r}(\mcM)$ is thus invariantly defined and does not depend on the choice of atlas \cite{hormander2007analysis,taylor2006pseudo}. We refer the reader to \cite[Section A.2.2]{harbrecht2024multilevel} for more details. In this paper, we assume throughout that $\mcM$ is an $n$-dimensional smooth closed manifold, e.g.\ $\mcM=\mathbb{T}^n$ (the $n$-torus) and $\mcM=\mathbb{S}^n$ (the $n$-sphere); extensions of our results to domains with boundary can be obtained by imposing appropriate boundary conditions.

For $t, t' > r/2$, we will consider the \emph{operator norm} of $\mcA \in \OPS^{r}$ given by
\begin{align}\label{eq:operator_norm}
\|\mcA\|_{H^{t}\to H^{-t'}}:=\sup_{\|w\|_{H^{t}}\le 1}\|\mcA w\|_{H^{-t'}}=\sup_{\|w\|_{H^t}\le 1, \|v\|_{H^{t'}}\le 1}|\langle \mcA w, v \rangle|,
\end{align}
where the duality pairing $\langle \cdot,\cdot \rangle$ is taken between $H^{t'}$ and $H^{-t'}$.

\subsection{Operator Learning}\label{ssec:operatorlearning}
Consider the statistical model
\begin{align}\label{eq:continuous_model}
f_i = \mcA u_i + w_i, \qquad 1 \le i \le N,
\end{align}
where $\mcA$ is an unknown operator, $\left\{(u_i,f_i)\right\}_{i=1}^{N}$ are given data pairs, and $\{w_i\}_{i=1}^{N}$ are noise terms. 
We are interested in the \emph{operator learning problem}:
\begin{center}
\fbox{\parbox{0.92\linewidth}{
\centering
Given data pairs $\left\{(u_i,f_i)\right\}_{i=1}^{N}$, estimate the unknown operator $\mcA$ under the norm \eqref{eq:operator_norm}.
}}
\end{center}

We will work under the following standing assumption on the unknown operator $\mcA$, the input functions $\{u_i\}_{i=1}^{N}$, and the noise terms $\{w_i\}_{i=1}^{N}$.

\begin{assumption}[Operator, input function, and noise]\label{assumption:operator_data_noise}

\begin{enumerate}[label=(\roman*)]
    \item \textbf{Operator assumption:} For some given $r \in \R,$ $\mcA \in  \OPS^{r}(\mcM).$ Moreover, $\mcA$ is self-adjoint and positive definite in the sense that
\[
\langle \mcA v,v\rangle\gtrsim \|v\|^2_{H^{r/2}},\quad \forall \ v\in H^{r/2}.
\]
\item \textbf{Input function assumption:} The input functions $\{u_i\}_{i=1}^{N}$ are i.i.d. samples from a centered Gaussian process on $\mcM$ with covariance operator $\mcC_u\in \OPS^{-2r_1}(\mcM)$ for some given $r_1>n/2+\max\left\{0,r\right\}$. The operator $\mcC_u$ is self-adjoint and positive definite, satisfying
\[
\langle \mcC_u v,v\rangle\gtrsim \|v\|^2_{H^{-r_1}},\quad \forall \ v\in H^{-r_1}.
\]
\item \textbf{Noise assumption:} The noise terms $\{w_i\}_{i=1}^{N}$ are i.i.d. samples from a centered Gaussian process on $\mcM$, independent of $\{u_i\}_{i=1}^{N}$, with covariance operator $\mcC_w \in \mathsf{OPS}^{-2r_2}(\mcM)$ for some given $r_2 > n/2$.
The operator $\mcC_w$ is self-adjoint and positive definite, satisfying
\[
\langle \mcC_{w} v,v\rangle\gtrsim \|v\|^2_{H^{-r_2}},\quad \forall \ v\in H^{-r_2}.
\]
\end{enumerate}
\end{assumption}

Our theory also covers the noiseless data setting, where  $f_i = \mcA u_i, 1\le i \le N;$ see Remark \ref{rem:noisefree}. Notice that we assume the parameters $r,r_1,r_2$ that determine the order of $\mcA$ and of the input/noise covariance operators to be given. In practice, these parameters may need to be estimated from data  (see, e.g., \cite[Section 5.4]{harbrecht2024multilevel} and \cite{korte2025smoothness}). The study of adaptive estimators agnostic to these nuisance parameters is an interesting direction for future work.

\begin{remark}[Interpretation of Assumption \ref{assumption:operator_data_noise} and examples]\label{remark:assumption_1}
\noindent \textbf{Operator $\mcA$.} Under Assumption \ref{assumption:operator_data_noise} (i), the operator $\mcA$ is a continuous bijection from $H^{s}(\mcM)$ to $H^{s-r}(\mcM)$ for every $s\in \R$; see \cite[Proposition 11]{harbrecht2024multilevel}. Moreover, for any real exponent $\varrho$, the fractional power satisfies  $\mcA^{\varrho}\in \OPS^{\varrho r}(\mcM)$; see \cite[Proposition 12]{harbrecht2024multilevel}.  

A canonical class of examples is given by Schr\"odinger (Hamiltonian) elliptic operators and their fractional powers:
\[
\mcH:=-\Delta_{\mcM}+V(x),\qquad \mcA:=\mcH^{\varrho},
\]
where the potential $V\in C^{\infty}(\mcM)$ is uniformly positive, i.e.\ $V(x)\ge v_{\min}>0$. This includes both differential operators ($\varrho>0$) and integral/smoothing operators ($\varrho<0$). 

Assumption~\ref{assumption:operator_data_noise} (i) is also compatible with a broad class of self-adjoint elliptic boundary integral operators arising from PDEs. In particular, when
$\mcM=\partial\mcD$ is a closed boundary surface of an $(n+1)$-dimensional domain $\mcD\subset \R^{n+1}$, this class includes boundary integral equations of the first kind for the Laplacian as well as boundary integral formulations associated with the
Navier--Lam\'e and Stokes systems; see
\cite{dahmen2006compression,harbrecht2006wavelet,dahmen1997wavelet} for further background.
We believe that the results of this paper extend to broader classes of pseudo-differential and boundary integral operators, including operators that are not self-adjoint and/or not strongly elliptic. We leave these extensions to future work.

Another related (but slightly different) example is provided by infinitesimal generators of reversible diffusions. For the overdamped Langevin diffusion $dX_t=-\nabla U(X_t)\,dt+\sqrt{2/\eta}\,dW_t$, the generator $\mcG =-\nabla U\cdot\nabla +\eta^{-1}\Delta$ is a second-order elliptic differential operator (hence a PDO), but it is generally not self-adjoint in $L^2(\mcM)$. Instead, it is symmetric in the weighted space
$L^2_\pi(\mcM)$, where $\pi(x)\propto e^{-\eta U(x)}$ is the stationary measure. Indeed,
\[
\langle f,(-\mcG)g\rangle_\pi
=\eta^{-1}\int_\mcM \nabla f\cdot\nabla g\,\pi\,dx
=\langle (-\mcG)f,g\rangle_\pi .
\]
Hence $-\mcG$ is self-adjoint and positive semi-definite on $L^2_\pi(\mcM)$ (with a nullspace consisting of constants). We leave a detailed treatment of learning this class of operators using the techniques developed in this paper to future work.

\medskip
\noindent \textbf{Input function $u$ and noise $w$.} Assumption~\ref{assumption:operator_data_noise} (ii)–(iii) states that the input functions $\{u_i\}_{i=1}^{N}$ and the noise terms $\{w_i\}_{i=1}^{N}$ are centered Gaussian processes, sampled independently as $u_i \sim \mathcal N(0,\mcC_u), w_i \sim \mathcal N(0,\mcC_w).$ The assumptions on $\mcC_u$ and $\mcC_w$ imply that $u_i$ and $w_i$ admit variational representations as solutions of coloring-operator equations driven by spatial white noise, with coloring operators in the H\"ormander classes:
\[
\mcL_u u=\mcW_u,\qquad \mcL_w w=\mcW_w,
\]
where $\mcL_u=\mcC_u^{-1/2}\in\OPS^{r_1}(\mcM), \mcL_w=\mcC_w^{-1/2}\in \OPS^{r_2}(\mcM)$, and $\mcW_u,\mcW_w$ are independent Gaussian white noises on $L^2(\mcM)$; see \cite[(2.1), Proposition 1]{harbrecht2024multilevel}. This setting includes Whittle–Mat\'ern-type Gaussian random functions widely used in modeling and applications \cite{whittle1954stationary,matern1960spatial, williams2006gaussian,lindgren2011explicit,stein2012interpolation}. For instance, one may take
\[
\mcL_u:=(\tau_u(x)I-\nabla\cdot(\kappa_u(x)\nabla))^{r_1/2},\qquad \mcL_w:=(\tau_w(x)I-\nabla\cdot(\kappa_w(x)\nabla))^{r_2/2},
\]
where $\tau_u,\tau_w\in C^{\infty}(\mcM)$ are uniformly positive, i.e. $\tau_u(x)\ge \tau_{u,\min}>0$ and $\tau_w(x)\ge \tau_{w,\min}>0$, and $\kappa_u,\kappa_w$ are smooth and uniformly elliptic (e.g. smooth symmetric positive-definite matrix fields); see \cite{sanz2022spde} and \cite[Appendix~C]{harbrecht2024multilevel}.

The parameters $r_1,r_2$ control Sobolev regularity. Spectral asymptotics of $\mcC_u$ and $\mcC_w$ (via Weyl's law), see \cite{cox2020regularity,herrmann2020multilevel,harbrecht2024multilevel}, imply that
\[
u_i \in H^{s}(\mcM) \quad \text{ for all } s<r_1-n/2 \quad \text{a.s.},\qquad  w_i \in H^{s}(\mcM) \quad \text{ for all } s<r_2-n/2 \quad \text{a.s.}
\]
Thus, the conditions $r_1>n/2$ and $r_2>n/2$ ensure $u_i,w_i\in L^2(\mcM)$ almost surely. Moreover, since $\mcA:H^{s}(\mcM)\to H^{s-r}(\mcM)$ is an isomorphism for every $s\in \R$, we have
\[
\mcA u_i \in H^{s-r}(\mcM) \quad \text{ for all } s<r_1 - n/2 \quad \text{a.s.},
\]
so $\mcA u_i\in L^2(\mcM)$ whenever $r_1 - n/2 - r > 0$. Combining these observations, under the assumptions
\[
r_1>n/2+\max\left\{0,r\right\},\qquad r_2>n/2,
\]
all functions in the statistical model \eqref{eq:continuous_model} ---namely $u_i, w_i$, and $f_i=\mcA u_i+w_i$--- belong to $L^2(\mcM)$ almost surely for every $1\le i\le N$.
\end{remark}

\section{Wavelet–Matrix Formulation of Operator Learning}\label{sec:waveletformulation}
In this section, we recast the operator learning problem in terms of learning a bi-infinite wavelet matrix representation. To that end, we discretize the continuous operator-function model \eqref{eq:continuous_model} into a discrete, infinite-dimensional matrix-vector model, using a biorthogonal wavelet system.  

\subsection{Biorthogonal Wavelets}\label{subsec:wavelets_MRA}
Here, we briefly summarize the biorthogonal wavelet framework used in this paper; additional details, properties, and a sketch of the construction of the associated multiresolution analyses (MRAs) can be found in Appendix~\ref{appendix:wavelets}. Let 
\[
\Psi=\{\psi_{j,k}: j\ge j_0,\, k\in\nabla_j\},\qquad \widetilde{\Psi}=\{\widetilde{\psi}_{j,k}: j\ge j_0,\, k\in\nabla_j\}
\]
denote a pair of biorthogonal wavelet bases of $L^2(\mcM)$ with  $|\nabla_j|\asymp 2^{jn}.$ The indices $\lambda = (j,k)$ encode the information about scale $(j)$ and location $(k)$.
We introduce the wavelet index set 
\[
\mcJ:=\{\lambda=(j,k): j\ge j_0,\, k\in\nabla_j\}.
\]
Any function $u\in L^2(\mcM)$ admits expansions in the primal and dual bases given by
\begin{align*}
  u=\bfu^{\top}\Psi=\sum_{\lambda\in\mcJ}u_{\lambda}\psi_{\lambda},
\qquad 
 u = \widetilde{\bfu}^{\top}\widetilde{\Psi}=\sum_{\lambda\in\mcJ}\widetilde{u}_{\lambda} \widetilde{\psi}_{\lambda}, 
\end{align*}
where $u_{\lambda}:=\langle u,\widetilde{\psi}_{\lambda}\rangle$ and $\widetilde{u}_{\lambda}:=\langle u,\psi_{\lambda}\rangle.$

The \emph{bi-infinite wavelet matrix representation} $\bfA \in \R^{\mcJ\times \mcJ}$ of $\mcA\in\OPS^{r}(\mcM)$ is defined by
\[\bfA_{\lambda,\lambda'}:=\langle \mcA\psi_{\lambda'},\psi_{\lambda}\rangle,
\quad \forall \ \lambda, \lambda'\in\mcJ .
\]
The action of $\mcA$ on $u$ can be expressed in the dual basis by
\begin{align*}
\mcA u=\sum_{\lambda\in \mcJ} (\bfA \bfu)_{\lambda}\widetilde{\psi}_{\lambda}.
\end{align*}
Using biorthogonality, the pairing of $\mcA u$ with  $v=\bfv^{\top}\Psi =\sum_{\lambda\in\mcJ}v_{\lambda}\psi_{\lambda}$ is given by
\[
\langle \mcA u,v \rangle=\bigg\langle \sum_{\lambda\in\mcJ}(\bfA\bfu)_{\lambda}\widetilde{\psi}_{\lambda},\sum_{\lambda\in \mcJ} v_{\lambda}\psi_{\lambda} \bigg\rangle=\langle \bfA\bfu,\bfv \rangle=\bfv^{\top}\bfA\bfu.
\]
Consequently, $\mcA$ admits the (formal) wavelet expansion
\begin{align}\label{eq:mcA_expansion}
\mcA=\sum_{\lambda,\lambda'\in \mcJ} \bfA_{\lambda,\lambda'}\, \widetilde{\psi}_{\lambda}\otimes \widetilde{\psi}_{\lambda'}=\sum_{\lambda,\lambda'\in \mcJ} \langle \mcA \psi_{\lambda'},\psi_{\lambda} \rangle \widetilde{\psi}_{\lambda}\otimes \widetilde{\psi}_{\lambda'}.
\end{align} 

The operator norm of $\mcA \in \OPS^r$ defined in \eqref{eq:operator_norm} is equivalent to a discrete norm for its bi-infinite matrix representation $\bfA$. This equivalence between continuous and discrete norms, and the related idea of preconditioning,  will play a key role in our theory.  
Specifically, for each index $\lambda=(j,k)$, let $|\lambda|:=j$ and define the bi-infinite diagonal scaling matrix
\[
\bfD^s:=\mathrm{diag}\Big(2^{s|\lambda|}:\lambda\in \mcJ\Big),\qquad s\in \R.
\]
For $w= \bfw^{\top}\Psi\in H^t(\mcM)$ and $v=\bfv^{\top}\Psi\in H^{t'}(\mcM)$, the wavelet characterization of Sobolev norms (Lemma~\ref{lemma:wavelet_property} (iv)) yields
\[
\|w\|^2_{H^t}\asymp \|\bfD^t \bfw\|_2^2,\qquad \|v\|^2_{H^{t'}}\asymp \|\bfD^{t'} \bfv\|_2^2, \qquad \text{if }\, t,t'\in(-\widetilde{\gamma},\gamma), 
\]
where $\gamma$ and $\widetilde{\gamma}$ are primal/dual regularity parameters. Hence, the norm \eqref{eq:operator_norm} can be expressed as
\begin{align}\label{eq:matrix_norm}
\|\mcA\|_{H^{t}\to H^{-t'}}&=\sup_{\|w\|_{H^t}\le 1, \|v\|_{H^{t'}}\le 1}|\langle \mcA w, v \rangle|\nonumber\\
&\asymp \sup_{\|\bfD^t \bfw\|_2\le 1, \|\bfD^{t'} \bfv\|_2\le 1} |\langle \bfA\bfw,\bfv\rangle| \nonumber\\
&=\sup_{\|\bfD^t \bfw\|_2\le 1, \|\bfD^{t'} \bfv\|_2\le 1} |\langle \bfD^{-t'}\bfA\bfD^{-t}\bfD^t\bfw,\bfD^{t'}\bfv\rangle|\nonumber\\
&= \|\bfD^{-t'}\bfA\bfD^{-t}\|,
\end{align}
where $\|\cdot\|$ in the last line denotes the operator norm from $\ell^2(\mcJ)$ to $\ell^2(\mcJ)$. Notice that since $\mcA$ is self-adjoint under our standing Assumption \ref{assumption:operator_data_noise}, $\bfA$ is symmetric. Hence, we may assume that $t \le t'$ in what follows without loss of generality; see Remark \ref{rem:ttprime}.

 As discussed in \cite{dahmen1997wavelet} and overviewed in Appendix \ref{appendix:wavelets}, a biorthogonal system $(\Psi,\widetilde{\Psi})$ can be
characterized by four parameters
\[
\gamma,\ \widetilde\gamma \quad\text{(primal/dual regularity)}, 
\qquad 
d,\ \widetilde d 
\quad\text{(primal/dual approximation order)}.
\]
The parameters $\gamma,\widetilde\gamma$ control the Sobolev
regularity of the primal and dual wavelets, while $d, \widetilde d$ correspond to the
number of vanishing moments and determine approximation accuracy. 
 For convenience, we now collect all the conditions on wavelet parameters that we will require. In addition to imposing constraints on the wavelet parameters, the following assumption specifies the admissible range for a parameter $\sigma$ used in our estimation procedure and error bounds. This parameter is related to the off-diagonal decay of the wavelet coefficients of $\mcA$ across scale separation, as will become clear in Proposition \ref{app:proof_property_UAW} \ref{approximate_sparsity} (ii). 
\begin{assumption}[Wavelets]\label{assumption:wavelets}
We are given a pair $(\Psi,\widetilde{\Psi})$ of biorthogonal wavelet systems with regularities $(\gamma,\widetilde{\gamma})$ and approximation orders $(d,\widetilde{d})$ satisfying:
\begin{enumerate}[label=(\roman*)]
    \item $ \frac{r}{2}, \, r_1, \, -r_2, \, t,\, t' \in (-\widetilde{\gamma}, \gamma).$  
      \item  $d>\max\{t,t'\}$, $\widetilde{d} > -n/2-r/2$. 
    \item $\min\left\{\gamma-\frac{r}{2},\widetilde{\gamma}+\frac{r}{2}, \widetilde{d}+\frac{n}{2}+\frac{r}{2}\right\}>\sigma >
\max\left\{\frac{n}{2}+\max\{t,t'\}-\frac{r}{2},
\frac{3n}{2} - t + \frac{r}{2},\ \frac{t'+\max\{t',r_1\}-r}{\min\{t',r_1\}+t-r}n
\right\}$.
\end{enumerate}
\end{assumption}

Throughout, we treat the parameters $n,r,r_1,r_2,t,t'$ as fixed.
Conditions~(i)--(iii) in Assumption~\ref{assumption:wavelets} can be satisfied by choosing biorthogonal wavelets with sufficient smoothness and sufficiently many vanishing moments,
i.e., by taking $\gamma,\widetilde{\gamma},d,\widetilde{d}$ sufficiently large relative to these fixed parameters.

\subsection{Operator Learning in the Bi-Infinite Matrix Framework}

Recall the statistical model \eqref{eq:continuous_model}. 
In the dual basis we can write, for $1 \le i \le N,$
\begin{align}\label{eq:output_expansion}
f_i = \widetilde{\bff}_i^{\top}\widetilde{\Psi}= \sum_{\lambda\in \mcJ}\langle f_i,\psi_{\lambda}\rangle \widetilde{\psi}_{\lambda},\quad \mcA u_i=\sum_{\lambda\in \mcJ} (\bfA \bfu_i)_{\lambda}\widetilde{\psi}_{\lambda}, \quad  w_i= \widetilde{\bfw}_i^{\top}\widetilde{\Psi}=\sum_{\lambda\in \mcJ}\langle w_i,\psi_{\lambda}\rangle \widetilde{\psi}_{\lambda}.
\end{align} 
Hence,  \eqref{eq:continuous_model} admits the following infinite-dimensional matrix–vector representation:
\[
\widetilde{\bff}_i=\bfA \mathbf{u}_i+\widetilde{\bfw}_i,\qquad 1\le i\le N,
\] 
where $\bfA\in \R^{\mcJ \times \mcJ}$ is the bi-infinite matrix representation of $\mcA$, and $\widetilde{\bff}_i,\widetilde{\bfw}_i\in \R^{\mcJ}$ are the wavelet coefficients of the output and noise functions $f_i$ and $w_i$, respectively. In compact form,
\begin{align}\label{eq:discrete_model}
\bfF=\bfU\bfA+\bfW,
\end{align}
where we used that $\bfA$ is symmetric since $\mcA$ is self-adjoint, and we defined
\[
\bfF:=[\widetilde{\bff}_1, \ldots, \widetilde{\bff}_N]^{\top} \in \R^{N\times \mcJ}, \quad \bfU:=[\bfu_1,\ldots, \bfu_N]^{\top}\in \R^{N\times \mcJ}, \quad \bfW :=[\widetilde{\bfw}_1,\ldots, \widetilde{\bfw}_N]^{\top}\in \R^{N\times \mcJ}.
\]

Thus, we have derived the following \emph{bi-infinite matrix learning problem}: 
\begin{center}
\fbox{\parbox{0.92\linewidth}{
\centering
Given $\bfU$ and $\bfF$, estimate the unknown bi-infinite matrix $\bfA$ 
under the norm \eqref{eq:matrix_norm}.
}}
\end{center}
This bi-infinite matrix learning problem will play a crucial role in our operator learning theory: we will construct an estimator of $\mcA$ from an estimator of its bi-infinite matrix $\bfA,$ and we will leverage the equivalence of the norms \eqref{eq:operator_norm} and \eqref{eq:matrix_norm} in our analysis.  
We postpone the construction of the estimator to Section \ref{sec:define_estimator} and its error analysis to Section \ref{sec:convergencerate}. Both the construction and the analysis leverage the following properties of the bi-infinite matrix $\bfA$, the data matrix $\bfU$, and the noise matrix $\bfW$ ensuing from  Assumption~\ref{assumption:operator_data_noise} and the wavelet discretization induced by $(\Psi,\widetilde{\Psi})$. These properties are standard \cite{dahmen1997wavelet,dahmen2006compression,schneider2013multiskalen,harbrecht2024multilevel}, but for completeness we include a brief formal derivation in Appendix~\ref{app:proof_property_UAW}.

\begin{proposition}[Properties of $\bfA,\bfU,\bfW$]\label{prop:property_UAW}     Under Assumption \ref{assumption:operator_data_noise} and Assumption \ref{assumption:wavelets} (i)-(ii), the matrices $\bfA, \bfU$, and $\bfW$ satisfy:
    \begin{enumerate}[label=(\Roman*)]
        \item \label{approximate_sparsity} \textbf{Approximate sparsity of $\bfA$:}  For $(j,k)\in \mcJ$, we define $S_{j,k}:=\mathrm{conv \, hull}\bigl(\mathrm{supp}(\psi_{j,k})\bigr)\subset \M$ the convex hull of the support of $\psi_{j,k}$.        
        \begin{enumerate}[label=(\roman*)]
            \item For all $\lambda=(j,k),\lambda'=(j',k')\in \mcJ$ such that $\mathrm{dist}(S_{j,k},S_{j',k'})\gtrsim 2^{-\min\{j,j'\}}$,
        \begin{align*}
        |\bfA_{\lambda,\lambda'}| \lesssim 2^{-(j+j')(\widetilde{d}+n/2)}\mathrm{dist}(S_{j,k},S_{j',k'})^{-(n+r+2\widetilde{d})}.
        \end{align*}
        \item For all $\lambda=(j,k),\lambda'=(j',k')\in \mcJ$ such that $\mathrm{dist}(S_{j,k},S_{j',k'})\lesssim 2^{-\min\{j,j'\}}$, and for any $0<\sigma<\min\left\{\gamma-r/2,\widetilde{\gamma}+r/2\right\}$,
\[
|\bfA_{\lambda,\lambda'}| \lesssim 2^{(j+j')r/2} 2^{-\sigma|j-j'|}.
\]
        \end{enumerate}
        \item \label{diagonal_preconditioning} \textbf{Diagonal preconditioning:}
There exist constants $0<c_{-}<c_{+}< \infty$ such that: 
 \begin{enumerate}[label=(\roman*)]
        \item \label{preconditioning_operator}  $\bfA$ is a symmetric and positive definite operator on $\ell^2(\mcJ)$, and 
\[
c_{-}\le \sigma_{\min}(\bfD^{-r/2} \bfA \bfD^{-r/2})\le \sigma_{\max}(\bfD^{-r/2} \bfA \bfD^{-r/2})\le c_{+}.
\]
\item \label{preconditioning_data} 
The population covariance matrix of the input data, $\widetilde{\bfC}_u:=\E \left[\bfU^{\top}\bfU/N\right]$, is a symmetric, positive definite, and compact operator on $\ell^2(\mcJ)$. Furthermore, 
\[
c_{-}\le \sigma_{\min}(\bfD^{r_1} \widetilde{\bfC}_u \bfD^{r_1})\le \sigma_{\max}(\bfD^{r_1} \widetilde{\bfC}_u \bfD^{r_1})\le c_{+}.
\]
\item \label{preconditioning_noise} The population covariance matrix of the noise, $\bfC_w:=\E \left[\bfW^{\top}\bfW/N\right]$, is a symmetric, positive definite, and compact operator on $\ell^2(\mcJ)$. Furthermore, 
\[
c_{-}\le \sigma_{\min}(\bfD^{r_2} \bfC_w \bfD^{r_2})\le \sigma_{\max}(\bfD^{r_2} \bfC_w \bfD^{r_2})\le c_{+}.
\]    
\end{enumerate}
\end{enumerate}
\end{proposition}

\begin{remark}
Combining Proposition~\ref{prop:property_UAW} \ref{approximate_sparsity} (i)–(ii) with Assumption~\ref{assumption:wavelets} (ii), we obtain the following uniform bound for all $\lambda=(j,k), \lambda'=(j',k')\in \mcJ$:
\[
|\bfA_{\lambda,\lambda'}|\lesssim 2^{(j+j')r/2}\cdot 2^{-\sigma|j-j'|}\left(1+2^{\min\{j,j'\}}\mathrm{dist}(S_{j,k},S_{j',k'})\right)^{-(n+r+2\widetilde{d})}.
\]
Such estimates are classical in the wavelet–Galerkin literature and underpin the construction of efficient numerical schemes; see, for example, \cite[(7.11), (9.27)]{dahmen1997wavelet} and \cite[(2.28), (2.29)]{cohen2001adaptive}.
\end{remark}

\section{Construction of the Estimator}\label{sec:define_estimator}

In this section, we construct our estimator for the unknown operator $\mcA$ in two stages. In the first stage, we use the \emph{a priori} estimates from Proposition \ref{prop:property_UAW} to identify a subset of significant entries of the bi-infinite matrix representation of $\mcA.$ In the second stage, we estimate these significant entries using the data. The final estimator of $\mcA$ will take the form
\begin{align}\label{eq:estimator_pre1}
\widehat{\mcA}:=\sum_{(\lambda,\lambda')\in \supp(J,t,t')} \widehat{\bfA}_{\lambda,\lambda'}\, \widetilde{\psi}_{\lambda}\otimes \widetilde{\psi}_{\lambda'}.
\end{align}

We will describe the two stages in turn. First, in Subsection \ref{ssec:compression} we introduce the truncation and compression procedure underlying the definition of the set $\supp(J,t,t')$ of significant entries to be estimated. Next, in Subsection \ref{ssec:estimation} we define the entrywise estimates $\widehat{\bfA}_{\lambda,\lambda'}$ for $(\lambda, \lambda') \in \supp(J,t,t')$ using a nested-support regression strategy.

\subsection{First Stage: Truncation and Compression}\label{ssec:compression}

This subsection introduces the set $\supp(J,t,t')$ of significant entries to be estimated. The key idea is that \emph{any} bi-infinite matrix $\bfA$ satisfying Proposition \ref{prop:property_UAW} can be well approximated by truncation into a finite-dimensional matrix and subsequent matrix compression. The set $\supp(J,t,t')$ corresponds to the non-zero entries of the compressed matrix, and is independent of $\bfA.$ We emphasize that the truncation and compression procedures do not make use of the given data, and are based on \emph{a priori} estimates. 

Matrix truncation crops the operator discarding fine scales, as formalized in the following:
\begin{definition}[Matrix truncation]\label{def:matrix_truncation}
    For $J>j_0,$ consider the set of wavelet indices with level at most $J,$ given by \[
\Lambda_{J}:=\left\{(j,k):j_0\le j\le J,k\in \nabla_{j} \right\}.
\]
Given a bi-infinite matrix $\bfA,$ we define its \emph{a priori} truncation
    $\bfA_{\Lambda_J}:=(\bfA_{\lambda,\lambda'})_{\lambda,\lambda'\in \Lambda_J}  \in \R^{\Lambda_{J} \times \Lambda_{J}}.$
\end{definition}
 In our later developments, truncation causes a bias in estimation of $\mcA$ since we will only estimate (a subset of) entries of $\bfA$ up to scale level $J.$ The choice of $J$ in terms of the sample size $N$ will be determined through a bias–variance tradeoff. We will characterize the error induced by matrix truncation in Proposition \ref{prop:Error-Truncation} below.

Compression exploits the sparsity of the bi-infinite matrix representation $\bfA$ of the PDO $\mcA$ in the wavelet basis. In particular, Proposition \ref{prop:property_UAW} \ref{approximate_sparsity} shows that matrix coefficients $\bfA_{\lambda,\lambda'}=\langle \mcA \psi_{\lambda'},\psi_{\lambda} \rangle$ decay rapidly as the supports of the wavelets $\psi_{\lambda},\psi_{\lambda'}$ become separated or as their scales differ, motivating the following definition.  Recall that for $(j,k)\in \mcJ$, $S_{j,k}=\mathrm{conv \, hull}(\mathrm{supp}\bigl(\psi_{j,k})\bigr)$ denotes the convex hull of the support of $\psi_{j,k}$.

\begin{definition}[Matrix compression]\label{def:supp_Jtt'}  Given $J, t,t',r,n,\sigma$ and a biorthogonal wavelet pair $(\Psi,\widetilde{\Psi})$ with regularity $\gamma,\widetilde{\gamma}$ and approximation orders $d,\widetilde{d}$, define the support set
\begin{align}\label{eq:supp_Jtt'}
&\supp(J,t,t')=\supp\big(J,t,t',r,n,\sigma, \widetilde{d}\, \big) \nonumber\\
&:=\bigg\{(\lambda,\lambda')\in \Lambda_J\times \Lambda_J: \mathrm{dist}(S_{j,k},S_{j',k'})\le \tau_{jj'}, j \le \frac{t+t'-r}{\sigma-\frac{n}{2}+t'-\frac{r}{2}}J+\frac{\sigma-\frac{n}{2}-(t-\frac{r}{2})}{\sigma-\frac{n}{2}+t'-\frac{r}{2}} j',\nonumber\\
&\qquad j'\le \frac{t+t'-r}{\sigma-\frac{n}{2}+t-\frac{r}{2}}J+\frac{\sigma-\frac{n}{2}-(t'-\frac{r}{2})}{\sigma-\frac{n}{2}+t-\frac{r}{2}} j\bigg\},
\end{align}
where the threshold $\tau_{jj'}$ is chosen as
\begin{equation}\label{eq:tau}
\tau_{jj'}:=a\max \left\{2^{-\min\{j,j'\}}, 2^{(J(t+t'-r)-jt'-j't-(j+j
')\widetilde{d})/(2\widetilde{d}+r)}\right\}
\end{equation}
for some sufficiently large constant $a > 1$.

The indicator matrix associated to $\supp(J,t,t')$ is
\begin{align}\label{eq:mask_Jtt'}
\bfM_{(J,t,t')}:=\mathbf{1}_{\supp(J,t,t')}\in\{0,1\}^{\Lambda_J\times \Lambda_J},
\end{align}
where the entry $(\bfM_{(J,t,t')})_{\lambda,\lambda'}=1$ if and only if $(\lambda,\lambda')\in\supp(J,t,t')$. 

The \emph{a priori} compressed matrix $\bfA^{\varepsilon}_{\Lambda_J}\in \R^{\Lambda_J\times \Lambda_J}$ is defined by
\[
\bfA^{\varepsilon}_{\Lambda_J}:= \bfM_{(J,t,t')} \odot \bfA_{\Lambda_J},
\]
 where $\odot$ denotes the Hadamard (entrywise) product.
\end{definition}

\begin{figure}[htbp]
\centering
\includegraphics[scale=0.85]{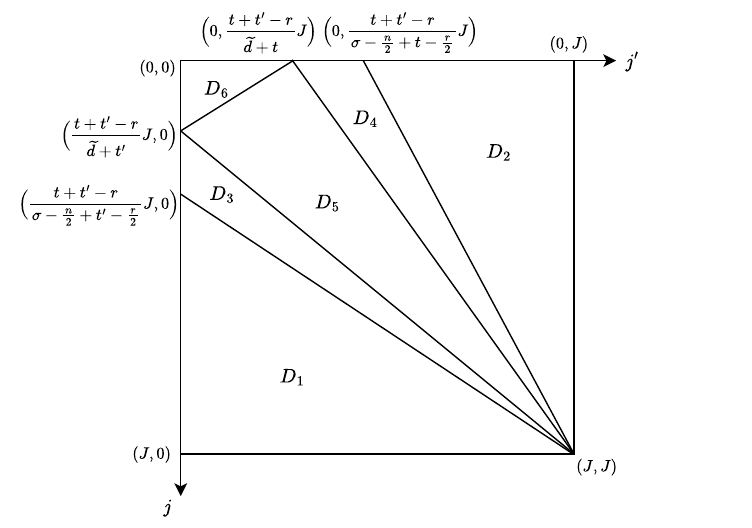}
\caption{Illustration of the $(J,t,t')$-compression of Definition~\ref{def:supp_Jtt'} in the $(j,j')$–plane. 
The index set is partitioned into regions $D_1$–$D_6$: blocks in $D_1\cup D_2$ are discarded; 
in $D_3$ and $D_4$, entries are retained only when $\mathrm{dist}(S_{j,k},S_{j',k'})\le\tau_{jj'}$ with 
$\tau_{jj'}\asymp 2^{-j'}$ and $\tau_{jj'}\asymp 2^{-j}$, respectively; in $D_5$, entries are retained only when $\mathrm{dist}(S_{j,k},S_{j',k'})\le\tau_{jj'}$ with 
$\tau_{jj'}\asymp 2^{(J(t+t'-r)-jt'-j't-(j+j')\widetilde d)/(2\widetilde d+r)}$; $D_6$ is uncompressed (all entries retained). 
}
\label{fig:figure1}
\end{figure}

Throughout, we write $\supp(J,t,t')$ to emphasize its dependence on $(J,t,t')$, suppressing the parameters $(r,n,\sigma,\widetilde{d})$ that remain fixed. 
In our later developments, matrix compression introduces a second layer of bias, which will be bounded in Proposition \ref{prop:Error-Compression}.

\begin{remark}[Discussion: matrix compression and comparison with \cite{dahmen2006compression}]
In wavelet methods for scientific computing and numerical analysis ---particularly in boundary integral equation solvers--- matrix compression techniques play a central role in achieving (near-)optimal computational complexity, and are a key ingredient in wavelet--Galerkin methods \cite{dahmen1997wavelet, cohen2000wavelet,cohen2004adaptive2, dahmen2006compression, harbrecht2006wavelet, harbrecht2024multilevel}. In the present work, motivated by operator learning rather than classical PDE solvers, we introduce a compression scheme specifically tailored to the statistical estimation setting. Below we discuss the motivation and key properties of the matrix compression in Definition~\ref{def:supp_Jtt'} and compare it with the compression scheme developed in \cite{dahmen2006compression}; see \cite[Section~7]{dahmen2006compression}.

\begin{enumerate}
\item \textbf{Key properties of matrix compression.}
Proposition~\ref{prop:Error-Compression} shows that the weighted compression error is of order
$2^{-J(t+t'-r)}$ (up to a factor $J$), matching the truncation error $2^{-J(t+t'-r)}$ induced by restricting to indices up to level $J$; see Definition~\ref{def:matrix_truncation} and Proposition~\ref{prop:Error-Truncation}. This is a standard feature of wavelet compression schemes; see \cite[Theorem~9.1]{dahmen2006compression} and \cite[Proposition~6]{harbrecht2024multilevel}.
Moreover, the total number of nonzero entries of the compressed matrix
$\bfA^{\varepsilon}_{\Lambda_J}\in \R^{\Lambda_J\times \Lambda_J}$ is of order $\mathcal{O}(2^{Jn})$, i.e.,
linear in the number of wavelet indices $|\Lambda_J|\asymp 2^{Jn}$; see the bound
$\mathrm{nnz}\big(\bfM_{(J,t,t')}\big)\lesssim 2^{Jn}$ in Proposition~\ref{prop:sparsity_count}.
This ``linear complexity'' property guarantees optimally sparse approximations and optimal computational complexity; see \cite[Theorem~11.1]{dahmen2006compression}. It is used in \cite{harbrecht2006wavelet} to obtain quadrature strategies with linear scaling, and analogous ideas appear in covariance/precision approximation in \cite{harbrecht2024multilevel}. 

\item \textbf{Choice of $\tau_{jj'}$.}
The thresholding based on $\tau_{jj'}$ exploits Proposition~\ref{prop:property_UAW} \ref{approximate_sparsity} (i), which implies that the matrix coefficients
$\bfA_{\lambda,\lambda'}=\langle \mcA \psi_{\lambda'},\psi_{\lambda} \rangle$ decay rapidly as the supports of $\psi_{\lambda}$ and $\psi_{\lambda'}$ separate. Accordingly, we keep only those entries with
$\mathrm{dist}(S_{j,k},S_{j',k'})\le \tau_{jj'}$.

A key conceptual difference from the compression strategies in
\cite{dahmen2006compression, harbrecht2024multilevel}
(e.g., \cite[Section~7]{dahmen2006compression} and \cite[Definition~1]{harbrecht2024multilevel})
is that in our learning setting $\tau_{jj'}$ depends on the regularities $t,t'$ appearing in the error metric $\|\cdot\|_{H^t\to H^{-t'}}$, whereas in \cite{dahmen2006compression, harbrecht2024multilevel} the threshold depends instead on a user-specified hyperparameter $d'\in (d,\widetilde{d}+r)$. For prescribed parameters $t$ and $t'$, our choice typically yields smaller values of $\tau_{jj'}$, leading to more aggressive entry deletion while preserving the same bias level $2^{-J(t+t'-r)}$. More precisely, since
$d'>d>\max\{t,t'\}$,
\[
2^{(J(t+t'-r)-jt'-j't-(j+j')\widetilde{d})/(2\widetilde{d}+r)}
\le
2^{(2J(d'-r/2)-(j+j')(d'+\widetilde{d}))/(2\widetilde{d}+r)},
\]
and when $j,j'<J$ the inequality is strict; cf.\ $\tau_{jj'}$ in \cite[(3.20)]{harbrecht2024multilevel}.

From a statistical perspective, since our estimator (defined later) targets the sparse compressed matrix $\bfA^{\varepsilon}_{\Lambda_J}$ and is supported on $\supp(J,t,t')$, its variance analysis depends crucially on the sparsity pattern induced by the compression scheme. Our learning-oriented compression is advantageous: within each block, it retains fewer coefficients (equivalently, a smaller support), which leads to a smaller variance bound. We emphasize that this refinement is essential for variance analysis in Proposition~\ref{prop:Error-Var}, because under the matrix operator norm the variance aggregates across scales in a way that is not captured by the global number of nonzero entries alone. In particular, both our compression and the matrix compression scheme in \cite{dahmen2006compression} guarantee a global nnz of optimal order $\mathcal{O}(2^{Jn})$, i.e., linear in $|\Lambda_J|$. However, obtaining a sharp variance bound requires near-optimal sparsity \emph{within each block}; this is precisely where our learning-oriented compression provides an advantage. 

\item \textbf{Slope conditions.}
Beyond entrywise thresholding, we retain only those $(j,j')$ blocks satisfying the two slope conditions in \eqref{eq:supp_Jtt'}, which exploit Proposition~\ref{prop:property_UAW} \ref{approximate_sparsity} (ii). By Assumption~\ref{assumption:wavelets} (iii), we have $\sigma-n/2>\max\{t,t'\}-r/2$ and $\widetilde{d}>\sigma-n/2-r/2$, which ensures that the slope constraints define a nontrivial admissible region in the $(j,j')$ plane; see Figure~\ref{fig:figure1}. As will become clear in the proof of Proposition~\ref{prop:Error-Compression}, these slope conditions are chosen sharply so that the discarded blocks in $D_1\cup D_2$ contribute at most $2^{-J(t+t'-r)}$ to the bias.  In this sense, the slope conditions are deliberately aggressive: they discard as many blocks as possible while still ensuring that the resulting bias remains of order at most $2^{-J(t+t'-r)}$. At the same time, discarding more blocks reduces the number of retained coefficients and hence helps control the variance of the estimator. 

 In summary, both the choice of the thresholding levels $\tau_{jj'}$ and the design of the slope conditions are guided by the same principle: keep as few entries as possible within each $(j,j')$ block while ensuring that the induced bias is at most of order $2^{-J(t+t'-r)}$ under the weighted matrix operator norm $\|\bfD^{-t'}(\cdot)\bfD^{-t}\|$ in \eqref{eq:matrix_norm}, matching the truncation bias introduced in Definition \ref{def:matrix_truncation}. This is why $\tau_{jj'}$ and the slope conditions depend on the prescribed $(t,t')$ in the learning error metric. By contrast, the matrix compression scheme in \cite{dahmen2006compression} is designed to accommodate a range of weight parameters simultaneously (up to the approximation order of the MRAs); see \cite[Theorem~9.1]{dahmen2006compression} and \cite[Proposition~6]{harbrecht2024multilevel}.   

\item \textbf{Asymmetry of $\supp(J,t,t')$.}
Our compression is inherently asymmetric when $t\neq t'$. In this case, $\tau_{jj'}\neq \tau_{j'j}$, and the slope conditions are likewise asymmetric, as reflected in the partition in Figure~\ref{fig:figure1} (illustrated for $t<t'$). This asymmetry stems from the fact that our compression is tailored to the asymmetric error metric $\|\cdot\|_{H^{t}\to H^{-t'}}$ in our learning setting and, equivalently, to the weighted matrix operator norm $\|\bfD^{-t'}(\cdot)\bfD^{-t}\|$. In contrast, the thresholding parameters and matrix compression schemes in \cite{dahmen2006compression,harbrecht2024multilevel} are symmetric.

\item \textbf{On the ``second compression'' in \cite{schneider2013multiskalen, dahmen2006compression}.}
Our compression uses only Proposition~\ref{prop:property_UAW} \ref{approximate_sparsity}. In addition, \cite{schneider2013multiskalen,dahmen2006compression} derives sharper estimates for entries corresponding to wavelets with overlapping supports but widely separated scales, and then applies a second thresholding step. Concretely, for a function $f$ on $\mcM$, let
$\mathrm{sing\,supp}(f):=\{x\in\mcM:\ f \text{ is not smooth at }x\}$ and let
$S'_{j,k}:=\mathrm{sing\,supp}(\psi_{j,k})\subset \mcM$ be the singular support of $\psi_{j,k}$.
For $(j,k),(j',k')\in \mcJ$ such that $j'>j$ and $\mathrm{dist}(S'_{j,k},S_{j',k'})\gtrsim 2^{-j'}$,
\cite[Theorem~6.3]{dahmen2006compression} proves
\[
|\bfA_{\lambda,\lambda'}|\lesssim
2^{jn/2}2^{-j'(\widetilde{d}+n/2)}\,
\mathrm{dist}(S'_{j,k},S_{j',k'})^{-(r+\widetilde{d})}.
\]
This motivates a second thresholding step that retains only entries satisfying
\[
\mathrm{dist}(S'_{j,k},S_{j',k'})\le \tau'_{jj'},
\]
where $\tau'_{jj'}$ is a parameter distinct from $\tau_{jj'}$; see \cite[Definition~1]{harbrecht2024multilevel} and \cite[Section~7]{dahmen2006compression}. We do not use this refinement here: in our setting it does not appear to improve the variance bound significantly, while discarding the regions $D_1\cup D_2$ yields a simpler analysis and already provides strong variance estimates (at least for fixed $t,t'$). \qedhere
\end{enumerate}
\end{remark}

\subsection{Second Stage: Estimation}\label{ssec:estimation}

In this subsection, we construct an estimator $\widehat{\bfA} \in \R^{\Lambda_{J} \times \Lambda_{J}}$ for the sparse compressed matrix $\bfA_{\Lambda_J}^{\varepsilon} \in \R^{\Lambda_{J} \times \Lambda_{J}}.$
The sparse support $\supp(J,t,t')$ will determine the structure of our regression-based estimator: the support specifies which matrix entries are actively regressed and which entries are set to zero and treated as bias.

We first observe that \eqref{eq:discrete_model} can be written columnwisely. For any index $\lambda'=(j',k')\in\mcJ$, the $\lambda'$-th column of the matrix equation \eqref{eq:discrete_model} reads
\begin{align}\label{eq:model_column}
\bfF_{\cdot,\lambda'}=\bfU \bfA_{\cdot,\lambda'}+\bfW_{\cdot,\lambda'},
\end{align}
where $\bfF_{\cdot,\lambda'}\in\R^{N}, \bfA_{\cdot,\lambda'}\in \R^{\mcJ},$ and $\bfW_{\cdot,\lambda'}\in \R^{N}$ denote the $\lambda'$-th columns of $\bfF,\bfA$, and $\bfW$, respectively. For each $\lambda'$, this yields an \emph{infinite-dimensional linear regression problem} with response vector $\bfF_{\cdot,\lambda'}$, design matrix $\bfU$, coefficient vector $\bfA_{\cdot,\lambda'}$, and noise vector $\bfW_{\cdot,\lambda'}$. 
Based on \eqref{eq:model_column}, our strategy is to estimate the compressed matrix $\bfA_{\Lambda_J}^{\varepsilon}$ columnwisely via regression, while carefully exploiting the sparse support structure of $\bfA_{\Lambda_J}^{\varepsilon}$.

Note that the $\lambda'$-th column of $\bfA_{\Lambda_J}^{\varepsilon}=\bfM_{(J,t,t')} \odot \bfA_{\Lambda_J}$ is supported on the finite set 
\[
\left\{\lambda\in \mcJ: (\lambda,\lambda')\in \supp(J,t,t')\right\}.
\]
 However, directly regressing on this index set can induce a non-negligible \emph{omitted-variable bias}, since the neglected components ---though omitted from the approximation--- still interact with the retained ones through the data; see Remark \ref{remark:structured_regression} for a detailed explanation. This bias is data-dependent and can be substantially larger than the deterministic truncation and compression errors incurred in the first stage. To address this issue, 
a key idea of our regression estimator is as follows:  for each column $\lambda'$, we choose a slightly larger index set containing the support of $(\bfA_{\Lambda_J}^{\varepsilon})_{\cdot,\lambda'}$ as \emph{regression support}, perform regression over this enlarged set, and then restrict the resulting estimator back to $\supp(J,t,t')$.  This \emph{nested-support regression} strategy allows us to efficiently control the omitted-variable bias, keeping it at most of the same order as the truncation/compression error.

More precisely, let $\widetilde{J},\widetilde{t},$ and $\widetilde{t}'$ be three tuning parameters. Lemma \ref{lemma:supp_monotonicity} in Appendix \ref{ssec:Appendixsupportmonotonicity} ensures that the inclusion
\begin{equation}\label{eq:inclusion}
   \supp(J,t,t')\subset \supp(\widetilde{J},\widetilde{t},\widetilde{t}') 
\end{equation}
 holds provided that $\widetilde{J}\ge J,\widetilde{t}\ge t$, and $\widetilde{t}'\ge t'.$ The values of the hyperparameters $\widetilde{J},\widetilde{t},\widetilde{t}'$ defining the enlarged regression support $\supp(\widetilde{J},\widetilde{t},\widetilde{t}')\subset \Lambda_{\widetilde{J}}\times \Lambda_{\widetilde{J}}$ will be chosen so that the omitted-variable bias is of the same order as the truncation and compression errors.   
 
Now, for each $\lambda'\in\Lambda_{J}$, define
\[
\Omega_{\lambda'}:=\left\{\lambda\in \mcJ: (\lambda,\lambda')\in \supp(\widetilde{J},\widetilde{t},\widetilde{t}') \right\}\subset \Lambda_{\widetilde{J}},\qquad \Omega_{\lambda'}^c:=\mcJ\setminus \Omega_{\lambda'}
\]
and decompose the design matrix and coefficient vector accordingly:
\begin{align}\label{eq:partition}
\bfU=[\bfU_{\cdot,\Omega_{\lambda'}}, \bfU_{\cdot,\Omega_{\lambda'}^c}],\quad \bfA_{\cdot,\lambda'}=\begin{bmatrix} \bfA_{\Omega_{\lambda'},\lambda'}  \\ \bfA_{\Omega_{\lambda'}^c,\lambda'} \end{bmatrix}.
\end{align}
Here we use the convention that, for any subset $\Lambda\subset \mcJ$, $\bfU_{\cdot,\Lambda}\in R^{N\times \Lambda}$ denotes the submatrix of $\bfU$ formed by the columns indexed by $\Lambda$, and for any vector $\bfv\in \R^{\mcJ}$, $\bfv_{\Lambda}\in \R^{\Lambda}$  denotes the subvector of $\bfv$ with components indexed by $\Lambda$. The inclusion $\supp(J,t,t')\subset \supp(\widetilde{J},\widetilde{t},\widetilde{t}')$ in \eqref{eq:inclusion} implies that, for each column $\lambda'$, the vector $\bfA_{\Omega_{\lambda'},\lambda'}$  (strictly)  contains all nonzero entries of $(\bfA_{\Lambda_J}^{\varepsilon})_{\cdot,\lambda'}$ that we aim to estimate.

Substituting \eqref{eq:partition} into \eqref{eq:model_column}, we obtain
\begin{align}\label{eq:model_column_partitioned}
\bfF_{\cdot,\lambda'}=\bfU \bfA_{\cdot,\lambda'}+\bfW_{\cdot,\lambda'}=\bfU_{\cdot,\Omega_{\lambda'}}\bfA_{\Omega_{\lambda'},\lambda'}+\bfU_{\cdot,\Omega_{\lambda'}^c} \bfA_{\Omega_{\lambda'}^c,\lambda'}+\bfW_{\cdot,\lambda'}.
\end{align}

 Assuming that $\bfU_{\cdot,\Omega_{\lambda'}}^{\top} \bfU_{\cdot,\Omega_{\lambda'}}$ is invertible,  we estimate $\bfA_{\Omega_{\lambda'},\lambda'}$ using the ordinary least squares estimator \begin{align}\label{eq:estimator_column}(\bfU_{\cdot,\Omega_{\lambda'}}^{\top} \bfU_{\cdot,\Omega_{\lambda'}})^{-1} \bfU_{\cdot,\Omega_{\lambda'}}^{\top} \bfF_{\cdot,\lambda'},
 \end{align}
ignoring the contribution from $\bfA_{\Omega_{\lambda'}^c,\lambda'}$. We then concatenate the columnwise regression estimates, restrict the resulting matrix to $\supp(J,t,t')$, and exploit the symmetry of $\bfA$ to assemble the full estimator $\widehat{\bfA}$ of the compressed matrix $\bfA^{\varepsilon}_{\Lambda_J}.$  Using $\widehat{\bfA}$, we define the final estimator $\widehat{\mcA}$ of the unknown operator $\mcA$ via \eqref{eq:estimator_pre1}. The complete definitions of $\widehat{\bfA}$ and $\widehat{\mcA}$ are given below.

\begin{tcolorbox}[scale =1,
  enhanced,
  colback=white,
  frame hidden, boxrule=0pt,            
  borderline north={0.8pt}{0pt}{black},   
  borderline north={0.8pt}{17pt}{black!60!white},
  borderline south={0.8pt}{0pt}{black!60!white}, 
  title=\textbf{Algorithm: Construction of the Estimator $\widehat{\mcA}$},
  label=alg:construction,
  colbacktitle=white,
  coltitle=black,              
  titlerule=0pt,              
  left=1mm,right=1mm,top=1mm,bottom=1mm
]

\noindent\textbf{Inputs:}
\begin{itemize}[leftmargin=2em, itemsep=0pt]
  \item Maximum wavelet scales $J$ and $\widetilde{J}$; regularity parameters $t,t',\widetilde{t},\widetilde{t}'$.
  \item Compression support $\mathrm{supp}(J,t,t')$, regression support $\mathrm{supp}(\widetilde{J},\widetilde{t},\widetilde{t}')$.
  \item Data matrices $\bfU_{\cdot,\Lambda_{\widetilde{J}}} \in \mathbb{R}^{N \times \Lambda_{\widetilde{J}}}$ and $\bfF_{\cdot,\Lambda_{\widetilde{J}}} \in \mathbb{R}^{N \times \Lambda_{\widetilde{J}}}$.
\end{itemize}

\noindent\textbf{Output:} Estimator $\widehat{\mcA}: H^t(\mcM) \to H^{-t'}(\mcM)$.

\medskip
\noindent\textbf{Step 1: Columnwise regression.}
\vspace{0.3em}

Let $\Omega_{\lambda'}:=\left\{\lambda\in \mcJ: (\lambda,\lambda')\in \supp(\widetilde{J},\widetilde{t},\widetilde{t}') \right\}$. Compute the preliminary estimator $\widehat{\bfA}^{(1)} \in \mathbb{R}^{\Lambda_{\widetilde{J}}\,\times \Lambda_{J}}$ by concatenating columnwise regression estimates:
\begin{align}\label{eq:estimator_column_box}
(\widehat{\bfA}^{(1)})_{\cdot,\lambda'}
:=
\begin{bmatrix}
(\bfU_{\cdot,\Omega_{\lambda'}}^{\top} \bfU_{\cdot,\Omega_{\lambda'}})^{-1}
\bfU_{\cdot,\Omega_{\lambda'}}^{\top} \bfF_{\cdot,\lambda'} \\
\boldsymbol{0}
\end{bmatrix}
\in \mathbb{R}^{\Lambda_{\widetilde{J}}},
\qquad
\lambda' \in \Lambda_{J},
\end{align}
where
$(\bfU_{\cdot,\Omega_{\lambda'}}^{\top} \bfU_{\cdot,\Omega_{\lambda'}})^{-1}
\bfU_{\cdot,\Omega_{\lambda'}}^{\top} \bfF_{\cdot,\lambda'} \in \mathbb{R}^{\Omega_{\lambda'}}$
and
$\boldsymbol{0} \in \mathbb{R}^{\Lambda_{\widetilde{J}} \,\setminus\, \Omega_{\lambda'}}$.

\vspace{1em}
\noindent\textbf{Step 2: Restrict to the target sparsity pattern and use symmetry.}
\vspace{0.3em}

Define the estimator $\widehat{\bfA} \in \mathbb{R}^{\Lambda_J \times \Lambda_J}$ of the compressed matrix $\bfA^{\varepsilon}_{\Lambda_J}\in \R^{\Lambda_J\times \Lambda_J}$ entrywise by
\begin{align}\label{eq:estimator_matrix_box}
\widehat{\bfA}_{\lambda,\lambda'}
:=
\begin{cases}
 (\widehat{\bfA}^{(1)})_{\lambda,\lambda'}  & \text{if } (\lambda,\lambda')\in \supp(J,t,t') \text{ and }  j \le j', \\[2mm]
 (\widehat{\bfA}^{(1)})_{\lambda',\lambda} & \text{if } (\lambda,\lambda')\in \supp(J,t,t') \text{ and } j > j',\\[2mm]
 0 & \text{if } (\lambda,\lambda')\notin  \supp(J,t,t').
\end{cases}
\end{align}

\noindent\textbf{Step 3: Operator construction.}
\vspace{0.3em}

Define the estimator $\widehat{\mcA}$ as
\begin{align}\label{eq:estimator_operator_box}
\widehat{\mcA}
:= \sum_{(\lambda,\lambda') \in \supp(J,t,t')}
\widehat{\bfA}_{\lambda,\lambda'}\,
\widetilde{\psi}_{\lambda} \otimes \widetilde{\psi}_{\lambda'}.
\end{align}
\end{tcolorbox}

\smallskip

\begin{remark}[Interpretation of the construction of $\widehat{\mcA}$]
Since the index set $\supp(J,t,t')$ is not symmetric, the compressed matrix
$\bfA^\varepsilon_{\Lambda_J}$ (and hence our estimator) is not symmetric.
In \textbf{Step~1}, we construct a preliminary estimator $\widehat{\bfA}^{(1)}$
via column-wise regressions, without using the symmetry of the target matrix $\bfA$.
We now explain the motivation for the symmetrization procedure in \textbf{Step~2},
which uses the lower-variance side of $\widehat{\bfA}^{(1)}$ to estimate its symmetric counterpart, thereby reducing the overall variance.

As will become clear from the variance analysis in Proposition~\ref{prop:Error-Var}
(see \eqref{eq:var_aux3}), the standard deviation of an entry in the $(j,j')$-block of
$\widehat{\bfA}^{(1)}$ is of order $2^{j r_1-j' r_2}/\sqrt{N}$, up to logarithmic factors.
Under Assumption~\ref{assumption:operator_data_noise}, we have
$r_1>n/2+\max\{0,r\}$ and $r_2>n/2$, and therefore, whenever $j<j'$,
\[
2^{j r_1-j' r_2} < 2^{j' r_1-j r_2}.
\]
Consequently, for $j<j'$ (i.e., for blocks \emph{above} the diagonal $j=j'$ in the $(j,j')$-plane in Figure \ref{fig:figure1}),
the entries of $\widehat{\bfA}^{(1)}$ in the $(j,j')$-block have smaller variance than
their symmetric counterparts in the transposed $(j',j)$-block (which lies \emph{below} the diagonal).

This variance imbalance can be understood from the model
\[
\bfF=\bfU\bfA+\bfW.
\]
Roughly speaking, due to the regularity assumptions $r_1,r_2>n/2$, the wavelet coefficients
of the input matrix $\bfU$ and the noise matrix $\bfW$ decay with the scale index. In particular,
when estimating the $\lambda'$-th column of $\bfA$ at a finer scale (large $j'$), the effective noise level is smaller.
Moreover, the entries of $\bfU$ at coarser scales are typically larger in magnitude, so
the corresponding regression has a stronger signal. Together, these effects yield smaller variance
for estimating $\bfA$ above the diagonal than for estimating the entries below the diagonal.

Motivated by this observation, in \textbf{Step~2} we exploit the symmetry of the target matrix $\bfA$ to transfer the more accurate
estimates (the lower-variance side, above the diagonal) to the opposite side (the higher-variance side, below the diagonal). Since our goal is to estimate the compressed matrix $\bfA^\varepsilon_{\Lambda_J}$ supported on $\supp(J,t,t')$, when $t\le t'$,
Lemma~\ref{lemma:upper_contain_lower} below guarantees that the reflection of the below-diagonal support
is contained in the above-diagonal support. This inclusion allows us to define the final estimator
$\widehat{\bfA}$ in \eqref{eq:estimator_matrix_box} by symmetrizing $\widehat{\bfA}^{(1)}$ using the
lower-variance entries, thereby reducing the overall variance compared with $\widehat{\bfA}^{(1)}$.
\end{remark}

\smallskip

\begin{lemma}\label{lemma:upper_contain_lower}
Let $\bfM_{(J,t,t')}$ be the indicator matrix introduced in  Definition~\ref{def:supp_Jtt'}. If $t \le t'$, then for any pair of indices $(\lambda,\lambda')=((j,k),(j',k'))$ with $j > j'$, 
we have
\[
(\bfM_{(J,t,t')})_{\lambda,\lambda'} = 1 
\quad\Longrightarrow\quad
(\bfM_{(J,t,t')})_{\lambda',\lambda} = 1.
\]
\end{lemma}

\smallskip

\begin{remark}[Effective data and computational cost]\label{remark:cost1}
Observe that, in the algorithm constructing the estimator $\widehat{\bfA}$, we only use finite-dimensional data matrices \[
\bfU_{\cdot,\Lambda_{\widetilde{J}}} \in \mathbb{R}^{N \times \Lambda_{\widetilde{J}}},\qquad \bfF_{\cdot,\Lambda_{\widetilde{J}}} \in \mathbb{R}^{N \times \Lambda_{\widetilde{J}}},
\]
rather than the full coefficient matrices $\bfU,\bfF$. In our algorithm, we assume oracle access to the wavelet coefficients of the input–output data,
i.e., the quantities $(\langle u_i,\widetilde{\psi}_{\lambda}\rangle,\langle f_i,\psi_{\lambda}\rangle)$.
This does not reduce the generality of the learning problem, since the (bi)orthogonal wavelet
transform is an invertible linear map. For algorithmic and implementation details of
wavelet–Galerkin methods, we refer the reader to \cite{harbrecht2006wavelet,harbrecht2021fast}. In the following discussion on the computational cost of our learning algorithm, we ignore the cost of constructing the wavelets and computing these coefficients, and assume that $\bfU,\bfF$ are given.

We now discuss the computational cost of the algorithm. For a fixed column index $\lambda'$, we
compute
\[
(\bfU_{\cdot,\Omega_{\lambda'}}^{\top}\bfU_{\cdot,\Omega_{\lambda'}})^{-1}\,
\bfU_{\cdot,\Omega_{\lambda'}}^{\top}\bfF_{\cdot,\lambda'},
\]
where
\[
\Omega_{\lambda'}
:=\bigl\{\lambda\in\mcJ:\;(\lambda,\lambda')\in\supp(\widetilde{J},\widetilde{t},\widetilde{t}')\bigr\}.
\]
Forming the Gram matrix $\bfU_{\cdot,\Omega_{\lambda'}}^{\top}\bfU_{\cdot,\Omega_{\lambda'}}$
and the cross term $\bfU_{\cdot,\Omega_{\lambda'}}^{\top}\bfF_{\cdot,\lambda'}$ costs
$\mathcal{O} \left(N|\Omega_{\lambda'}|^{2}\right)$, and solving the resulting linear system
(e.g.\ via a Cholesky factorization) costs $\mathcal{O} \left(|\Omega_{\lambda'}|^{3}\right)$.
Hence the per-column cost is
\[
\mathcal{O}\left(N|\Omega_{\lambda'}|^{2}+|\Omega_{\lambda'}|^{3}\right).
\]
Since we only perform column-wise regressions for indices $\lambda'\in \Lambda_J$ to construct $\widehat{\bfA}^{(1)}$, summing over all such columns $\lambda'\in\Lambda_J$ yields a total computational cost of order
\begin{align}\label{eq:cost}
\mathcal{O}\bigg(
N\bigg(\sum_{\lambda'\in\Lambda_J}|\Omega_{\lambda'}|^{2}\bigg)
+\sum_{\lambda'\in\Lambda_J}|\Omega_{\lambda'}|^{3}
\bigg),
\end{align}
which depends crucially on the sparsity pattern of the regression support $\supp(\widetilde{J},\widetilde{t},\widetilde{t}')$. In Remark~\ref{remark:cost2}, we provide a more explicit bound under a specific choice of $\widetilde{J},\widetilde{t},\widetilde{t}'$.
\end{remark}

\smallskip

\begin{remark}[Structured infinite-dimensional linear regression]\label{remark:structured_regression} In this remark, we interpret the column-wise model \eqref{eq:model_column_partitioned} and explain the structure behind the error analysis for the estimator \eqref{eq:estimator_column}. To streamline the discussion, we rewrite \eqref{eq:model_column} as an abstract infinite-dimensional linear regression problem,
\begin{align*}
y=\bfX\theta+\xi,
\end{align*}
where $y\in\R^N$ is the response vector, $ \bfX\in \R^{N\times \mcJ}$ is a design matrix indexed by an infinite set $\mcJ$, $\theta\in \R^{\mcJ}$ is an unknown coefficient vector, and $\xi\in \R^{N}$ is a noise vector. The goal is to recover $\theta$ given $(\bfX, y)$.

Suppose that the infinite-dimensional vector $\theta$ exhibits structure through \emph{a priori} information, suggesting that its dominant coordinates are contained in a finite index set $\Omega\subset \mcJ$; denote the complement by $\Omega^c:=\mcJ\backslash \Omega$. Decomposing both the design matrix and the parameter vector according to this partition yields
\[
y=\bfX\theta+\xi=\bfX_{\cdot,\Omega}\theta_{\Omega}+\bfX_{\cdot,\Omega^c} \theta_{\Omega^c}+\xi.
\]
If $\bfX_{\cdot,\Omega}^{\top} \bfX_{\cdot,\Omega}$ is invertible, we define the estimator that regresses on $\Omega$ and sets the remaining coordinates to zero,
\begin{align*}
\widehat{\theta} \,:=\,
\begin{pmatrix}
\widehat{\theta}_{\Omega} \\
\boldsymbol{0}
\end{pmatrix},\qquad \widehat{\theta}_{\Omega}:= (\bfX_{\cdot,\Omega}^{\top} \bfX_{\cdot,\Omega})^{-1} \bfX_{\cdot,\Omega}^{\top} y.
\end{align*}
A direct calculation gives
\begin{align*}
(\bfX_{\cdot,\Omega}^{\top} \bfX_{\cdot,\Omega})^{-1} \bfX_{\cdot,\Omega}^{\top} y-\theta_{\Omega}=(\bfX_{\cdot,\Omega}^{\top} \bfX_{\cdot,\Omega})^{-1} \bfX_{\cdot,\Omega}^{\top}\bfX_{\cdot,\Omega^c} \theta_{\Omega^c}+(\bfX_{\cdot,\Omega}^{\top} \bfX_{\cdot,\Omega})^{-1} \bfX_{\cdot,\Omega}^{\top}\xi,
\end{align*}
and therefore
\begin{align}\label{eq:error_abstract}
\widehat{\theta}-\theta=\begin{pmatrix}
(\bfX_{\cdot,\Omega}^{\top} \bfX_{\cdot,\Omega})^{-1} \bfX_{\cdot,\Omega}^{\top}\bfX_{\cdot,\Omega^c} \theta_{\Omega^c}+(\bfX_{\cdot,\Omega}^{\top} \bfX_{\cdot,\Omega})^{-1} \bfX_{\cdot,\Omega}^{\top}\xi \\
-\theta_{\Omega^c}
\end{pmatrix}.
\end{align}

On the index set $\Omega^c$, the error is the \emph{bias} term $-\theta_{\Omega^c}$, which is deterministic and independent of the design $\bfX$ and the noise $\varepsilon$. Since $\Omega$ is chosen to contain the significant entries of $\theta$, this term is expected to be small.

On the index set $\Omega$, the error decomposes into two terms. The first term 
\[(\bfX_{\cdot,\Omega}^{\top} \bfX_{\cdot,\Omega})^{-1} \bfX_{\cdot,\Omega}^{\top}\bfX_{\cdot,\Omega^c} \theta_{\Omega^c}
\]
is the \emph{omitted-variable bias} ($\OVB$) term that depends on the covariance structure of the design. In classical finite-dimensional regression, $\OVB$ arises when relevant covariates are omitted and their effect is partially attributed to included covariates; see, e.g., \cite[Chapter 18]{barreto2005introductory} and \cite[Chapter 3-3]
{wooldridge2016introductory}. Here, $\OVB$ appears from a slightly different perspective: when $\theta$ is infinite-dimensional, any finite-dimensional regression necessarily ignores a (presumably negligible) tail $\theta_{\Omega^c}$, and this tail can bias the estimate on $\Omega$ when the columns of $\bfX$ are correlated.

For random design with i.i.d.\ rows, one can make this precise. Let
\[
\bfC:=\E[\bfX^{\top}\bfX]/N
\]
denote the population covariance of each row of the design. Under standard regularity conditions, one typically has $\bfX_{\cdot,\Omega}^{\top}\bfX_{\cdot,\Omega}/N\to \bfC_{\Omega,\Omega}$ and $\bfX_{\cdot,\Omega}^{\top}\bfX_{\cdot,\Omega^c}/N\to \bfC_{\Omega,\Omega^c}$ as $N\to\infty$, and hence the $\OVB$ term converges to \[
\bfC_{\Omega,\Omega}^{-1}\bfC_{\Omega,\Omega^c}\,\theta_{\Omega^c}.
\]
In general, this limit is not zero without additional structure on $\theta_{\Omega^c}$ and/or $\bfC$ (e.g.\ approximate orthogonality across the partition), so it contributes a genuine bias. Our nested-support strategy is designed to balance the two bias contributions: the deterministic bias $-\theta_{\Omega^c}$ and the $\OVB$ induced by $\theta_{\Omega^c}$ through design correlations; see Proposition \ref{prop:Error-OVB} in Subsection \ref{ssec:proof_main1}.

The second term on the index set $\Omega$ in \eqref{eq:error_abstract},
\[
(\bfX_{\cdot,\Omega}^{\top} \bfX_{\cdot,\Omega})^{-1} \bfX_{\cdot,\Omega}^{\top}\xi,
\]
is the usual \emph{variance} term. In the standard setting where $\xi\sim \mathcal{N}(0,\sigma_{\xi}^2 I_N)$ is independent of $\bfX$, one has $\E\left[(\bfX_{\cdot,\Omega}^{\top} \bfX_{\cdot,\Omega})^{-1} \bfX_{\cdot,\Omega}^{\top}\xi \right]=0$, and, conditional on $\bfX$,
\[
(\bfX_{\cdot,\Omega}^{\top} \bfX_{\cdot,\Omega})^{-1} \bfX_{\cdot,\Omega}^{\top}\xi
\sim \mathcal{N} \left(0, \sigma_{\xi}^2(\bfX_{\cdot,\Omega}^{\top} \bfX_{\cdot,\Omega})^{-1}\right).
\]
In particular, under standard conditions ensuring $\bfX_{\cdot,\Omega}^{\top}\bfX_{\cdot,\Omega}/N\to \bfC_{\Omega,\Omega}$, the typical size of this term decays at rate $N^{-1/2}$.

In our analysis of the matrix estimator \eqref{eq:estimator_matrix_box} in Subsection~\ref{ssec:proof_main1}, the total error admits an analogous decomposition for each column $\lambda'$. We note, however, that the regression support $\supp(\widetilde{J},\widetilde{t},\widetilde{t}')$ strictly contains the compression support $\supp(J,t,t')$: this enlargement-restriction step in the nested-support strategy mitigates the omitted-variable bias, but introduces additional technical difficulty in the analysis. Moreover, since both the regression support set $\Omega_{\lambda'}$ and the noise level depend on $\lambda'$, we first derive column-wise error bounds that exploit the associated sparsity pattern and noise scaling, and then aggregate these bounds across scales to control the overall matrix error under the weighted operator norm \eqref{eq:matrix_norm}.
\end{remark}

\section{Convergence Rates for Operator Learning}\label{sec:convergencerate}
\subsection{First Main Result}
We are now ready to state and discuss our first main result:
\begin{theorem}\label{thm:main1}    
Suppose that Assumption \ref{assumption:operator_data_noise} and Assumption \ref{assumption:wavelets} hold. Let $t'\ge t>r/2$. Consider the estimator $\widehat{\mcA}=\widehat{\mcA}(\{u_i,f_i\}_{i=1}^{N})$ defined in Section \ref{sec:define_estimator} with parameters
\begin{align}\label{eq:def_parameter}
\widetilde{J}:=\left\lceil \frac{t+t'-r+\varepsilon_1}{\min\{t',r_1\}+t-r}\, J\right\rceil, \quad \widetilde{t}:=t',\quad \widetilde{t}':=\max\{t',r_1\},\quad J:=\left\lceil \frac{\log_2 N}{(2+\rho)(t+t'-r)} \right\rceil,
\end{align}
where $\varepsilon_1:= n(t+t'-r)/(\sigma-n/2+t-r/2)$ and
\begin{align}\label{eq:rho}
\rho:=2\max\left\{ \frac{-t-r_2+n/2}{\sigma-n/2+t-r/2},\, \frac{-t'-r_2+n/2}{\sigma-n/2+t'-r/2},\,\frac{-t-t'+r_1-r_2+n}{t+t'+2\widetilde{d}} ,\,\frac{-t-t'+r_1-r_2}{t+t'-r},0\right\}.
\end{align}
For any $\delta\in(0,1)$, if $N\gtrsim \log (1/\delta)$, then with probability at least $1-\delta$,
\[
\|\widehat{\mcA}-\mcA\|_{H^{t}\to H^{-t'}}
\;\lesssim\;
N^{-\frac{1}{2+\rho}}\,
\sqrt{\log\!\Big(\frac{N}{\delta}\Big)}\,
\log N.
\]
\end{theorem}

Theorem~\ref{thm:main1} establishes a high-probability error bound of order $N^{-\frac{1}{2+\rho}}$ (up to logarithmic factors) for learning elliptic pseudo-differential operator from noisy data. It holds under general assumptions on the inputs and noise, together with an appropriate choice of wavelets. The implicit constant depends only on the parameters listed in Table \ref{tab:key_parameters}, which we treat as fixed throughout.

The rest of this subsection is organized as follows. Remark \ref{rem:ttprime} discusses the assumption $t\le t'$ and explains why it entails no loss of generality. Remark \ref{remark:rate_discussion} provides a detailed discussion of the exponent $\rho$ in \eqref{eq:rho}, which governs the convergence rate; in particular, it highlights several qualitative features of the rate and shows that it is nearly optimal. Remark \ref{rem:sparsity} shows that the estimator $\widehat{\mcA}$ has only $\mathcal{O}(2^{Jn})=\mathcal{O}\big(N^{\frac{n}{(2+\rho)(t+t'-r)}}\big)$ nonzero coefficients, enabling fast numerical algorithms for downstream tasks. Remark \ref{remark:cost2} continues the computational-cost discussion in Remark \ref{remark:cost1} and shows that, under the parameter choice \eqref{eq:def_parameter}, the cost of constructing $\widehat{\mcA}$ is nearly $\mathcal{O}(N 2^{Jn})= \mathcal{O}\big(N^{1+\frac{n}{(2+\rho)(t+t'-r)}}\big)$, which is linear in the sample size $N$ and linear in the learned-matrix dimension $2^{Jn}$. Remark~\ref{rem:noisefree} considers the noiseless setting and shows that our estimator achieves a super-algebraic convergence rate. Finally, Example \ref{rem:example} illustrates our results in the concrete setting of learning the Green's function of elliptic PDEs, and compares both the convergence rate and computational cost with those in \cite{schafer2024sparse}.

\begin{table}[htbp]
\centering
\small
\renewcommand{\arraystretch}{1.5}
\setlength{\tabcolsep}{15pt}
\begin{tabular}{|c|c|}
\hline
\textbf{Parameter} & \textbf{Description} \\
\hline
$N$ & Sample size: number of input--output data pairs $\{(u_i,f_i)\}_{i=1}^N$ \\
\hline
$n$ & Dimension of the physical domain (manifold $\mcM$) \\
\hline
$r$ & Order of the pseudo-differential operator $\mcA$ \\
\hline
$t,t'$ & Input/output Sobolev index in the error norm $\|\cdot\|_{H^{t}\to H^{-t'}}$ \\
\hline
$r_1,r_2$ & \begin{tabular}[c]{@{}l@{}}
Regularity of the input function/noise in Assumption~\ref{assumption:operator_data_noise}: \\
\qquad\qquad \quad $u\in H^{(r_1-n/2)-}, w\in H^{(r_2-n/2)-}$
\end{tabular} \\
\hline
$\gamma,\widetilde{\gamma}$ & Regularity of the primal/dual wavelets \\
\hline
$d,\widetilde d$ & Approximation order of the primal/dual wavelets \\
\hline
$\sigma$ & \begin{tabular}[c]{@{}l@{}}
Exponent controlling scale-separation decay of $\bfA$ \\
\quad (Proposition~\ref{prop:property_UAW} \ref{approximate_sparsity} (ii); Assumption~\ref{assumption:wavelets} (iii))
\end{tabular} \\
\hline
\end{tabular}
\caption{Summary of key parameters.}
\label{tab:key_parameters}
\end{table}

\begin{remark}[On the assumption $t\le t'$]\label{rem:ttprime}
In Theorem~\ref{thm:main1}, we assumed that $t\le t'$ only to simplify the presentation, which entails no loss of generality. Indeed, when $t\ge t'$, we may take the adjoint $(\widehat{\mcA})^{*}$ of the estimator $\widehat{\mcA}$ defined in~\eqref{eq:estimator_operator_box} as our estimator. Since $\mcA$ is self-adjoint,
\[
(\widehat{\mcA})^{*}-\mcA=(\widehat{\mcA}-\mcA)^{*}.
\]
Moreover, by the dual characterization of Sobolev norms and the definition of the adjoint,
\begin{align*}
\|(\widehat{\mcA})^{*}-\mcA\|_{H^{t}\to H^{-t'}}
&=\|(\widehat{\mcA}-\mcA)^{*}\|_{H^{t}\to H^{-t'}} \\
&=\sup_{\|w\|_{H^{t}}\le 1,\|v\|_{H^{t'}}\le 1}|\langle (\widehat{\mcA}-\mcA)^{*}w,\,v\rangle| \\
&=\sup_{\|w\|_{H^{t}}\le 1,\|v\|_{H^{t'}}\le 1}|\langle w,\,(\widehat{\mcA}-\mcA)v\rangle|
=\|\widehat{\mcA}-\mcA\|_{H^{t'}\to H^{-t}}.
\end{align*}
Therefore, applying Theorem~\ref{thm:main1} with $t$ and $t'$ exchanged yields an upper bound for
$\|(\widehat{\mcA})^{*}-\mcA\|_{H^{t}\to H^{-t'}}$. Since the exponent $\rho$ in~\eqref{eq:rho} is symmetric in $(t,t')$, the resulting convergence rate coincides with that in the regime $t\le t'$.
\end{remark}

\begin{remark}[Discussion on the convergence rate]\label{remark:rate_discussion}

We make several remarks on the exponent $\rho$ in \eqref{eq:rho}, which governs the convergence rate.

\begin{enumerate}
    \item \textbf{Dependence on $t,t'$.}
The indices $t,t'$ quantify the strength of the operator norm $\|\cdot\|_{H^t\to H^{-t'}}$ used to measure the estimation
error. Roughly speaking, \emph{larger} $t,t'$ correspond to a \emph{weaker} norm, hence a less stringent error metric. When $t$ and $t'$ are sufficiently large relative to $r_1,r_2,n$, namely when 
\begin{align}\label{eq:rate_discussion_aux1}
\min\{t,t'\}\ge -r_2+\frac{n}{2}
\quad\text{and}\quad
t+t'\ge r_1-r_2+n,
\end{align}
we have $\rho=0$, and thus attain the parametric rate $N^{-1/2}$ (up to logarithmic factors). Indeed, under \eqref{eq:rate_discussion_aux1} the effective variance grows at most polynomially in
$J$ (as will be reflected in our error analysis), so the bias--variance tradeoff takes the form
\[
2^{-J(t+t'-r)}+\frac{\poly(J)}{\sqrt{N}}
\quad \Longrightarrow \quad
N^{-1/2}\,\polylog(N).
\]
When \eqref{eq:rate_discussion_aux1} fails, the variance grows exponentially in $J$, leading to a
genuinely nonparametric rate $(\rho>0)$:
\[
2^{-J(t+t'-r)}+\frac{2^{J\rho(t+t'-r)/2}\poly(J)}{\sqrt{N}}
\quad \Longrightarrow \quad
N^{-\frac{1}{2+\rho}}\,\polylog(N).
\]

\item  \textbf{Dependence on $n,r,r_1,r_2$.}
Recall that $n$ is the dimension of the manifold $\mcM$, and $r$ is the order of the PDO $\mcA$, i.e., $\mcA\in \OPS^r(\mcM)$. The parameters $r_1$ and $r_2$
control the regularity of the input and the noise: $u\in H^{(r_1-n/2)-}$ and $w\in H^{(r_2-n/2)-}$ almost surely; see Remark~\ref{remark:assumption_1}. From \eqref{eq:rho},
$\rho$ is increasing in $n,r,r_1$ and decreasing in $r_2$. Consequently, a lower-dimensional
manifold (smaller $n$), a lower-order operator (smaller $r$), a rougher input (smaller $r_1$), and/or a smoother noise (larger $r_2$) lead to a smaller $\rho$, and hence a faster convergence rate
$N^{-\frac{1}{2+\rho}}$.

\item \textbf{Dependence on $\sigma,\widetilde d$.}
The parameters $\sigma$ and $\widetilde d$ depend on the choice of biorthogonal wavelets.
The parameter $\sigma$ quantifies the off-diagonal decay of the entries of $\bfA$ across scale
separation (especially for wavelets with overlapping supports); see
Proposition~\ref{prop:property_UAW} \ref{approximate_sparsity} (ii). It is required to satisfy
Assumption~\ref{assumption:wavelets} (iii): its lower bound is fixed, while its admissible upper
bound increases with the wavelet regularity and approximation parameters
$(\gamma,\widetilde\gamma,\widetilde d)$. In particular, by choosing smoother wavelets with higher approximation order, one can take $\sigma$ larger (indeed, $\sigma$ may scale
proportionally to $\gamma,\widetilde\gamma,\widetilde d$). When \eqref{eq:rate_discussion_aux1} fails, at least one of the first three terms in
\eqref{eq:rho} is strictly positive. In this case, increasing $\sigma$ and/or $\widetilde d$ reduces $\rho$, and therefore improves the convergence rate $N^{-\frac{1}{2+\rho}}$.

\item \textbf{Intrinsic convergence rate.}
In an idealized regime where $\sigma$ and $\widetilde d$ are taken sufficiently large (formally,
$\sigma,\widetilde d\to\infty$), the first three contributions in \eqref{eq:rho} vanish and $\rho$
reduces to
\[
\rho
=2\max\left\{\frac{-t-t'+r_1-r_2}{t+t'-r},\,0\right\}.
\]
This exponent is independent of the wavelet choice and determines an intrinsic convergence rate governed by $(t,t',r_1,r_2,r)$. Intuitively, increasing $\sigma$ and $\widetilde d$ allows one to fully exploit the
off-diagonal smoothness of $\mcA$, yet the rate is still governed by the interplay between the error metric $(t,t')$, the input function regularity $r_1$, and the noise regularity $r_2$. In this regime, the bias--variance tradeoff takes the form
\begin{align}\label{eq:rate_discussion_aux2}
2^{-J(t+t'-r)}+\frac{2^{J\max\{-t-t'+r_1-r_2,\,0\}}}{\sqrt{N}}
\quad\Longrightarrow\quad
N^{-\min\left\{\frac{t+t'-r}{2(r_1-r_2-r)},\,\frac12\right\}}.
\end{align}
The quantity $r_1-r_2-r$ may be interpreted as an \emph{effective smoothness}, i.e., the gap
between the regularity of the signal $\mcA u\in H^{(r_1-r-n/2)-}$ and that of the noise
$w\in H^{(r_2-n/2)-}$. Larger values of $r_1-r_2-r$ (smoother $\mcA u$ and/or rougher noise $w$) make
the problem harder and lead to slower learning rates. In particular, when $r_1-r_2>t+t'$ (equivalently,
$r_1-r_2-r>t+t'-r$), we have
\[
\rho=\frac{2(-t-t'+r_1-r_2)}{t+t'-r},
\qquad\text{and hence}\qquad
N^{-\frac{1}{2+\rho}}
= N^{-\frac{t+t'-r}{2(r_1-r_2-r)}}.
\]

Finally, we remark that the factor
$2^{J\max\{-t-t'+r_1-r_2,0\}}/\sqrt{N}$ corresponds to the standard deviation of estimating a single
entry in the $(0,0)$ and $(J,J)$ blocks of the compressed matrix $\bfA^{\varepsilon}_{\Lambda_J}$. To see this, consider the simplest idealized setting in which $\bfA$ is exactly diagonal ---an assumption that typically fails for a general PDO $\mcA$ in wavelet coordinates. In this case, the $\lambda'$-th column
equation in \eqref{eq:model_column} reduces to
\[
\bfF_{\cdot,\lambda'}
=\bfU \bfA_{\cdot,\lambda'}+\bfW_{\cdot,\lambda'}
=\bfA_{\lambda',\lambda'} \bfU_{\cdot,\lambda'}+\bfW_{\cdot,\lambda'},
\]
where $\bfU_{\cdot,\lambda'}\sim \mcN \bigl(0,(\sigma^{U}_{\lambda'})^2 I_N\bigr)$ and
$\bfW_{\cdot,\lambda'}\sim \mcN \bigl(0,(\sigma^{W}_{\lambda'})^2 I_N\bigr)$ are independent, with
$\sigma^{U}_{\lambda'}\asymp 2^{-j'r_1}$ and $\sigma^{W}_{\lambda'}\asymp 2^{-j'r_2}$ by
Proposition~\ref{prop:property_UAW} \ref{diagonal_preconditioning} (ii)-(iii). Thus, recovering the $\lambda'$-th column of $\bfA$ reduces to estimating the single scalar parameter $\bfA_{\lambda',\lambda'}\in\R$.

In this Gaussian location model, the scalar least squares estimator (equivalently, the maximum
likelihood estimator) satisfies
\[
|\widehat{\bfA}_{\lambda',\lambda'}-\bfA_{\lambda',\lambda'}|
\asymp \frac{\sigma^{W}_{\lambda'}}{\sqrt{N}\,\sigma^{U}_{\lambda'}}
\asymp \frac{2^{j'(r_1-r_2)}}{\sqrt{N}},
\]
where $\sigma^{U}_{\lambda'}/\sigma^{W}_{\lambda'}$ is the signal-to-noise ratio. Moreover, this rate
is minimax optimal (e.g., by a direct Le Cam's two-point argument). Finally, the norm
$\|\cdot\|_{H^{t}\to H^{-t'}}$ contributes the additional weight $2^{-j'(t+t')}$, so the
\emph{weighted} standard deviation scales as
\[
\frac{2^{j'(-t-t'+r_1-r_2)}}{\sqrt{N}}.
\]
This equals $1/\sqrt{N}$ when $j'=0$, and equals $2^{J(-t-t'+r_1-r_2)}/\sqrt{N}$ when $j'=J$.
This explains the appearance of $2^{J\max\{-t-t'+r_1-r_2,0\}}/\sqrt{N}$ in
\eqref{eq:rate_discussion_aux2}, and suggests that, without further assumptions on the operator or on the covariance
structure of the input function and the noise, the rate in \eqref{eq:rate_discussion_aux2} cannot be improved. \qedhere
\end{enumerate}

\end{remark}

\smallskip

\begin{remark}[Sparsity of $\widehat{\mcA}$: $\mathcal{O}(2^{Jn})$ nonzero coefficients]\label{rem:sparsity}
Our estimator $\widehat{\mcA}$ is optimally sparse in wavelet coordinates, in the sense that its matrix representation
$\widehat{\bfA}\in \R^{\Lambda_J\times \Lambda_J}$ has only $\mathcal{O}(2^{Jn})$ nonzero entries, which is linear in $|\Lambda_J|\asymp 2^{Jn}$. Indeed, $\widehat{\bfA}$ is supported on $\supp(J,t,t')$ by construction, and Proposition~\ref{prop:sparsity_count} guarantees that
$\mathrm{nnz}(\widehat{\bfA})= \mathrm{nnz}\big(\bfM_{(J,t,t')}\big)\lesssim 2^{Jn}$.
Under the bias--variance optimal choice of $J$ in \eqref{eq:def_parameter}, we have $\mathcal{O}(2^{Jn})=\mathcal{O}\big(N^{\frac{n}{(2+\rho)(t+t'-r)}}\big)$ nonzero coefficients in $\widehat{\bfA}$.
\end{remark}

\begin{remark}[Computational cost: nearly $\mathcal{O}(N2^{Jn})$]\label{remark:cost2}
In this remark, we continue the discussion of computational cost from Remark \ref{remark:cost1} and give a more precise bound of this cost under our choice of parameters in~\eqref{eq:def_parameter}.
Recall from \eqref{eq:cost} that the computational cost of our estimator is determined by the sparsity structure of the regression support $\supp(\widetilde{J},\widetilde{t},\widetilde{t}')$:
\[
\mathcal{O}\bigg(
N\bigg(\sum_{\lambda'\in\Lambda_J}|\Omega_{\lambda'}|^{2}\bigg)
+\sum_{\lambda'\in\Lambda_J}|\Omega_{\lambda'}|^{3}
\bigg).
\]
 To simplify notation, let $\widetilde{J}:=(1+\varepsilon')J$ and define
\[
\nu:=\frac{n(\widetilde{t}+\widetilde{t}'-r)}{\sigma-n/2+\widetilde{t}'-r/2}.
\]
By the column-wise support bound in Proposition~\ref{prop:sparsity_count}, for $\lambda'=(j',k')\in\Lambda_J$,
\[
|\Omega_{\lambda'}|\;\lesssim\;\widetilde{J}+2^{\nu(\widetilde{J}-j')}.
\]
For $p\in\{2,3\}$, using $(x+y)^p\le 2^{p-1}(x^p+y^p)$ (for $x,y>0$) and $|\nabla_{j'}|\asymp 2^{j'n}$, we obtain
\begin{align*}
\sum_{\lambda'\in\Lambda_J}|\Omega_{\lambda'}|^p
&\lesssim \sum_{j'=0}^J \sum_{k'\in\nabla_{j'}}\Big(\widetilde{J}+2^{\nu(\widetilde{J}-j')}\Big)^p\\
&\lesssim \sum_{j'=0}^J 2^{j'n}\Big(\widetilde{J}^p+2^{p\nu(\widetilde{J}-j')}\Big)\\
&\lesssim \widetilde{J}^p\sum_{j'=0}^J 2^{j'n} \;+\; 2^{p\nu\widetilde{J}}\sum_{j'=0}^J 2^{(n-p\nu)j'}.
\end{align*}
Assume $\nu<n/p$ (which holds for $\sigma$ large enough). Then the geometric sum satisfies
$\sum_{j'=0}^J 2^{(n-p\nu)j'}\lesssim 2^{(n-p\nu)J}$, and hence
\[
\sum_{\lambda'\in\Lambda_J}|\Omega_{\lambda'}|^p
\lesssim J^p 2^{Jn}+
2^{Jn}\,2^{p\nu(\widetilde{J}-J)}
\lesssim
2^{Jn\Big(1+\frac{p\varepsilon'(\widetilde{t}+\widetilde{t}'-r)}{\sigma-n/2+\widetilde{t}'-r/2}\Big)}.
\]
In particular,
\[
\sum_{\lambda'\in\Lambda_J}|\Omega_{\lambda'}|^2
\lesssim
2^{Jn\Big(1+\frac{2\varepsilon'(\widetilde{t}+\widetilde{t}'-r)}{\sigma-n/2+\widetilde{t}'-r/2}\Big)},
\qquad
\sum_{\lambda'\in\Lambda_J}|\Omega_{\lambda'}|^3
\lesssim
2^{Jn\Big(1+\frac{3\varepsilon'(\widetilde{t}+\widetilde{t}'-r)}{\sigma-n/2+\widetilde{t}'-r/2}\Big)}.
\]
Therefore, the total computational cost satisfies
\begin{align}\label{eq:cost_aux1}
N\bigg(\sum_{\lambda'\in\Lambda_J}|\Omega_{\lambda'}|^2\bigg) + \sum_{\lambda'\in\Lambda_J}|\Omega_{\lambda'}|^3
\lesssim
N 2^{Jn(1+\varepsilon'')} + 2^{Jn\left(1+\frac{3}{2}\varepsilon''\right)}\lesssim
N 2^{Jn(1+\varepsilon'')},
\end{align}
provided that
\[
N \gtrsim 2^{Jn\varepsilon''/2},
\qquad\text{where}\quad
\varepsilon'':=\frac{2\varepsilon'(\widetilde{t}+\widetilde{t}'-r)}{\sigma-n/2+\widetilde{t}'-r/2}>0.
\]
A sufficient (stronger) condition is
\[
N\gtrsim 2^{\widetilde{J}n(\widetilde{t}+\widetilde{t}'-r)/(\sigma-n/2+\widetilde{t}'-r/2)}.
\]
This holds under our choice $J=\left\lceil \frac{\log_2 N}{(2+\rho)(t+t'-r)} \right\rceil$ in \eqref{eq:def_parameter}; see \eqref{eq:sample_size_condition}.  

Since $\varepsilon''\to 0$ as $\sigma\to\infty$ (i.e., as we choose smoother biorthogonal wavelets), the computational cost~\eqref{eq:cost_aux1} can be made arbitrarily close to $\mathcal{O}(N2^{Jn})$, that is, linear in the sample size $N$ and, up to constants, linear in the number of wavelet indices $|\Lambda_J|\asymp 2^{Jn}$. Recall that the estimator in \eqref{eq:estimator_matrix_box} satisfies
$\widehat{\bfA}\in\R^{\Lambda_J\times\Lambda_J}$.
In contrast, any direct estimator that explicitly forms $\bfU_{\cdot,\Lambda_J}^{\top}\bfU_{\cdot,\Lambda_J}$
already incurs a cost of $\mathcal{O}(N2^{2Jn})$, i.e., quadratic in $|\Lambda_J|$.
Furthermore, under the choice of $J$ in~\eqref{eq:def_parameter}, which optimally balances truncation,
compression, and estimation errors, we obtain
\begin{align}\label{eq:cost_final}
\mathcal{O}(N 2^{Jn})= \mathcal{O}\left(N^{1+\frac{n}{(2+\rho)(t+t'-r)}}\right).
\end{align}

Finally, note that the dominant contribution to the above bounds comes from columns with $j'=J$. Indeed, when $j'=J$ and $k'\in\nabla_J$, one expects the corresponding column of the compressed matrix
$\bfA^{\varepsilon}_{\Lambda_J}$ to contain only $\mathcal{O}(1)$ significant entries. We conjecture that, with a sharper choice of the enlarged regression supports, one can improve the bound \eqref{eq:cost_aux1} and achieve the
\emph{exact} optimal computational cost $\mathcal{O}(N2^{Jn})$, while keeping the statistical convergence rate established in Theorem~\ref{thm:main1}; we leave this question for future work.
\end{remark}

\begin{remark}[Noiseless setting: super-algebraic convergence rate]\label{rem:noisefree}  We can adapt the analysis of $\widehat{\mcA}$ to the noiseless setting in a straightforward way. Assume that the data pairs $\{u_i,f_i\}_{i=1}^{N}$ satisfy
\begin{align}\label{eq:continuous_model_noise_free}
f_i = \mcA u_i, \qquad 1 \le i \le N,
\end{align}
so that in the wavelet–matrix formulation
\begin{align}\label{eq:discrete_model_noise_free}
\bfF=\bfU\bfA.
\end{align}
Under Assumption~\ref{assumption:operator_data_noise} (without (iii)) and Assumption \ref{assumption:wavelets}, we use the same estimator
$\widehat{\mcA}$ as in \eqref{eq:estimator_operator_box}. Following the same argument as in the noisy case yields
\[
\|\widehat{\mcA}-\mcA\|_{H^{t}\to H^{-t'}}\lesssim \ErrorTruncation+\ErrorCompression+\ErrorOVB,
\]
where the variance term $\ErrorVar$ is absent. By Propositions~\ref{prop:Error-Truncation}, \ref{prop:Error-Compression}, and
\ref{prop:Error-OVB}, if
\begin{align*}
N\gtrsim 2^{\widetilde{J}n(\widetilde{t}+\widetilde{t}'-r)/(\sigma-n/2+\widetilde{t}'-r/2)}+\log(1/\delta),
\end{align*}
then with probability at least $1-\delta$,
\begin{align*}
\|\widehat{\mcA}-\mcA\|_{H^{t}\to H^{-t'}}
\lesssim
\ErrorTruncation+\ErrorCompression+\ErrorOVB
\lesssim
J\,2^{-J(t+t'-r)}.
\end{align*}

Consequently, we may choose $\widetilde{J}$ by saturating the sample size constraint,
\[
N\asymp 2^{\widetilde{J}n(\widetilde{t}+\widetilde{t}'-r)/(\sigma-n/2+\widetilde{t}'-r/2)}
\quad\Longrightarrow\quad
\widetilde{J}=\left\lceil
\frac{\sigma-n/2+\widetilde{t}'-r/2}{n(\widetilde{t}+\widetilde{t}'-r)}\,\log_2 N
\right\rceil,
\]
and then set, using the scaling relation between $J$ and $\widetilde{J}$ in \eqref{eq:def_parameter},
\begin{align}\label{eq:J_noiseless}
J=\left\lceil
\frac{\min\{t',r_1\}+t-r}{t+t'-r+\varepsilon_1}\,\widetilde{J}
\right\rceil=\left\lceil
\frac{\min\{t',r_1\}+t-r}{t+t'-r+\varepsilon_1}\cdot \frac{\sigma-n/2+\widetilde{t}'-r/2}{n(\widetilde{t}+\widetilde{t}'-r)}\,\log_2 N
\right\rceil.
\end{align}
With this choice,
\begin{align}\label{eq:rate_supe_ralgebraic}
\|\widehat{\mcA}-\mcA\|_{H^{t}\to H^{-t'}}
\lesssim
J\,2^{-J(t+t'-r)}
\,\lesssim_{\,q}\, N^{-q}\log N,
\end{align}
where
\begin{align}\label{eq:q}
q
:=
(t+t'-r)\cdot
\frac{\min\{t',r_1\}+t-r}{t+t'-r+\varepsilon_1}\cdot
\frac{\sigma-n/2+\widetilde{t}'-r/2}{n(\widetilde{t}+\widetilde{t}'-r)}.
\end{align}
In particular, $q$ grows linearly with $\sigma$ (and $\sigma$ can grow linearly with $\gamma,\widetilde{\gamma},\widetilde{d}$), so $q$ can be made arbitrarily large by choosing sufficiently smooth biorthogonal wavelets with sufficiently high approximation orders. Equivalently, for any $q_0>0$, one can take the wavelet so that $q\ge q_0$, which yields the high-probability bound
\[
\|\widehat{\mcA}-\mcA\|_{H^{t}\to H^{-t'}}\;\lesssim_{\,q_0}\; N^{-q_0} \log N.
\]
Thus, in the noiseless setting, the error decays super-algebraically in $N$ (up to a logarithmic factor) as the wavelet smoothness increases. 

Moreover, under the choice of $J$ in \eqref{eq:J_noiseless}, the estimated matrix $\widehat{\bfA}\in \R^{\Lambda_J\times \Lambda_J}$  has size $|\Lambda_J| \asymp 2^{Jn}  \asymp N^{\frac{qn}{t+t'-r}},$
and $\widehat{\bfA}$ contains $\mathcal{O}(2^{Jn})=\mathcal{O}\big(N^{\frac{qn}{t+t'-r}}\big)$ nonzero wavelet coefficients. Moreover, the computational cost of constructing the estimator is nearly $\mathcal{O}(N2^{Jn})=\mathcal{O}\big(N^{1+\frac{qn}{t+t'-r}}\big)$, where $q$ is as in \eqref{eq:q}.
\end{remark}

\begin{example}[Learning Green's function of a uniformly elliptic PDE]\label{rem:example}

In the discussion above, the size of the learned matrix, the sparsity level (i.e., the number of nonzero coefficients), and the computational cost all depend on the sample size $N$. This reflects a standard feature of nonparametric statistics: the effective dimension of the learning problem grows with the amount of data. Since our bounds depend on multiple parameters, we illustrate their implications below through a concrete example.

 Consider the (suitably normalized) solution operator (Green's function) associated with a uniformly elliptic PDE of order $2s$, where $s>0$. In this setting, the solution operator belongs to $\OPS^{-2s}(\mcM)$, i.e., $r=-2s$. Recall that we work on an $n$-dimensional smooth closed manifold; extensions to more general domains with appropriate boundary conditions are natural. To simplify the discussion, we consider the operator norm $\|\cdot\|_{L^2(\mcM)\to L^2(\mcM)}$, i.e., we set $t=t'=0$. 

\begin{itemize}
    \item \textbf{Noisy setting.}  Since $r=-2s$, $t=t'=0$, and $r_2>n/2$, the parameter $\rho$ in \eqref{eq:rho} becomes
\[
\rho=\max\left\{\frac{r_1-r_2+n}{\widetilde{d}} ,\,\frac{r_1-r_2}{s},0\right\}.
\]
By Theorem \ref{thm:main1} and Remarks \ref{rem:sparsity} and \ref{remark:cost2}, the convergence rate, the size of the matrix estimator $\widehat{\bfA}$ and its number of nonzero entries (nnz), and the computational cost of constructing the estimator scale with $N$ as follows:
\begin{align*}
&\text{Convergence rate:} \quad \mathcal{O}\big(N^{-\frac{1}{2+\rho}}\log^{\frac{3}{2}} (N)\big). \\
&\text{Size and nnz of $\widehat{\bfA}$:} \quad \mathcal{O}\big(N^{\frac{n}{2s(2+\rho)}}\big).\\
&\text{Computational cost:} \quad \mathcal{O}\big(N^{1+\frac{n}{2s(2+\rho)}}\big).
\end{align*}
Therefore, to achieve a target accuracy tolerance $\epsilon$ under $\|\cdot\|_{L^2\to L^2}$ in the noisy setting, i.e., to ensure $N^{-\frac{1}{2+\rho}}\log^{\frac{3}{2}} N \asymp \epsilon$, it suffices to take
\begin{align*}
\text{\#input--output noisy data pairs:} \quad \mathcal{O}\big(\epsilon^{-(2+\rho)}\log^{\frac{3(2+\rho)}{2}} (\epsilon^{-1})\big).
\end{align*}
This in turn implies:
\begin{align*}
&\text{Size and nnz of $\widehat{\bfA}$:} \quad \mathcal{O}\big(\epsilon^{-\frac{n}{2s}}\log^{\frac{3n}{4s}} (\epsilon^{-1})\big).\\
&\text{Computational cost:} \quad \mathcal{O}\big(\epsilon^{-(2+\rho)-\frac{n}{2s}}\log^{\frac{6s(2+\rho)+3n}{4s}} (\epsilon^{-1})\big).
\end{align*}

\item \textbf{Noiseless setting.} 

When $r=-2s$ and $t=t'=0$, the exponent $q$ in \eqref{eq:rate_supe_ralgebraic} and \eqref{eq:q} becomes
\[
q=\frac{2s}{1+\frac{n}{\sigma-n/2+s}}\cdot \frac{\sigma-n/2+r_1+s}{n(r_1+2s)}.
\]
By Remark \ref{rem:noisefree},  we now have:
\begin{align*}
&\text{Convergence rate:} \quad \mathcal{O}\big(N^{-q}\log N\big). \\
&\text{Size and nnz of $\widehat{\bfA}$:} \quad \mathcal{O}\big(N^{\frac{qn}{2s}}\big).\\
&\text{Computational cost:} \quad \mathcal{O}\big(N^{1+\frac{qn}{2s}}\big).
\end{align*}

Therefore, to achieve a target accuracy tolerance $\epsilon$ under $\|\cdot\|_{L^2\to L^2}$ in the noiseless setting, i.e., to ensure $N^{-q}\log N \asymp \epsilon$, it suffices to take
\begin{align*}
\text{\#input--output noiseless data pairs:} \quad \mathcal{O}\big(\epsilon^{-\frac{1}{q}}\log^{\frac{1}{q}} (\epsilon^{-1})\big).
\end{align*}
This in turn implies:
\begin{align*}
&\text{Size and nnz of $\widehat{\bfA}$:} \quad \mathcal{O}\big(\epsilon^{-\frac{n}{2s}}\log^{\frac{n}{2s}} (\epsilon^{-1})\big).\\
&\text{Computational cost:} \quad \mathcal{O}\big(\epsilon^{-\frac{1}{q}-\frac{n}{2s}}\log^{\frac{1}{q}+\frac{n}{2s}} (\epsilon^{-1})\big).
\end{align*}
\end{itemize}

We compare the above results with \cite{schafer2024sparse}, which studies recovery of a discretized (finite-dimensional) Green’s matrix from carefully designed, noiseless vector measurements. Their method exploits approximate sparsity in a Cholesky factorization of solution operators for elliptic PDEs of order $2s$, where $s$ is a positive integer, and then approximates the continuous Green’s function via piecewise constant or piecewise affine interpolation. In particular, to achieve accuracy $\epsilon$ for recovering the continuous Green's function under the $\|\cdot\|_{L^2\to L^2}$ norm, \cite{schafer2024sparse} requires computational cost $\mathcal{O}\big(\epsilon^{-n}\log^{2n+2}(\epsilon^{-1})\big)$ using $\mathcal{O}\big(\log^{n+1}(\epsilon^{-1})\big)$ noiseless, carefully chosen data pairs; see \cite[Theorem 3.5]{schafer2024sparse}. The leading $\mathcal{O}(\epsilon^{-n})$ dependence arises from the final low-order interpolation step that uses the recovered discrete Green's matrix to approximate the continuous Green's function; consequently, the PDE order $2s$ (and hence the strength of elliptic smoothing) does not appear in the leading $\epsilon$-exponent.

By contrast, our noiseless guarantees show that with generic Gaussian input data, achieving accuracy $\epsilon$ requires $\mathcal{O}\big(\epsilon^{-\frac{1}{q}}\log^{\frac{1}{q}} (\epsilon^{-1})\big)$ data pairs and the computational cost is of order $\mathcal{O}\big(\epsilon^{-\frac{1}{q}-\frac{n}{2s}}\log^{\frac{1}{q}+\frac{n}{2s}} (\epsilon^{-1})\big)$. By choosing smoother wavelets with higher approximation order, the exponent $q$ can be made large, so the required number of data pairs can grow arbitrarily slower than any fixed polynomial in $\epsilon^{-1}$. In the large-$q$ regime, our computational cost approaches $\mathcal{O}\big(\epsilon^{-\frac{n}{2s}}\log^{\frac{n}{2s}}(\epsilon^{-1})\big)$, which (when $s>1/2$) improves the $\epsilon$-dependence relative to $\mathcal{O}\big(\epsilon^{-n}\log^{2n+2}(\epsilon^{-1})\big)$ in \cite[Theorem 3.5]{schafer2024sparse}. We explain why the computational cost $\mathcal{O}\big(\epsilon^{-\frac{n}{2s}}\log^{\frac{n}{2s}}(\epsilon^{-1})\big)$ is optimal up to logarithmic factors in this setting. To achieve accuracy $\epsilon$ in $\|\cdot\|_{L^2\to L^2}$, one can at most discard information beyond the scale $J$ such that $2^{-2Js}\asymp \epsilon$, since $\mathcal{O}(2^{-2Js})$ is the unavoidable bias incurred by ignoring finer scales. In $n$ dimensions, the number of degrees of freedom up to scale $J$ is of order $\mathcal{O}(2^{Jn})$, and each degree of freedom must be learned at a cost of at least $\mathcal{O}(1)$. Therefore, any method must have total cost at least 
\[
\mathcal{O}(2^{Jn})= \mathcal{O}\big((2^{-2J s})^{-\frac{n}{2s}}\big)= \mathcal{O}\big(\epsilon^{-\frac{n}{2s}}\big),
\]
matching our complexity bound up to logarithmic factors.

In addition, our framework accommodates noisy observations and generic Gaussian inputs, works under general Sobolev-to-Sobolev operator norms, and applies to fractional elliptic PDEs (where the underlying differential operator is nonlocal) ---settings not covered by \cite{schafer2024sparse}. In the noisy case, achieving $\epsilon$ accuracy entails computational cost of order (up to logarithmic factors)
\[
\mathcal{O}\big(\epsilon^{-(2+\rho)-\frac{n}{2s}}\big), \quad \rho=\max\left\{\frac{r_1-r_2+n}{\widetilde{d}} ,\,\frac{r_1-r_2}{s},0\right\},
\]
which is worse by a factor of roughly $\epsilon^{-(2+\rho)}$ compared with our noiseless scaling. The $\mathcal{O}(\epsilon^{-2})$ factor is unavoidable in general: it reflects the central-limit-theorem scaling induced by additive noise. In particular, to estimate each degree of freedom to accuracy $\epsilon$ from noisy samples, one typically needs $\mathcal{O}(\epsilon^{-2})$ observations. The factor $\mathcal{O}\big(\epsilon^{-\frac{n}{2s}}\big)$ accounts for the number of degrees of freedom that must be learned (up to the accuracy-determining scale), as discussed in the previous paragraph. To our knowledge, this provides the first explicit and nearly optimal statistical accuracy--computational cost tradeoff for learning Green's functions of elliptic PDEs from noisy data.
\end{example}

\smallskip

\subsection{Proof of Theorem \ref{thm:main1}}\label{ssec:proof_main1}

This subsection proves Theorem~\ref{thm:main1}. We first decompose the total error into three contributions: truncation, compression, and estimation. The truncation error is the bias from restricting the infinite-dimensional matrix $\bfA$ to its finite-resolution version $\bfA_{\Lambda_J}$, while the compression error is the deterministic approximation error from sparsifying $\bfA_{\Lambda_J}$ to $\bfA^{\varepsilon}_{\Lambda_J}$; see Subsection \ref{ssec:compression}. Both terms are data-independent and are controlled by Propositions~\ref{prop:Error-Truncation} and~\ref{prop:Error-Compression}, respectively.
The estimation error is the only data-dependent term and quantifies the statistical uncertainty in estimating the compressed matrix $\bfA_{\Lambda_J}^{\varepsilon}$. Lemma~\ref{lemma:error_decomposition} further decomposes it into an omitted-variable bias term (still random) and a variance term, which are bounded in Propositions~\ref{prop:Error-OVB} and~\ref{prop:Error-Var}, respectively. Proofs of the lemma and propositions are deferred to Appendices~\ref{app:truncation}, \ref{app:compression}, and~\ref{app:estimation}.

\begin{proof}[Proof of Theorem \ref{thm:main1}]

Let $\bfD^s_{\Lambda_J}:=(\bfD^s_{\lambda,\lambda'})_{\lambda,\lambda'\in \Lambda_J}\in \R^{\Lambda_J\times \Lambda_J}$ denote the diagonal weight matrix restricted to $\Lambda_J$. For any index set $\Lambda\subset \mcJ$, let $P_{\Lambda}:\ell^2(\mcJ)\to\ell^2(\Lambda)$ be the coordinate projection and $I_{\Lambda}:=P_{\Lambda}^*$ be its adjoint (the zero-padding injection).
For our matrix estimator $\widehat{\bfA}\in\R^{\Lambda_J\times\Lambda_J}$ in \eqref{eq:estimator_matrix_box}, and the \emph{a priori} truncation matrix $\bfA_{\Lambda_J}\in\R^{\Lambda_J\times\Lambda_J}$ in Definition \ref{def:matrix_truncation}, define their zero-padded extensions
\[
\widehat{\bfA}^{\uparrow} := I_{\Lambda_J}\widehat{\bfA}P_{\Lambda_J}\in \R^{\mcJ\times\mcJ}, \qquad
\bfA_{\Lambda_J}^{\uparrow} := I_{\Lambda_J}\bfA_{\Lambda_J}P_{\Lambda_J}\in \R^{\mcJ\times\mcJ}.
\]

We begin with the operator-norm equivalence \eqref{eq:matrix_norm},
\[
\|\widehat{\mcA}-\mcA\|_{H^{t}\to H^{-t'}}
\asymp \| \bfD^{-t'}(\widehat{\bfA}^{\uparrow}-\bfA)\bfD^{-t}\|.
\]
Adding and subtracting $\bfA_{\Lambda_J}^{\uparrow}$ yields
\begin{align}\label{eq:error_decomposition_aux1}
\|\widehat{\mcA}-\mcA\|_{H^{t}\to H^{-t'}}
&\asymp \| \bfD^{-t'}(\widehat{\bfA}^{\uparrow}-\bfA)\bfD^{-t}\| \nonumber \\
&\le
\|\bfD^{-t'}(\bfA_{\Lambda_J}^{\uparrow}-\bfA)\bfD^{-t}\|+\| \bfD^{-t'}(\widehat{\bfA}^{\uparrow}-\bfA_{\Lambda_J}^{\uparrow})\bfD^{-t}\|
\nonumber\\
&=\underbrace{\| \bfD^{-t'} (\bfA_{\Lambda_J}^{\uparrow} - \bfA) \bfD^{-t} \|}_{=:\ErrorTruncation}+\|\bfD_{\Lambda_J}^{-t'}(\widehat{\bfA}-\bfA_{\Lambda_J})\bfD_{\Lambda_{J}}^{-t}\|.
\end{align}
Decomposing the second term  gives
\[
\|\bfD_{\Lambda_J}^{-t'}(\widehat{\bfA}-\bfA_{\Lambda_J})\bfD_{\Lambda_{J}}^{-t}\|\le \underbrace{\| \bfD^{-t'}_{\Lambda_J} ( \bfA_{\Lambda_J}^{\varepsilon}-\bfA_{\Lambda_J}) \bfD^{-t}_{\Lambda_J} \|}_{=:\ErrorCompression}+\underbrace{\| \bfD^{-t'}_{\Lambda_J} (\widehat{\bfA} - \bfA_{\Lambda_J}^{\varepsilon}) \bfD^{-t}_{\Lambda_J} \|}_{=:\ErrorEstimation}.
\]
By Lemma~\ref{lemma:error_decomposition}, the estimation error further decomposes into the omitted–variable bias
and the variance contributions:
\[
\ErrorEstimation=\| \bfD^{-t'}_{\Lambda_J} (\widehat{\bfA} - \bfA_{\Lambda_J}^{\varepsilon}) \bfD^{-t}_{\Lambda_J} \|\le \ErrorOVB+\ErrorVar.
\]
Hence,
\begin{align*}
\|\widehat{\mcA}-\mcA\|_{H^{t}\to H^{-t'} }
\lesssim  \ErrorTruncation+\ErrorCompression+\ErrorOVB+\ErrorVar.
\end{align*}

Propositions \ref{prop:Error-Truncation} and \ref{prop:Error-Compression} give the deterministic bounds
\[
\ErrorTruncation\lesssim 2^{-J(t+t'-r)},\qquad \ErrorCompression\lesssim J2^{-J(t+t'-r)}.
\]
Moreover, Propositions \ref{prop:Error-OVB} and \ref{prop:Error-Var} show that, for any $\delta\in (0,1)$, if 
\begin{align}\label{eq:sample_size_condition}
N\gtrsim 2^{\widetilde{J}n(\widetilde{t}+\widetilde{t}'-r)/(\sigma-n/2+\widetilde{t}'-r/2)}+\log (1/\delta),
\end{align}
then with probability at least $1-\delta$,
\[
\ErrorOVB \lesssim J2^{-J(t+t'-r)}, \qquad  \ErrorVar\lesssim \sqrt{\frac{J+\log(1/\delta)}{N}}\cdot J\cdot 2^{\rho J(t+t'-r)/2},
\]
where $\widetilde{J},\widetilde{t},\widetilde{t}'$ are defined in \eqref{eq:def_parameter} and $\rho$ is defined in \eqref{eq:rho}.

Thus, under \eqref{eq:sample_size_condition}, with probability at least $1-\delta$ we have
\[
\|\widehat{\mcA}-\mcA\|_{H^{t}\to H^{-t'}}\lesssim J2^{-J(t+t'-r)}+\sqrt{\frac{J+\log(1/\delta)}{N}}\cdot J\cdot 2^{\rho J(t+t'-r)/2}.
\]

To balance the bias and variance terms, choose
\begin{align*}
J:=\left\lceil \frac{\log_2 N}{(2+\rho)(t+t'-r)} \right\rceil.
\end{align*}
Then
\[
2^{-J(t+t'-r)}\asymp N^{-\frac{1}{2+\rho}},\qquad 2^{\rho J(t+t'-r)/2}\asymp N^{\frac{\rho}{2(2+\rho)}},
\]
and hence
\[
\|\widehat{\mcA}-\mcA\|_{H^{t}\to H^{-t'}}
\;\lesssim\;
N^{-\frac{1}{2+\rho}}\,
\sqrt{\log\!\Big(\frac{N}{\delta}\Big)}\,
\log N.
\]

It remains to verify the sample-size condition~\eqref{eq:sample_size_condition} under this choice of $J$. Since $J\asymp \frac{\log_2 N}{(2+\rho)(t+t'-r)} $, we have
\[
N=2^{(2+\rho)J(t+t'-r)}\gtrsim 2^{\widetilde{J}n(\widetilde{t}+\widetilde{t}'-r)/(\sigma-n/2+\widetilde{t}'-r/2)},
\]
provided
\[
(2+\rho)J(t+t'-r)>\widetilde{J}n\frac{\widetilde{t}+\widetilde{t}'-r}{\sigma-n/2+\widetilde{t}'-r/2}.
\]
By the definition~\eqref{eq:def_parameter} of $\widetilde{J},\widetilde{t},\widetilde{t}'$,
this is equivalent to
\[
(2+\rho)J(t+t'-r)> \frac{t+t'-r+\varepsilon_1}{\min\{t',r_1\}+t-r}\, Jn\, \frac{t'+\max\{t',r_1\}-r}{\sigma-n/2+\max\{t',r_1\}-r/2},
\]
where $\varepsilon_1=n(t+t'-r)/(\sigma-n/2+t-r/2)$. Since $\rho\ge 0$, it suffices that
\[
2>  \frac{n+\frac{n^2}{\sigma-n/2+t-r/2} }{\min\{t',r_1\}+t-r}\, \cdot \frac{t'+\max\{t',r_1\}-r}{\sigma-n/2+\max\{t',r_1\}-r/2}.
\]
Noting that $\max\{t',r_1\}\ge r_1>n/2+\max\{0,r\}\ge n/2+r/2$, we see that a sufficient condition for the above inequality is
\[
\sigma >
\max\left\{
\frac{3n}{2} - t + \frac{r}{2},\ \frac{t'+\max\{t',r_1\}-r}{\min\{t',r_1\}+t-r}n
\right\}.
\]
This is guaranteed by Assumption~\ref{assumption:wavelets}~(v), and thus the sample-size condition~\eqref{eq:sample_size_condition} holds for all sufficiently large $N$. The claimed bound follows.
\end{proof}

In the proof of Theorem~\ref{thm:main1}, we invoked Lemma~\ref{lemma:error_decomposition} to derive the error decomposition, and Propositions~\ref{prop:Error-Truncation}, \ref{prop:Error-Compression}, \ref{prop:Error-OVB}, and \ref{prop:Error-Var} to bound, respectively, the truncation error, the compression error, the omitted-variable bias, and the variance term. The proofs of these results are provided in Appendices~\ref{app:truncation}, \ref{app:compression}, and~\ref{app:estimation}.

\smallskip

\begin{lemma}[Decomposition of  $\ErrorEstimation$]\label{lemma:error_decomposition}
    Under the setting and assumptions of Theorem \ref{thm:main1}, the estimation error $\| \bfD^{-t'}_{\Lambda_J} (\widehat{\bfA} - \bfA_{\Lambda_J}^{\varepsilon}) \bfD^{-t}_{\Lambda_J} \|$ admits the decomposition:
\begin{align*}
\| \bfD^{-t'}_{\Lambda_J} (\widehat{\bfA} - \bfA_{\Lambda_J}^{\varepsilon}) \bfD^{-t}_{\Lambda_J} \|\le \ErrorOVB+ \ErrorVar,
\end{align*}
where the omitted-variable bias and variance contributions are 
\begin{align*}
   \ErrorOVB &:=\left\| \bfD^{-t'}_{\Lambda_J} \left((\bfM_{(J,t,t')})_{\mathrm{up}}\odot \left(\OVB_{\Lambda_J,\cdot}\right)\right) \bfD^{-t}_{\Lambda_J} \right\| + \left\| \bfD^{-t}_{\Lambda_J} \left(((\bfM_{(J,t,t')})_{\mathrm{low}})^{\top}\odot \left(\OVB_{\Lambda_J,\cdot}\right)\right) \bfD^{-t'}_{\Lambda_J} \right\|, \\ 
   \ErrorVar &:=\left\| \bfD^{-t'}_{\Lambda_J} \left((\bfM_{(J,t,t')})_{\mathrm{up}}\odot \left(\Var_{\Lambda_J,\cdot}\right)\right) \bfD^{-t}_{\Lambda_J} \right\| + \left\| \bfD^{-t}_{\Lambda_J} \left(((\bfM_{(J,t,t')})_{\mathrm{low}})^{\top}\odot \left(\Var_{\Lambda_J,\cdot}\right)\right) \bfD^{-t'}_{\Lambda_J} \right\|,
\end{align*}
with
\begin{align*}
    \OVB&:=\sum_{\lambda'\in \Lambda_J} \left(E_{\Omega_{\lambda'}}(\bfU_{\cdot,\Omega_{\lambda'}}^{\top} \bfU_{\cdot,\Omega_{\lambda'}})^{-1}
\bfU_{\cdot,\Omega_{\lambda'}}^{\top}\bfU_{\cdot,\Omega_{\lambda'}^c} \bf\bfA_{\Omega^c_{\lambda'},\lambda'}\right) e_{\lambda'}^{\top}\in \R^{\Lambda_{\widetilde{J}}\,\times \Lambda_J }, \\
\Var &:=\sum_{\lambda'\in \Lambda_J} \left(E_{\Omega_{\lambda'}}(\bfU_{\cdot,\Omega_{\lambda'}}^{\top} \bfU_{\cdot,\Omega_{\lambda'}})^{-1} \bfU_{\cdot,\Omega_{\lambda'}}^{\top}\bfW_{\cdot,\lambda'}\right) e_{\lambda'}^{\top}\in \R^{\Lambda_{\widetilde{J}}\,\times \Lambda_J }.
\end{align*}
Here $\OVB_{\Lambda_J,\cdot}$ and $\Var_{\Lambda_J,\cdot}$ denote the restrictions of $\OVB$ and $\Var$ to rows indexed by $\Lambda_J$, and $\bfM_{(J,t,t')}$ is the indicator matrix associated with $\supp(J,t,t')$ defined in~\eqref{eq:mask_Jtt'}.
\end{lemma}

\smallskip

\begin{proposition}[Bound for $\ErrorTruncation$]
\label{prop:Error-Truncation}
Suppose Assumption \ref{assumption:operator_data_noise} (i) and Assumption \ref{assumption:wavelets} hold. For $t,t'\ge r/2$,
\[
\ErrorTruncation=\| \bfD^{-t'} (\bfA_{\Lambda_J}^{\uparrow} - \bfA) \bfD^{-t} \|\lesssim 2^{-J(t+t'-r)}.
\]
\end{proposition}

\smallskip

\begin{proposition}[Bound for $\ErrorCompression$]\label{prop:Error-Compression}
Suppose Assumption 
\ref{assumption:operator_data_noise} (i) and Assumption \ref{assumption:wavelets} hold. For $\bfA^{\varepsilon}_{\Lambda_J}$ defined in Definition \ref{def:supp_Jtt'} and $t_{\ell},t_{r}\in [r/2,\gamma)$,
\[
\|\bfD^{-t_l}_{\Lambda_{J}}(\bfA^{\varepsilon}_{\Lambda_J}-\bfA_{\Lambda_J})\bfD^{-t_r}_{\Lambda_{J}}\|\lesssim  J 2^{-J\,\left(\min\{t',t_{l}\}+\min\{t,t_{r}\}-r\right)}.
\]
In particular,  for $t,t'\ge r/2$,
\[
\ErrorCompression=\|\bfD^{-t'}_{\Lambda_{J}}(\bfA^{\varepsilon}_{\Lambda_J}-\bfA_{\Lambda_J})\bfD^{-t}_{\Lambda_{J}}\|\lesssim  J\, 2^{-J(t+t'-r)}.
\]
\end{proposition}

\smallskip

\begin{proposition}[Bound for $\ErrorOVB$]\label{prop:Error-OVB}
Under the setting and assumptions of Theorem~\ref{thm:main1}, the following holds. For any $\delta\in (0,1)$, if
\[
N\gtrsim 2^{\widetilde{J}n(\widetilde{t}+\widetilde{t}'-r)/(\sigma-n/2+\widetilde{t}'-r/2)}+\log (1/\delta),
\]
then with probability at least $1-\delta$,
\begin{align*}
\ErrorOVB\lesssim J\, 2^{-J(t+t'-r)}.
\end{align*}
\end{proposition}

\begin{proposition}[Bound for $\ErrorVar$]\label{prop:Error-Var}
Under the setting and assumptions of Theorem~\ref{thm:main1}, the following holds. For any $\delta\in (0,1)$, if
\[
N\gtrsim 2^{\widetilde{J}n(\widetilde{t}+\widetilde{t}'-r)/(\sigma-n/2+\widetilde{t}'-r/2)}+\log (1/\delta),
\]
then with probability at least $1-\delta$,
\[
\ErrorVar\lesssim \sqrt{\frac{J+\log(1/\delta)}{N}}\cdot J\cdot 2^{\rho J(t+t'-r)/2},
\]
where $\rho$ is the exponent defined in \eqref{eq:rho}.
\end{proposition}

\section{Convergence Rates for Data-Driven PDE Solver}\label{sec:stability}

In the classical wavelet--Galerkin framework, an effective matrix compression strategy is typically expected to preserve the optimal convergence order of the underlying Galerkin scheme. In that setting, no data are involved and the operator $\mcA$ is fully known. In contrast, in our setting the operator $\mcA$ is unknown and must be learned from data. This raises the question of whether the learned operator $\widehat{\mcA}$ provides an effective PDE solver for previously unseen right-hand side $f.$

In this section, we study the convergence rate of the numerical solution obtained by plugging the learned sparse operator $\widehat{\mcA}$ into the elliptic pseudo-differential equation. 
 We show that this rate inherits the statistical accuracy of $\widehat{\mcA}$ and is analogous to that of the classical wavelet--Galerkin method, while exhibiting new features specific to the data-driven setting ---most notably that the maximal wavelet scale is chosen as a function of the sample size in the learning stage. This analysis highlights the usefulness of operator learning for downstream PDE solvers.

Subsection~\ref{subsec:stability_error_bound} presents the error bound and Subsection~\ref{subsec:stability_proof} contains its proof.

\subsection{Second Main Result}\label{subsec:stability_error_bound}

Before presenting the result, we introduce a slight modification of the estimator in Theorem \ref{thm:main1} so that the resulting matrix estimator $\widehat{\bfA}$ satisfies the $V_{J}$ ellipticity condition, which will be crucial in establishing the convergence rate of the numerical solution; see \cite{dahmen2006compression}.

Recall that in the proof of Theorem \ref{thm:main1}, combining Proposition \ref{prop:Error-OVB}, \ref{prop:Error-Var}, and \ref{prop:Error-Compression}, we established that the estimator $\widehat{\bfA}\in \R^{\Lambda_J\times \Lambda_J}$ satisfies the following bound. Let $\delta\in (0,1)$. If
\[
N\gtrsim 2^{\widetilde{J}n(\widetilde{t}+\widetilde{t}'-r)/(\sigma-n/2+\widetilde{t}'-r/2)}+\log (1/\delta),
\]
then with probability at least $1-\delta$,
\begin{align*}
\|\bfD^{-t'}_{\Lambda_J} ( \widehat{\bfA}-\bfA_{\Lambda_J}) \bfD^{-t}_{\Lambda_J}  \| &\le \ErrorOVB+\ErrorVar+\ErrorCompression\\
&\lesssim J\, 2^{-J(t+t'-r)}+\sqrt{\frac{J+\log(1/\delta)}{N}}\cdot J\cdot 2^{\rho J(t+t'-r)/2},
\end{align*}
where $\rho$ is the exponent defined in \eqref{eq:rho}.

We now claim that, after slightly modifying the estimator in Theorem \ref{thm:main1} and adapting the analysis, the following refined bound can be attained: there exists some constant $\kappa>0$ such that, with probability at least $1-\delta$,
\begin{align}\label{eq:enhanced_estimate_0}
\|\bfD^{-t'}_{\Lambda_J} ( \widehat{\bfA}-\bfA_{\Lambda_J}) \bfD^{-t}_{\Lambda_J}\|  \lesssim \varepsilon 2^{-J(t+t'-r)}+\sqrt{\frac{\log(1/\delta)}{N}}\cdot (J/\varepsilon)^{\kappa}\cdot 2^{\rho J(t+t'-r)/2}.
\end{align}
Here $\varepsilon>0$ is a tunable hyperparameter, and $\rho$ is as in \eqref{eq:rho}.
Choosing $J=J_{*}$ such that
\begin{align}\label{eq:J_*}
\varepsilon 2^{-J_{*}(t+t'-r)}\asymp \sqrt{\frac{\log (1/\delta)}{N}}\cdot (J_{*}/\varepsilon)^{\kappa} \cdot 2^{\rho J_{*}(t+t'-r)/2},
\end{align}
we therefore conclude that, if $N\gtrsim \log (1/\delta)$, then with probability at least $1-\delta$,
\begin{align}\label{eq:enhanced_estimate}
\| \bfD^{-t'}_{\Lambda_{J_{*}}} (\widehat{\bfA} - \bfA_{\Lambda_{J_{*}}}) \bfD^{-t}_{\Lambda_{J_{*}}} \|
&\lesssim \varepsilon 2^{-J_{*}(t+t'-r)}+\sqrt{\frac{\log (1/\delta)}{N}}\cdot (J_{*}/\varepsilon)^{\kappa}\cdot 2^{\rho J_{*}(t+t'-r)/2} \nonumber\\
&\asymp \varepsilon 2^{-J_{*}(t+t'-r)} \nonumber\\
&\asymp_{\mathrm{polylog}} N^{-\frac{1}{2+\rho}}.
\end{align}
Here we use the notation $\asymp_{\mathrm{polylog}}$ 
to hide poly-logarithmic factors in $N$ and $1/\delta$.

\begin{remark}[Modification for the enhanced bound \eqref{eq:enhanced_estimate_0}]
We briefly explain how to modify the estimator $\widehat{\bfA}$ in order to obtain the bound \eqref{eq:enhanced_estimate_0}. 
We introduce a new support set $\supp^{\mathrm{new}}(J,t,t')$ by modifying both the thresholding parameters $\tau_{jj'}$ and the slope conditions in the definition of 
$\supp(J,t,t')$ in~\eqref{eq:supp_Jtt'} as follows:
\begin{align}\label{eq:supp_Jtt'_new}
\supp^{\mathrm{new}}(J,t,t') &:=\bigg\{(\lambda,\lambda')\in \Lambda_J\times \Lambda_J: \mathrm{dist}(S_{j,k},S_{j',k'})\le \tau^{\mathrm{new}}_{jj'},\nonumber\\
&\qquad\qquad j \le \frac{t+t'-r}{\sigma-\frac{n}{2}+t'-\frac{r}{2}}J+\frac{\sigma-\frac{n}{2}-(t-\frac{r}{2})}{\sigma-\frac{n}{2}+t'-\frac{r}{2}} j'+\frac{\log_2(J/\varepsilon)}{\sigma-\frac{n}{2}+t'-\frac{r}{2}},\nonumber\\
&\qquad\qquad j'\le \frac{t+t'-r}{\sigma-\frac{n}{2}+t-\frac{r}{2}}J+\frac{\sigma-\frac{n}{2}-(t'-\frac{r}{2})}{\sigma-\frac{n}{2}+t-\frac{r}{2}} j+\frac{\log_2(J/\varepsilon)}{\sigma-\frac{n}{2}+t-\frac{r}{2}}\bigg\},
\end{align}
where the new thresholding parameter is chosen as
\begin{align}\label{eq:thresholding_new}
\tau^{\mathrm{new}}_{jj'}\asymp \max \left\{2^{-\min\{j,j'\}}, 2^{\frac{J(t+t'-r)-jt'-j't-(j+j
')\widetilde{d}+\log_2(J/\varepsilon)}{2\widetilde{d}+r}}\right\}.
\end{align}
We then define $\widehat{\bfA}$ using this updated support $\supp^{\mathrm{new}}(J,t,t')$. We defer the proof of \eqref{eq:enhanced_estimate_0} for this modified estimator to
Appendix~\ref{app:stability}.
\end{remark}

The next theorem quantifies the convergence rate of the PDE solution obtained by replacing $\mcA$ with the learned sparse operator satisfying \eqref{eq:enhanced_estimate}.

\begin{theorem}\label{thm:main2}
Suppose Assumptions \ref{assumption:operator_data_noise} and \ref{assumption:wavelets} hold, and assume that the operator $\mcA$ is unknown. Let $f\in H^{-r/2}(\mcM)$, and let $u$ be the exact solution of $\mcA u=f$. Let $\widehat{\bfA}\in \R^{\Lambda_{J_{*}}\times \Lambda_{J_{*}}}$ be the data-driven sparse estimator satisfying $\eqref{eq:enhanced_estimate}$, and let $\widehat{\bfu}$ solve \[
\widehat{\bfA}\widehat{\bfu}=\widetilde{\bff}_{\Lambda_{J_*}},
\]
where $\widetilde{\bff}_{\Lambda_{J_*}} := \bigl(\langle f, \psi_{\lambda}\rangle\bigr)_{\lambda\in \Lambda_{J_*}}$ 
denotes the vector of wavelet coefficients of $f$ truncated to scales up to $J_{*}$. Define
\[
\widehat{u}:=\widehat{\bfu}^{\top} \Psi=\sum_{\lambda\in \Lambda_{J_{*}}} \widehat{u}_{\lambda}\psi_{\lambda}.
\]
Then, with probability at least $1-\delta$, for any $\alpha\in [0,t'-r/2]$,
\begin{align}\label{eq:main2}
\|\widehat{u}-u\|_{H^{r/2-\alpha}}&\lesssim 2^{-J_{*}(t+\alpha-r/2)}\|u\|_{H^t}\asymp_{\mathrm{polylog}} N^{-\frac{t+\alpha-r/2}{(2+\rho)(t+t'-r)}}\|u\|_{H^t},
\end{align}
where $\rho$ is the exponent defined in \eqref{eq:rho}.
\end{theorem}

\begin{remark}[Discussion on Theorem \ref{thm:main2}]

In Theorem \ref{thm:main2}, choosing $\alpha=t'-r/2$ yields the fastest convergence rate
\[
\|\widehat{u}-u\|_{H^{r-t'}}\lesssim 2^{-J_{*}(t+t'-r)}\|u\|_{H^t}\asymp_{\mathrm{polylog}} N^{-\frac{1}{2+\rho}}\|u\|_{H^t}.
\]
This matches the convergence rate for learning the unknown operator $\mcA$ established in Theorem \ref{thm:main1}. Thus, the data-driven PDE solver inherits the statistical accuracy of the estimator $\widehat{\mcA}$.

It is instructive to compare this with the classical deterministic theory of compressed wavelet Galerkin schemes \cite{dahmen2006compression}. When the operator $\mcA$ (or its wavelet discretization $\bfA$) is known exactly, \cite[Theorems 10.2–10.3]{dahmen2006compression} show that the proposed matrix-compression strategy preserves the optimal convergence rate of the Galerkin method: for $J>j_0$, if $u_J$ solves the compressed system, then for any $\alpha\in[0,d-r/2]$,
\begin{align}\label{eq:galerkin}
\|u_J-u\|_{H^{r/2-\alpha}}\lesssim 2^{-J(d+\alpha-r/2)}\|u\|_{H^d}, 
\end{align}
where $d$ is the approximation order of the primal wavelet basis $\Psi$; see also \cite[Lemma 3.1]{dahmen2006compression}. This demonstrates that, in the deterministic setting, compression reduces computational cost by sparsifying $\bfA$ while preserving the approximation accuracy of the underlying Galerkin scheme.

Our setting differs from this classical framework in several crucial ways:

\begin{enumerate}
    \item \textbf{Unknown operator.} The operator $\mcA$ is not available and must be learned from data. The estimator $\widehat{\mcA}$ is constructed to control the error $\|\widehat{\mcA}-\mcA\|_{H^t\to H^{-t'}}$. Consequently, the numerical solution $\widehat{u}$ reflects not only discretization and compression errors but also statistical errors (omitted-variable bias and variance). The resulting convergence rate therefore depends on the sample size $N$  used in the learning stage, which governs the accuracy with which $\widehat{\mcA}$ approximates $\mcA$.
    \item \textbf{Fixed regularity parameters.} Observe that \eqref{eq:main2} has a structural form similar to \eqref{eq:galerkin}. In our setting, the indices $t,t'$ are fixed and determine the metric in which $\mcA$ is learned, thereby directly controlling the attainable accuracy of the PDE solution. In contrast, the deterministic bound \eqref{eq:galerkin} depends on the wavelet approximation order $d$, whereas in the data-driven setting the relevant range of Sobolev regularity scales are prescribed by the operator-learning problem through the choice of $t,t'$.
    \item \textbf{Data-dependent resolution level.} The maximal wavelet scale $J_*$ is chosen as a function of the sample size $N$, balancing truncation, compression and  statistical estimation errors. This stands in stark contrast with the classical Galerkin scheme, where $J$ is a purely numerical refinement parameter independent of data.
\end{enumerate}

In summary, Theorem \ref{thm:main2} shows that ---even in the data-driven setting where the operator $\mcA$ must be learned--- the resulting PDE solver attains a convergence rate that mirrors the nearly optimal statistical rate of operator learning. Moreover, the bound \eqref{eq:main2} exhibits a structural form analogous to the classical approximation result \eqref{eq:galerkin}, thereby extending the classical wavelet--Galerkin methods with matrix compression to a statistical setting.
\end{remark}

\subsection{Proof of Theorem \ref{thm:main2}}\label{subsec:stability_proof}

\begin{proof}[Proof of Theorem \ref{thm:main2}]

The proof proceeds in three steps. In \textbf{Step 1}, we verify that the estimator $\widehat{\mcA}$ is $V_{J_{*}}$-elliptic. Given this, in \textbf{Step 2} we prove the claim for $\alpha=0$ by applying Strang's first lemma \cite{ciarlet2002finite}. In \textbf{Step 3}, we treat the case $\alpha\in (0,t'-r/2]$ using the $\alpha=0$ estimate established in \textbf{Step 2}.

\medskip

\noindent\textbf{Step 1: $V_{J_{*}}$ ellipticity of estimator $\widehat{\mcA}$.}

\smallskip

Given $\widehat{\bfA}\in \R^{\Lambda_{J_{*}}\times \Lambda_{J_{*}}}$, we define the estimator $\widehat{\mcA}$ as
\begin{align}\label{eq:enhanced_estimator}
\widehat{\mcA}
:= \sum_{(\lambda,\lambda') \in \Lambda_{J_{*}}\times \Lambda_{J_{*}} }
\widehat{\bfA}_{\lambda,\lambda'}\,
\widetilde{\psi}_{\lambda} \otimes \widetilde{\psi}_{\lambda'}.
\end{align}

It follows from \eqref{eq:enhanced_estimate} that
\begin{align*}
\| \bfD_{\Lambda_{J_{*}}}^{-r/2}(\widehat{\bfA}-\bfA_{\Lambda_{J_{*}}}) \bfD_{\Lambda_{J_{*}}}^{-r/2} \| &=\| \bfD_{\Lambda_{J_{*}}}^{t-r/2}\,\bfD_{\Lambda_{J_{*}}}^{-t}(\widehat{\bfA}-\bfA_{\Lambda_{J_{*}}}) \bfD_{\Lambda_{J_{*}}}^{-t} \,\bfD_{\Lambda_{J_{*}}}^{t-r/2}\|\\
&\le \|\bfD_{\Lambda_{J_{*}}}^{t-r/2}\| \cdot \| \bfD_{\Lambda_{J_{*}}}^{-t}(\widehat{\bfA}-\bfA_{\Lambda_J}) \bfD_{\Lambda_{J_{*}}}^{-t} \|\cdot  \|\bfD_{\Lambda_{J_{*}}}^{t-r/2}\|\\
&\lesssim 2^{2J_{*}(t-r/2)}\cdot \varepsilon 2^{-J_{*}(2t-r)}=\varepsilon. 
\end{align*}
Moreover,
\[
\| \bfD_{\Lambda_{J_{*}}}^{-t}(\widehat{\bfA}-\bfA_{\Lambda_{J_{*}}}) \bfD_{\Lambda_{J_{*}}}^{-r/2} \|\le \| \bfD_{\Lambda_{J_{*}}}^{-t}(\widehat{\bfA}-\bfA_{\Lambda_{J_{*}}}) \bfD_{\Lambda_J}^{-t} \|\cdot \|\bfD_{\Lambda_{J_{*}}}^{t-r/2}\|  \lesssim \varepsilon 2^{-J_{*}(t-r/2)}.
\]

By Proposition \ref{prop:property_UAW} \ref{diagonal_preconditioning} (i), there exists $c_{-}>0$ such that
$
\sigma_{\min}(\bfD^{-r/2} \bfA \bfD^{-r/2})\ge c_{-}.
$
Since the spectrum of $\bfA_{\Lambda_{J_{*}}}$ is contained in that of $\bfA$, it follows that
\[
\sigma_{\min}( \bfD_{\Lambda_{J_{*}}}^{-r/2}\bfA_{\Lambda_{J_{*}}} \bfD_{\Lambda_{J_{*}}}^{-r/2} )\ge \sigma_{\min}(\bfD^{-r/2} \bfA \bfD^{-r/2})\ge c_{-}.
\]

Then, by choosing $\varepsilon>0$ sufficiently small, we obtain the $V_{J_{*}}$-ellipticity of $\widehat{\mcA}$ in~\eqref{eq:enhanced_estimator}: for any $u_{J_{*}}\in V_{J_{*}}\subset H^{r/2}(\mcM)$,
\begin{align}\label{eq:ellipticity_Ahat}
&\langle \widehat{\mcA} u_{J_{*}},u_{J_{*}}\rangle=\langle \widehat{\bfA} \bfu_{J_{*}}, \bfu_{J_{*}} \rangle \nonumber\\
&=\langle \bfD_{\Lambda_{J_{*}}}^{-r/2}\bfA_{\Lambda_{J_{*}}}\bfD_{\Lambda_{J_{*}}}^{-r/2} \bfD_{\Lambda_{J_{*}}}^{r/2}\bfu_{J_{*}}, \bfD_{\Lambda_{J_{*}}}^{r/2}\bfu_{J_{*}} \rangle-\langle \bfD_{\Lambda_{J_{*}}}^{-r/2}(\bfA_{\Lambda_{J_{*}}}-\widehat{\bfA})\bfD_{\Lambda_{J_{*}}}^{-r/2} \bfD_{\Lambda_{J_{*}}}^{r/2}\bfu_{J_{*}}, \bfD_{\Lambda_{J_{*}}}^{r/2}\bfu_{J_{*}} \rangle \nonumber\\
&\ge \sigma_{\min}(\bfD_{\Lambda_{J_{*}}}^{-r/2}\bfA_{\Lambda_{J_{*}}} \bfD_{\Lambda_{J_{*}}}^{-r/2})\cdot \|\bfD_{\Lambda_{J_{*}}}^{r/2}\bfu_{J_{*}}\|^2-\| \bfD_{\Lambda_{J_{*}}}^{-r/2}(\widehat{\bfA}-\bfA_{\Lambda_{J_{*}}}) \bfD_{\Lambda_{J_{*}}}^{-r/2} \|\cdot \|\bfD_{\Lambda_{J_{*}}}^{r/2}\bfu_{J_{*}}\|^2 \nonumber\\
&\ge (c_{-}-\varepsilon) \|\bfD_{\Lambda_{J_{*}}}^{r/2}\bfu_{J_{*}}\|^2 \nonumber\\
&\asymp \|u_{J_{*}}\|_{H^{r/2}},
\end{align}
where the last step follows from Lemma \ref{lemma:wavelet_property} (iv) and $r/2\in (-\widetilde{\gamma},\gamma)$. 

\bigskip

\noindent\textbf{Step 2: $\alpha=0$.} We first prove the case $\alpha=0$. Using \eqref{eq:ellipticity_Ahat}, we apply Strang's first lemma \cite[Theorem 4.1.1]{ciarlet2002finite} to obtain that
\begin{align}\label{eq:stability_aux1}
\|\widehat{u}-u\|_{H^{r/2}}\lesssim \inf_{v_{J_{*}}\in V_{J_{*}}}\left\{\|v_{J_{*}}-u\|_{H^{r/2}}+\sup_{w_J\in V_{J_{*}},w_{J_{*}}\ne 0}\frac{|\langle (\widehat{\mcA}-\mcA)v_{J_{*}},w_{J_{*}} \rangle|}{\|w_{J_{*}}\|_{H^{r/2}}}\right\}.
\end{align}
By \eqref{eq:enhanced_estimate}, we have
\[
|\langle (\widehat{\mcA}-\mcA)v_{J_{*}},w_{J_{*}} \rangle|\lesssim \varepsilon 2^{-J_{*}(t+t'-r)} \|v_{J_{*}}\|_{H^t} \|w_{J_{*}}\|_{H^{t'}}.
\]
Moreover, Lemma \ref{lemma:wavelet_property} (ii) yields that, for $r/2<t'\le \gamma$,
\[
\|w_{J_{*}}\|_{H^{t'}}\lesssim 2^{J_{*}(t'-r/2)}\|w_{J_{*}}\|_{H^{r/2}}.
\]
Hence, taking $v_{J_{*}}=Q_{J_{*}} u\in V_{J_{*}}$ in \eqref{eq:stability_aux1} yields that
\begin{align*}
\|\widehat{u}-u\|_{H^{r/2}}&\lesssim \|u-Q_{J_{*}} u\|_{H^{r/2}}+\sup_{w_{J_{*}}\in V_{J_{*}}, w_{J_{*}}\ne 0}\frac{|\langle (\widehat{\mcA}-\mcA)Q_{J_{*}}u,w_{J_{*}} \rangle|}{\|w_{J_{*}}\|_{H^{r/2}}}\\
&\lesssim  \|u-Q_{J_{*}} u\|_{H^{r/2}}+2^{-J_{*}(t+t'-r)}\cdot 2^{J_{*}(t'-r/2)}\|Q_{J_{*}} u\|_{H^t}\\
&\lesssim 2^{-J_{*}(t-r/2)} \|u\|_{H^t},
\end{align*}
where the last step follows from $\|Q_{J_{*}} u\|_{H^t}\le \|u\|_{H^t}$ and $\|u-Q_{J_{*}} u\|_{H^{r/2}}\lesssim 2^{-J_{*}(t-r/2)} \|u\|_{H^t}$ by Lemma \ref{lemma:wavelet_property} (ii),  provided that $-\widetilde{d}\le r/2<t\le d$, \,$r/2<\gamma$, and $-\widetilde{\gamma}<t$.

\bigskip

\noindent\textbf{Step 3: $\alpha\in (0,t'-r/2]$.}
We begin with
\begin{align*}
    \|\widehat{u}-u\|_{H^{r/2-\alpha}}=\sup_{g\in H^{\alpha-r/2}} \frac{|\langle \widehat{u}-u,g \rangle|}{\|g\|_{H^{\alpha-r/2}}}.
\end{align*}
Since $\mcA:H^{\alpha+r/2}\to H^{\alpha-r/2}$ is an isomorphism, let $v\in H^{\alpha+r/2}$ satisfy $\mcA v=g$, then
\begin{align}\label{eq:stability_aux2}
\|\widehat{u}-u\|_{H^{r/2-\alpha}}=\sup_{v\in H^{\alpha+r/2}} \frac{|\langle \widehat{u}-u,\mcA v \rangle|}{\|\mcA v\|_{H^{\alpha-r/2}}}\asymp \sup_{v\in H^{\alpha+r/2}} \frac{|\langle \mcA(\widehat{u}-u),v \rangle|}{\|v\|_{H^{\alpha+r/2}}}.
\end{align}
Using the orthogonality $\langle \widehat{\mcA}\widehat{u},Q_{J_{*}} v \rangle=\langle Q_{J_{*}}f,Q_{J_{*}} v \rangle=\langle f, Q_{J_{*}} v \rangle=\langle \mcA u,Q_{J_{*}} v \rangle$, we can decompose
\begin{align}\label{eq:stability_aux3}
    \langle \mcA(\widehat{u}-u),v \rangle&=\langle \mcA(\widehat{u}-u),v-Q_{J_{*}} v \rangle+\langle \mcA(\widehat{u}-u),Q_{J_{*}} v \rangle \nonumber\\
    &=\langle \mcA(\widehat{u}-u),v-Q_{J_{*}} v \rangle+\langle (\mcA-\widehat{\mcA})\widehat{u},Q_{J_{*}} v \rangle.
\end{align}
The first term on the right-hand side of \eqref{eq:stability_aux3} is bounded by
\begin{align}\label{eq:stability_aux4}
|\langle \mcA(\widehat{u}-u),v-Q_{J_{*}} v \rangle|\lesssim \|\widehat{u}-u\|_{H^{r/2}}\|v-Q_{J_{*}} v \|_{H^{r/2}}\lesssim 2^{-J_{*}(t+\alpha-r/2)}\|u\|_{H^t}\|v\|_{H^{\alpha+r/2}},
\end{align}
where in the second inequality we used the bound $\|\widehat{u}-u\|_{H^{r/2}}\lesssim 2^{-J_{*}(t-r/2)}\|u\|_{H^t}$ obtained in \textbf{Step 2}, and $\|v-Q_{J_{*}} v\|_{H^{r/2}}\lesssim 2^{-J_{*}\alpha}\|v\|_{H^{\alpha+r/2}}$ by Lemma \ref{lemma:wavelet_property} (ii), provided that $-\widetilde{d}\le r/2\le \alpha+r/2\le d $,\, $r/2<\gamma$, and $-\widetilde{\gamma}<\alpha+r/2$. These conditions are guaranteed by $r/2\in (-\widetilde{\gamma},\gamma)$, $\alpha\in (0,t'-r/2]$, and $t'\le d$.

For the second term in \eqref{eq:stability_aux3},
\begin{align}\label{eq:stability_aux5}
&|\langle (\widehat{\mcA}-\mcA)\widehat{u},Q_{J_{*}} v \rangle| \nonumber\\
&\le |\langle (\widehat{\mcA}-\mcA)(\widehat{u}-Q_{J_{*}} u),Q_{J_{*}} v \rangle| +|\langle (\widehat{\mcA}-\mcA)Q_{J_{*}} u,Q_{J_{*}} v \rangle| \nonumber\\
&\le \|\widehat{\mcA}-\mcA\|_{H^{r/2}\to H^{-t'}} \|\widehat{u}-Q_{J_{*}} u\|_{H^{r/2}}\|Q_{J_{*}} v\|_{H^{t'}}+\|\widehat{\mcA}-\mcA\|_{H^{t}\to H^{-t'}} \|Q_{J_{*}} u\|_{H^{t}}\|Q_{J_{*}} v\|_{H^{t'}} \nonumber\\
&\lesssim 2^{-J_{*}(t'-r/2)}\cdot 2^{-J_{*}(t-r/2)} \cdot 2^{J_{*}(t'-\alpha-r/2)} \|u\|_{H^t}\|v\|_{H^{\alpha+r/2}} \nonumber\\
&\quad +2^{-J_{*}(t+t'-r)}\cdot 2^{J_{*}(t'-\alpha-r/2)} \|u\|_{H^t}\|v\|_{H^{\alpha+r/2}} \nonumber\\
&\asymp 2^{-J_{*}(t+\alpha-r/2)}\|u\|_{H^t}\|v\|_{H^{\alpha+r/2}},
\end{align}
where we used the following inequalities:
\[
 \|\widehat{\mcA}-\mcA\|_{H^{r/2}\to H^{-t'}} \lesssim 2^{-J_{*}(t'-r/2)},\qquad  \|\widehat{\mcA}-\mcA\|_{H^{t}\to H^{-t'}} \lesssim 2^{-J_{*}(t+t'-r)},
\]
\[
\|Q_{J_{*}} v\|_{H^{t'}}\lesssim 2^{J_{*}(t'-\alpha-r/2)}\|v\|_{H^{\alpha+r/2}},\quad \text{ if } \alpha+r/2\le t'\le \gamma, 
\]
and
\[
\|\widehat{u}-Q_{J_{*}} u\|_{H^{r/2}}\le \|\widehat{u}- u\|_{H^{r/2}}+\|u-Q_{J_{*}} u\|_{H^{r/2}}\lesssim 2^{-J_{*}(t-r/2)}\|u\|_{H^t}.
\]

Combining \eqref{eq:stability_aux2}, \eqref{eq:stability_aux3}, \eqref{eq:stability_aux4}, and \eqref{eq:stability_aux5} yields that, for any $\alpha\in (0,t'-r/2]$,
\begin{align*}
\|\widehat{u}-u\|_{H^{r/2-\alpha}}\asymp \sup_{v\in H^{\alpha+r/2}} \frac{|\langle \mcA(\widehat{u}-u),v \rangle|}{\|v\|_{H^{\alpha+r/2}}}\lesssim 2^{-J_{*}(t+\alpha-r/2)}\|u\|_{H^t}\asymp_{\mathrm{polylog}} N^{-\frac{t+\alpha-r/2}{(2+\rho)(t+t'-r)}}\|u\|_{H^t},
\end{align*}
where the last step follows from the choice of $J_*$ in \eqref{eq:J_*}. This completes the proof.
\end{proof}

\section{Conclusions, Discussion, and Future Directions}\label{sec:Conclusions}

This paper has established convergence rates for learning elliptic pseudo-differential operators from noisy (and noiseless) data. Within a wavelet--Galerkin framework, we formulated the learning task as a structured infinite-dimensional regression problem with multiscale sparsity. Building on this structure, we proposed a sparse, data- and computation-efficient estimator that combines a learning-oriented matrix compression scheme with a nested-support regression strategy to balance approximation and estimation errors. In addition to obtaining convergence rates for the estimator, we showed that the learned operator induces an efficient and stable Galerkin solver whose numerical error inherits its statistical accuracy. Our results therefore contribute to  bringing together operator learning, data-driven solvers, and wavelet methods in scientific computing.

We conclude with several open questions and future directions that arise from this work.

\paragraph{Ellipticity, data/noise distribution, computation, and error metric}

In this paper we assumed that $\mcA$ is strongly elliptic, whereas the wavelet-coordinate estimates in Proposition~\ref{prop:property_UAW} (I) hold more generally for PDOs beyond the elliptic class. It would be interesting to relax the ellipticity assumption and understand what weaker conditions suffice for the learning task. To simplify the exposition, we also focused on the setting where both the inputs and the noise are Gaussian random functions. Since our error analysis relies primarily on the covariance properties in Proposition~\ref{prop:property_UAW} (II) (ii)-(iii), it is natural to expect that both the estimation procedure and the resulting convergence guarantees extend beyond Gaussianity ---for instance, to settings where the input and noise distributions satisfy appropriate finite-moment (or sub-Gaussian/sub-exponential) assumptions. On the computational side, we noted at the end of Remark~\ref{remark:cost2} that it would be interesting to improve the runtime to an exact $\mathcal{O}(N2^{Jn})$ while preserving the same statistical rate. Finally, regarding the error metric, throughout the paper we measured the estimation accuracy in the operator norm $\|\cdot\|_{H^t\to H^{-t'}}$. One may also consider a prediction/generalization criterion under a testing distribution. For example, for a Gaussian test input $v\sim \mathcal{N}(0,\mathcal{C}_v)$, one can study
\[
\mathbb{E}_{v}\big\|(\widehat{\mathcal{A}}-\mathcal{A})v\big\|_{H^{-t'}}.
\]
Suppose that the wavelet representation $\widetilde{\bfC}_v$ of $\mathcal{C}_v$ satisfies an analogue of Proposition~\ref{prop:property_UAW} (II) (ii), namely, for some $t\in\mathbb{R}$,
\[
c_- \le \sigma_{\min} \big(\bfD^{t}\widetilde{\bfC}_v \bfD^{t}\big)
\le \sigma_{\max} \big(\bfD^{t}\widetilde{\bfC}_v \bfD^{t}\big)
\le c_+.
\]
Then a standard covariance computation yields
\begin{align*}
\mathbb{E}_{v}\big\|(\widehat{\mathcal{A}}-\mathcal{A})v\big\|_{H^{-t'}}^{2}
&= \mathrm{Tr} \Big( \bfD^{-t'}(\widehat{\bfA}-\bfA)\bfD^{-t}\big(\bfD^{t}\widetilde{\bfC}_v \bfD^{t}\big)\big(\bfD^{-t'}(\widehat{\bfA}-\bfA)\bfD^{-t}\big)^{\top}\Big)
\\
&\le \big\|\bfD^{t}\widetilde{\bfC}_v \bfD^{t}\big\|\,\|\bfD^{-t'}(\widehat{\bfA}-\bfA)\bfD^{-t}\|_{\mathrm{HS}}^{2}\\
& \lesssim \|\bfD^{-t'}(\widehat{\bfA}-\bfA)\bfD^{-t}\|_{\mathrm{HS}}^{2},
\end{align*}
where $\|\cdot\|_{\mathrm{HS}}$ denotes the Hilbert--Schmidt norm on operators $\ell^2(\mcJ) \to \ell^2(\mcJ)$. The multiscale techniques developed in this paper can be adapted to obtain convergence rates under such Hilbert--Schmidt-type criteria as well; we leave a detailed treatment to future work.

\paragraph{Adaptive operator estimation and learning}
An important avenue for future research is \emph{adaptive} operator estimation and learning. 
A guiding principle behind our learning methodology is the availability of wavelet-coordinate decay estimates for $\mcA$ (Proposition~\ref{prop:property_UAW} (I)), which inform learning-oriented compression and regression procedures that exploit multiscale sparsity and aim to estimate only the significant entries so as to balance approximation (bias) and estimation (variance) errors. 
Throughout this paper, we assume the order of $\mcA$ is known. 
It is therefore natural to ask whether one can develop methods that adapt to an unknown order, and hence to an unknown sparsity pattern. 
Concretely, instead of selecting a support via explicit thresholding parameters $\tau_{jj'}$ and slope conditions as in Definition~\ref{def:supp_Jtt'}, can one learn the effective support (i.e., the significant coordinates) from the data in an adaptive manner?

Related questions are classical in scientific computing and numerical analysis, where adaptive methods ---including adaptive wavelet schemes \cite{cohen2001adaptive,cohen2002adaptive,cohen2004adaptive,stevenson2009adaptive} and adaptive finite element methods \cite{morin2002convergence,bangerth2003adaptive,binev2004adaptive,stevenson2005optimal,bonito2024adaptive}--- play a central role in achieving optimal finite-term approximation and sparse representations of PDE solutions. 
In the context of operator learning, it is natural to explore analogous notions of adaptivity to the structure of the unknown operator, and to investigate suitable concepts of nonlinear approximation for operator classes.

From a statistical perspective, sparsity-inducing regularization and, in particular, the Lasso and its variants occupy a central place in modern high-dimensional statistics \cite{tibshirani1996regression,zou2006adaptive,buhlmann2011statistics,wainwright2019high,zhou2025thresholded}. 
It would be interesting to establish sparse oracle inequalities ---covering variable selection, support recovery, and estimation--- for operator learning problems such as those studied here.

\paragraph{Learning other structured operator classes}
This paper establishes convergence rates for learning elliptic PDOs in wavelet coordinates by leveraging \emph{a priori} structure that yields multiscale compressibility and enables sparsity-based estimation in an infinite-dimensional setting. The methodology and analysis developed here may extend to other operator classes with analogous structured representations. For instance, Fourier integral operators (FIOs) play a central role in wave propagation and hyperbolic PDEs \cite{hormander1971fourier,treves1980introduction,duistermaat1996fio,Hormander2009ALPDO,khoo2019switchnet,ying2022solving,wang2025operator}, and curvelets provide near-optimal sparse representations for broad families of FIOs \cite{CandesDonoho2000Curvelets,candes2004new,candes2003curvelets,candes2005curvelet,candes2006fast}, suggesting that an analogous multiscale, sparsity-based approach could be viable in suitable curvelet coordinates. Related structured operators also arise in parabolic problems \cite{von2003wavelet,schwab2009space,chegini2011adaptive}. More broadly, and more challengingly, one may aim to learn structured \emph{nonlinear} operators; a prototypical example is the parameter-to-solution map for linear, second-order, divergence-form elliptic PDEs \cite{cohen2015approximation,cohen2023near,reinhardt2024statistical}.

\paragraph{Bayesian formulations for operator learning}
We have mentioned that \cite{de2023convergence} establishes posterior contraction rates for learning linear operators that are diagonalizable in a known basis within a Bayesian framework.  More broadly, Bayesian formulations can quantify uncertainty in the learned operator and propagate it to downstream numerical predictions, yielding posterior credible sets for PDE solutions and related functionals. A key question is whether one can leverage the same multiscale compressibility that underpins our frequentist rates to design computationally tractable priors for structured operator classes (e.g., sparsity- or shrinkage-based priors in wavelet coordinates) and to establish corresponding posterior contraction guarantees. We believe that developing a general Bayesian framework for operator learning ---encompassing broader structured operator classes beyond the diagonalizable setting and providing contraction rates (and, ideally, frequentist coverage properties of credible sets)--- is a promising direction for future work.

\section*{Acknowledgments}
The authors were partly funded by the NSF CAREER award DMS-2237628. The authors thank Omar Al-Ghattas for helpful feedback and comments on the manuscript.

\bibliographystyle{siam} 
\bibliography{references}

\newpage

\appendix

\section{Auxiliary Materials Section \ref{sec:waveletformulation} }

\subsection{Background: Biorthogonal Wavelets and  Multiresolution Analysis}\label{appendix:wavelets}

Let $\{V_j\}_{j> j_0}$ be a sequence of nested, linear finite-dimensional subspaces $V_{j}\subset V_{j+1}\subset \cdots \subset L^2(\mcM)$. We say that the family $\{V_j\}_{j> j_0}$ has \emph{regularity} $\gamma>0$ and \emph{(approximation) order} $d\in\mathbb{N}$ if
\begin{align*}
   \gamma & = \sup\left\{s\in \R: V_j\subset H^s(\mcM), \ \forall \ j> j_0 \right\},  \\[0.5em]
   d &=\sup\left\{s\in \R: \inf_{v_j\in V_j} \|v-v_j\|_{L^2(\mcM)} \lesssim 2^{-js}\|v\|_{H^s(\mcM)}, \ \forall \ v\in H^s(\mcM), \ \forall\ j> j_0 \right\}.
\end{align*}
We shall suppose that the subspaces $\{V_j\}_{j> j_0}$ are $H^{r/2}(\mcM)$-conforming, i.e., we have $\gamma >\max\{0, r/2\}$ for some fixed order $r\in \R$.

We furthermore assume that $\mathrm{dim}(V_j)=O(2^{nj})$ and that, for each $j> j_0$, the space $V_j$ is spanned by a single-scale basis $\Phi_j$, i.e.,
\[
V_j =\mathrm{span}\, \Phi_j, \quad \mathrm{where } \ \, \Phi_j:=\{\phi_{j,k}: k\in \Delta_j\}, \quad \forall \ j> j_0.
\]
Here, the index set $\Delta_j$ describes the spatial localization of elements in $\Phi_j$. Analogously
to the spaces $\{V_j\}_{j> j_0}$, we assume without loss of generality that the finite index sets $\{\Delta_j\}_{j> j_0}$ are nested, $\Delta_j\subset \Delta_{j+1}\subset\cdots$. For each $j> j_0$, we associate with $\Phi_j$ the \emph{dual single-scale basis} defined by
\[
\widetilde{\Phi}_j:=\{\widetilde{\phi}_{j,k}: k\in \Delta_j\}, \quad \mathrm{with } \ \, \langle \phi_{j,k},\widetilde{\phi}_{j,k'} \rangle=\delta_{k,k'}, \ \forall \ k,k'\in \Delta_j.
\]
The vector spaces $\widetilde{V}_j:=\mathrm{span} \,\widetilde{\Phi}_j,\ j> j_0$, are also nested, $\widetilde{V}_{j}\subset \widetilde{V}_{j+1}\subset\cdots \subset L^2(\mcM)$, and the family $\{\widetilde{V}_j\}_{j> j_0}$ provides regularity $\widetilde{\gamma}>0$ and approximation order $\widetilde{d}$. For example, let the primal bases $\Phi_j$ be generated by tensor products of univariate $\mathrm{B}$-splines of order $d$, and let the dual bases $\widetilde{\Phi}_j$ be of order $\widetilde d\ge d$, with $d+\widetilde d$ even. Then $V_j =\mathrm{span}\, \Phi_j$ and $\widetilde{V}_j=\mathrm{span} \,\widetilde{\Phi}_j$ have approximation orders $d$ and $\widetilde d$, respectively.  Moreover, the corresponding regularity indices satisfy $\gamma=d-1/2$, while $\widetilde\gamma>0$ grows proportionally with $\widetilde d$. We refer to \cite{dahmen1999biorthogonal,dahmen1999wavelets,nguyen2003finite,nguyen2009finite,rekatsinas2018quadratic} for detailed constructions. 

In view of the biorthogonality of $\Phi_j, \widetilde{\Phi}_j$, we define the canonical projectors
\[
Q_j v:=\sum_{k\in \Delta_j} \langle v, \widetilde{\phi}_{j,k}\rangle \phi_{j,k},\qquad Q^{*}_j v:=\sum_{k\in \Delta_j} \langle v,\phi_{j,k} \rangle \widetilde{\phi}_{j,k},\quad \forall \ v\in L^2(\mcM),
\]
associated with the multiresolution sequences $\{V_j\}_{j> j_0}, \{\widetilde{V}_j\}_{j> j_0}$.  Moreover, the $L^2(\mcM)$-boundedness of $Q_j$ implies the Jackson and Bernstein inequalities; see Lemma \ref{lemma:wavelet_property} (ii) below.

 To define \emph{multiresolution analyses} (MRA), we start by introducing index sets $\nabla_j := \Delta_{j+1}\backslash \Delta_j, \, j>j_0$. Given single-scale bases $\Phi_j$ and $\widetilde{\Phi}_j$, one can construct \emph{biorthogonal complement bases}  
\[
\Psi_j=\{\psi_{j,k}:k\in \nabla_j\} \quad \mathrm{and} \quad \widetilde{\Psi}_j=\{\widetilde{\psi}_{j,k}:k\in\nabla_j\},\qquad j> j_0,
\]
satisfying the \emph{biorthogonality relation}
\[
\langle \psi_{j,k},\widetilde{\psi}_{j,k} \rangle=\delta_{(j,k),(j',k')}=\begin{cases}
  1  &  \text{ if } j=j' \text{ and } k=k', \\
  0  &  \text{ otherwise},
\end{cases}
\]
such that
\[
\mathrm{diam}\bigl(\supp(\psi_{j,k})\bigr) \asymp  2^{-j}, \quad j > j_0.
\]
We refer to \cite{nguyen2003finite,nguyen2009finite,rekatsinas2018quadratic} for particular constructions. In addition, we use the convention $\Psi_{j_0}:=\Phi_{j_0+1}, \widetilde{\Psi}_{j_0}:=\widetilde{\Phi}_{j_0+1}$, and $\nabla_{j_0}:=\Delta_{j_0+1}$.

 For $j> j_0$, define $W_j:=\mathrm{span} \, \Psi_{j}$ and $\widetilde{W}_j:=\mathrm{span}\, \widetilde{\Psi}_j$. The biorthogonality implies that, for all $j> j_0$,
\[
V_{j+1}=W_j \oplus V_j,\quad \widetilde{V}_{j+1}=\widetilde{W}_{j}\oplus \widetilde{V}_{j},\quad \widetilde{V}_j \perp W_j,\quad V_j\perp \widetilde{W}_j.
\]
Hence, $V_J$ and $\widetilde{V}_{J}$ can be written as a direct sum of the complement spaces $W_j$, respectively, $\widetilde{W}_j$, $j_0\le j<J$, using the convention $W_{j_0}:=V_{j_0+1}$ and $ \widetilde{W}_{j_0}:=\widetilde{V}_{j_0+1}$. With the convention $Q_{j_0}=Q^{*}_{j_0} :=0$, one has for $v_{J}\in V_{J}$ and for $ \widetilde{v}_J\in \widetilde{V}_J$ that
\[
v_J=\sum_{j=j_0}^{J-1} (Q_{j+1}-Q_{j})v_{J},\qquad \widetilde{v}_J=\sum_{j=j_0}^{J-1} (Q^{*}_{j+1}-Q^{*}_{j})\widetilde{v}_{J},
\]
where
\[
(Q_{j+1}-Q_j)v =\sum_{k\in \nabla_j}\langle v,\widetilde{\psi}_{j,k} \rangle \psi_{j,k},\qquad (Q^{*}_{j+1}-Q^{*}_j)v =\sum_{k\in \nabla_j}\langle v,\psi_{j,k} \rangle \widetilde{\psi}_{j,k}.
\]
From this observation, a biorthogonal dual pair of wavelet bases is now obtained  from the union of the coarse single-scale basis and the complement bases,
\[
\Psi = \bigcup_{j\ge j_0} \Psi_{j},\qquad \widetilde{\Psi} = \bigcup_{j\ge j_0} \widetilde{\Psi}_j.
\]
 We refer to $\Psi$ and $\widetilde{\Psi}$ as the \emph{primal} and \emph{dual} MRAs, respectively. Here and throughout, all basis functions in $\Psi$ and $\widetilde{\Psi}$ are assumed to be normalized in $L^2(\mcM)$.

The next lemma collects several key properties of the biorthogonal wavelet
system $(\Psi,\widetilde{\Psi})$ that will be used throughout our operator
learning analysis; see 
\cite{dahmen1997wavelet, dahmen1999biorthogonal, cohen2001adaptive, harbrecht2024multilevel}.

\begin{lemma}[Properties of biorthogonal wavelets]\label{lemma:wavelet_property}
Let $(\Psi,\widetilde{\Psi})$ be a biorthogonal wavelet system on $\mcM$
with parameters $(\gamma,\widetilde{\gamma},d,\widetilde d)$. Then the following hold.
  \begin{enumerate}[label=(\roman*)]
\item \textbf{Cardinality and locality.} For each $j \ge j_0$, the number of wavelets at level $j$ satisfies $|\nabla_j| \asymp 2^{jn}$. Moreover, the supports are localized in the sense that $\mathrm{diam}(\mathrm{supp}\,\psi_{j,k}) \asymp 2^{-j},$ and there exists a constant $M$ such that, for every $k\in \nabla_j$, at most $M$ indices $k'\in\nabla_j$ satisfy  $\mathrm{meas}\big(\mathrm{supp}(\psi_{j,k})\cap\mathrm{supp}(\psi_{j,k'})\big)\ne 0.$
\item \textbf{Approximation and stability.} Let $Q_j$ denote the canonical projector onto the approximation space $V_j$. We take $Q_{j_0}=0$ and, for $j>j_0$,
 \begin{align*}
Q_j v = \sum_{j'=j_0}^{j-1}\sum_{k\in \nabla_{j'}}\langle v,\widetilde{\psi}_{j',k} \rangle \psi_{j',k},\quad  \forall \ v\in L^2(\mcM).
\end{align*}
The Jackson and Bernstein inequalities hold: 
\[
\|v- Q_j v\|_{H^s(\mcM)}\lesssim 2^{-j(t-s)}\|v\|_{H^t(\mcM)},\quad \forall \ v\in H^t(\mcM),
\]
for all $-\widetilde{d}\le s \le t \le d,s<\gamma,-\widetilde{\gamma}<t$, and
\[
\|Q_j v\|_{H^{s}(\mcM)}\lesssim 2^{j(s-t)} \|Q_j v\|_{H^{t}(\mcM)}, \quad \forall \ v \in  H^{t}(\mcM),
\]
for all $t\le s\le \gamma$.
    \item \textbf{Vanishing moments.}  For any sufficiently smooth function $v$, \begin{align*}
|\langle v,\psi_{j,k}\rangle|\lesssim 2^{-j(\widetilde{d}+n/2)}|v|_{W^{\widetilde{d},\infty}(\mathrm{supp}(\psi_{j,k}))}, \quad |v|_{W^{\widetilde{d},\infty}(\Omega)}:=\sup_{|\alpha|=\widetilde{d},x\in \Omega}|\partial^{\alpha}v(x)|.
\end{align*} 
\item \textbf{Norm equivalence.}  For any $v \in H^t(\mcM)$, the following equivalences hold:
\begin{align*}
\|v\|_{H^t(\mcM)}^2 & \asymp \sum_{j \geq j_0} \sum_{k \in \nabla_j} 2^{2 j t}|\langle v, \widetilde{\psi}_{j, k}\rangle|^2, \quad t \in(-\widetilde{\gamma}, \gamma), \\
\|v\|_{H^t(\mcM)}^2 & \asymp \sum_{j \geq j_0} \sum_{k \in \nabla_j} 2^{2 j t}\left|\left\langle v, \psi_{j, k}\right\rangle\right|^2, \quad t \in(-\gamma, \widetilde{\gamma}).
\end{align*}   
\end{enumerate}
\end{lemma}

\subsection{Properties of \texorpdfstring{$\bfA,\bfU,\bfW$}{A, U, W}}\label{app:proof_property_UAW}

\begin{proof}[Proof of Proposition \ref{prop:property_UAW}]
 \begin{enumerate}[label=(\Roman*)]
    \item \textbf{Approximate sparsity:} The estimate in (i) is well known in the literature \cite{dahmen1997wavelet,dahmen2006compression, schneider2013multiskalen,harbrecht2024multilevel}: it shows that the matrix entries decay rapidly when the supports of the testing wavelets are well separated. For the reader’s convenience, we include a brief and formal derivation.  

Let $k_{\mcA}(x,y)=(2\pi)^{-n}\int_{\R^n}e^{i(x-y)\cdot \xi} a(x,\xi)d\xi$ be the (distributional) Schwartz kernel of $\mcA\in \OPS^{r}$. When $\mathrm{dist}(S_{j,k},S_{j',k'})\gtrsim 2^{-\min\{j,j'\}}$, the product $\psi_{j,k}(x)\psi_{j',k'}(y)$ is supported away from the diagonal $x=y$, hence $k_{\mcA}$ is smooth on that set and Fubini's theorem applies. Therefore,
\[
\langle \mcA \psi_{j',k'},\psi_{j,k} \rangle=\iint k_{\mcA}(x,y)\psi_{j,k}(x)\psi_{j',k'}(y)dxdy= \int\left(\int k_{\mcA}(x,y)\psi_{j',k'}(y)dy\right)\psi_{j,k}(x)dx.
\]
By Lemma \ref{lemma:wavelet_property} (iii) (vanishing moments up to order $\widetilde{d}$),
\begin{align*}
    |\langle \mcA \psi_{j',k'},\psi_{j,k} \rangle|
   \lesssim 2^{-j(\widetilde{d}+n/2)} \sup_{|\alpha|=\widetilde{d},x\in \supp(\psi_{j,k})} \left|\int \partial_{x}^{\alpha}k_{\mcA}(x,y)\psi_{j',k'}(y)dy\right|.
\end{align*}   
Applying the same argument in the $y$–variable yields
\begin{align*}
   |\langle \mcA \psi_{j',k'},\psi_{j,k} \rangle|  \lesssim 2^{-(j+j')(\widetilde{d}+n/2)}\sup_{x\in \supp(\psi_{j,k}), y\in \supp(\psi_{j',k'}),|\alpha|=|\beta|=\widetilde{d}} |\partial_x^{\alpha}\partial_y^{\beta}k_{\mcA}(x,y)|.
\end{align*}
For $\mcA\in \OPS^{r}$, the kernel satisfies the standard off–diagonal bounds (see \cite{harbrecht2024multilevel,dahmen1997wavelet,dahmen2006compression,schneider2013multiskalen})
\[
|\partial_{x}^{\alpha}\partial_y^{\beta}k_{\mcA}(x,y)|\lesssim_{\alpha,\beta}\mathrm{dist}(x,y)^{-(n+r+|\alpha|+|\beta|)},\quad x\ne y,
\]
hence
\begin{align*}
   |\langle \mcA \psi_{j',k'},\psi_{j,k} \rangle| &\lesssim  2^{-(j+j')(\widetilde{d}+n/2)} \sup_{x\in \supp(\psi_{j,k}), y\in \supp(\psi_{j',k'}),|\alpha|=|\beta|=\widetilde{d}}\mathrm{dist}(x,y)^{-(n+r+|\alpha|+|\beta|)}\\
    &= 2^{-(j+j')(\widetilde{d}+n/2)}\mathrm{dist}(S_{j,k},S_{j',k'})^{-(n+r+2\widetilde{d})}.
\end{align*}
We refer the reader to \cite[Proposition 2]{harbrecht2024multilevel} and \cite{dahmen1997wavelet,schneider2013multiskalen} for further background.

\medskip

The estimate in (ii) shows that matrix entries corresponding to wavelets with overlapping supports decay as the  difference in scales increases; see \cite[Section 9.4]{dahmen1997wavelet}, \cite[(2.28)]{cohen2001adaptive}, and \cite[Remark, p.2261]{dahmen2006compression}. We provide a brief derivation on this estimate following \cite[Section 9.4]{dahmen1997wavelet}. 

Without loss of generality, assume $j\ge j'$. Under Assumption \ref{assumption:operator_data_noise} (i), the operator $\mcA\in \OPS^{r}(\mcM)$ acts continuously and bijectively from $H^{s}(\mcM)$ to $H^{s-r}(\mcM)$ for any $s \in \R$; see \cite[Proposition 11]{harbrecht2024multilevel}. 
Using Cauchy-Schwarz inequality together with the bound $\|\mcA \psi_{j',k'}\|_{H^{-r/2+\sigma}}\lesssim \|\psi_{j',k'}\|_{H^{r/2+\sigma}}$, we obtain
\begin{align}\label{eq:second_estimate_aux2}
|\langle \mcA \psi_{j',k'},\psi_{j,k} \rangle|\le \|\mcA \psi_{j',k'}\|_{H^{-r/2+\sigma}} \|\psi_{j,k}\|_{H^{r/2-\sigma}}\lesssim \|\psi_{j',k'}\|_{H^{r/2+\sigma}} \|\psi_{j,k}\|_{H^{r/2-\sigma}}.
\end{align}
If $r/2+\sigma<\gamma$ and $r/2-\sigma>-\widetilde{\gamma}$, the norm equivalence in Lemma \ref{lemma:wavelet_property} (iv) can be applied to each factor on the right-hand side of \eqref{eq:second_estimate_aux2}. Using biorthogonality then yields
\begin{align*}
|\langle \mcA \psi_{j',k'},\psi_{j,k} \rangle| &\lesssim_{\sigma,r} \|\psi_{j',k'}\|_{H^{r/2+\sigma}} \|\psi_{j,k}\|_{H^{r/2-\sigma}}\\ 
&\asymp_{\sigma,r} 2^{j'(r/2+\sigma)}\cdot 2^{j(r/2-\sigma)}=2^{(j+j')r/2-\sigma(j-j')}.
\end{align*}
Hence the estimate follows. The parameter $\sigma$ must satisfy $0<\sigma<\min\{\gamma-r/2,\widetilde{\gamma}+r/2\}$.

\item \textbf{Diagonal preconditioning:} 

The argument follows the ideas in \cite[Propositions 3 and 13]{harbrecht2024multilevel}, but we provide a self-contained proof for the reader’s convenience.

We first consider $\mathcal{C}_u$ and $\widetilde{\bfC}_u$. Under Assumption \ref{assumption:operator_data_noise} (ii), $\mathcal{C}_u$ is self-adjoint and compact on $L^2(\mcM)$; see \cite[Proposition 1]{harbrecht2024multilevel}. Consequently, the bi-infinite matrix $\widetilde{\bfC}_u$ defines a symmetric and compact operator on $\ell^2(\mcJ)$. Moreover, \cite[Proposition 11]{harbrecht2024multilevel}, $\mcC_u:H^{-r_1}(\mcM)\to H^{r_1}(\mcM)$ is continuous, and hence
\[
\langle \mcC_u v,v \rangle\le \|\mcC_w v\|_{H^{r_1}} \|v\|_{H^{-r_1}}  \lesssim \|v\|_{H^{-r_1}}^2, \quad \forall \ v\in H^{-r_1}.
\]
Combined with Assumption \ref{assumption:operator_data_noise}(iii), this yields the norm equivalence
\[
\langle \mcC_{u} v,v\rangle\asymp \|v\|^2_{H^{-r_1}},\quad \forall \ v\in H^{-r_1}.
\]

Let $v=\widetilde{\bfv}^{\top}\widetilde{\Psi} =\sum_{\lambda\in\mcJ}\widetilde{v}_{\lambda}\widetilde{\psi}_{\lambda}$. Since $\widetilde{\bfC}_u$ is the matrix representation of $\mcC_u$ in the biorthogonal system $(\Psi,\widetilde{\Psi})$,
\[
\langle \mcC_{u} v,v\rangle= \langle \widetilde{\bfC}_u \widetilde{\bfv},\widetilde{\bfv} \rangle.
\]
By Lemma \ref{lemma:wavelet_property} (iv), if $-r_1 \in(-\gamma, \widetilde{\gamma})$, then
\[
\|v\|^2_{H^{-r_1}}\asymp \sum_{j \geq j_0} \sum_{k \in \nabla_j} 2^{-2 j r_1}\left|\left\langle v, \psi_{j, k}\right\rangle\right|^2=\|\bfD^{-r_1}\widetilde{\bfv}\|^2.
\]
Therefore, for all $\widetilde{\bfv}$ with $\bfD^{-r_1}\widetilde{\bfv}\in \ell^2(\mcJ)$,
\[
\langle \widetilde{\bfC}_u \widetilde{\bfv},\widetilde{\bfv} \rangle \asymp \|\bfD^{-r_1}\widetilde{\bfv}\|^2.
\]
This implies uniform spectral bounds for the diagonally preconditioned matrix:
\[
c_{-}\le \sigma_{\min}(\bfD^{r_1} \widetilde{\bfC}_u \bfD^{r_1})\le \sigma_{\max}(\bfD^{r_1} \widetilde{\bfC}_u \bfD^{r_1})\le c_{+}.
\]
The same argument can be applied to $\bfA$ in part (i) and $\bfC_w$ in part (iii), provided that $r/2, -r_2 \in (-\widetilde{\gamma}, \gamma)$.

\end{enumerate}
\end{proof}

\section{Auxiliary Materials Section \ref{sec:define_estimator}}
\subsection{Domain Partition}

Mathematically, the six disjoint regions $D_1$-$D_6$ in Figure \ref{fig:figure1} are defined as:
\begin{align*}
D_1&:=\bigg\{(j,j'):J \ge j> j'\ge 0, \ j\ge \frac{t+t'-r}{\sigma-\frac{n}{2}+t'-\frac{r}{2}}J+\frac{\sigma-\frac{n}{2}-(t-\frac{r}{2})}{\sigma-\frac{n}{2}+t'-\frac{r}{2}} j'\bigg\},
\\
D_2 &:=\bigg\{(j,j') : J\ge j'> j\ge 0, \ j'\ge \frac{t+t'-r}{\sigma-\frac{n}{2}+t-\frac{r}{2}}J+\frac{\sigma-\frac{n}{2}-(t'-\frac{r}{2})}{\sigma-\frac{n}{2}+t-\frac{r}{2}} j\bigg\},\\
D_3&:=\bigg\{(j,j'):J \ge j> j'\ge 0, \  j\le \frac{t+t'-r}{\sigma-\frac{n}{2}+t'-\frac{r}{2}}J+\frac{\sigma-\frac{n}{2}-(t-\frac{r}{2})}{\sigma-\frac{n}{2}+t'-\frac{r}{2}} j',\\
&\qquad \qquad\quad j\ge \frac{t+t'-r}{\widetilde{d}+t'}J+\frac{\widetilde{d}+r-t}{\widetilde{d}+t'}j'
\bigg\},
\\
D_4&:=\bigg\{(j,j'):J \ge j'> j\ge 0, \ j'\le \frac{t+t'-r}{\sigma-\frac{n}{2}+t-\frac{r}{2}}J+\frac{\sigma-\frac{n}{2}-(t'-\frac{r}{2})}{\sigma-\frac{n}{2}+t-\frac{r}{2}} j,\\
&\qquad \qquad\quad j'\ge  \frac{t+t'-r}{\widetilde{d}+t}J+\frac{\widetilde{d}+r-t'}{\widetilde{d}+t}j \bigg\}, \\
D_5 &:=\bigg\{(j,j'):J(t+t'-r)-jt'-j't-(j+j')\widetilde{d}<0, \ j\le \frac{t+t'-r}{\widetilde{d}+t'}J+\frac{\widetilde{d}+r-t}{\widetilde{d}+t'}j',\\
&\qquad \qquad\quad j'\le  \frac{t+t'-r}{\widetilde{d}+t}J+\frac{\widetilde{d}+r-t'}{\widetilde{d}+t}j\bigg\},\\
D_6 &:=\left\{(j,j'):J(t+t'-r)-jt'-j't-(j+j')\widetilde{d}\ge 0\right\}.
\end{align*}

\subsection{Support Monotonicity and Inclusion}\label{ssec:Appendixsupportmonotonicity}

\begin{lemma}[Support monotonicity]\label{lemma:supp_monotonicity}
 Let $(J,t,t')$ and $(\widetilde{J},\widetilde{t},\widetilde{t}')$ be two sets of parameters, and let $\supp(J,t,t')$ and $\supp(\widetilde{J},\widetilde{t},\widetilde{t}')$ be defined as in Definition \ref{def:supp_Jtt'}. Assume $\widetilde{J}\ge J,\widetilde{t}\ge t, \widetilde{t}'\ge t'$, and $\sigma-n/2>\max\{\widetilde{t},\widetilde{t}'\}-r/2, \ \widetilde{d}>\sigma-n/2-r/2$. Then,
\[
\supp(J,t,t')\subset \supp(\widetilde{J},\widetilde{t},\widetilde{t}').
\]
\end{lemma}
\begin{proof}
By definition,
\begin{align*}
\supp(J,t,t')&=\bigg\{(\lambda,\lambda')\in \Lambda_J\times \Lambda_J: \mathrm{dist}(S_{\lambda},S_{\lambda'})\le \tau_{jj'}, j \le \frac{t+t'-r}{\sigma-\frac{n}{2}+t'-\frac{r}{2}}J+\frac{\sigma-\frac{n}{2}-(t-\frac{r}{2})}{\sigma-\frac{n}{2}+t'-\frac{r}{2}} j',\nonumber\\
&\qquad j'\le \frac{t+t'-r}{\sigma-\frac{n}{2}+t-\frac{r}{2}}J+\frac{\sigma-\frac{n}{2}-(t'-\frac{r}{2})}{\sigma-\frac{n}{2}+t-\frac{r}{2}} j\bigg\},
\end{align*}
where
  \[ \tau_{jj'}=a\max \left\{2^{-\min\{j,j'\}}, 2^{(J(t+t'-r)-jt'-j't-(j+j
')\widetilde{d})/(2\widetilde{d}+r)}\right\}.
\]
Likewise,
\begin{align*}
\supp(\widetilde{J},\widetilde{t},\widetilde{t}')&:=\bigg\{(\lambda,\lambda')\in \Lambda_{\widetilde{J}}\times \Lambda_{\widetilde{J}}: \mathrm{dist}(S_{\lambda},S_{\lambda'})\le \widetilde{\tau}_{jj'}, j \le \frac{\widetilde{t}+\widetilde{t}'-r}{\sigma-\frac{n}{2}+\widetilde{t}'-\frac{r}{2}}\widetilde{J}+\frac{\sigma-\frac{n}{2}-(\widetilde{t}-\frac{r}{2})}{\sigma-\frac{n}{2}+\widetilde{t}'-\frac{r}{2}} j',\nonumber\\
&\qquad j'\le \frac{\widetilde{t}+\widetilde{t}'-r}{\sigma-\frac{n}{2}+\widetilde{t}-\frac{r}{2}}\widetilde{J}+\frac{\sigma-\frac{n}{2}-(\widetilde{t}'-\frac{r}{2})}{\sigma-\frac{n}{2}+\widetilde{t}-\frac{r}{2}} j\bigg\},
\end{align*}
with
\[
 \widetilde{\tau}_{jj'}=a\max \left\{2^{-\min\{j,j'\}}, 2^{(\widetilde{J}(\widetilde{t}+\widetilde{t}'-r)-j\widetilde{t}'-j'\widetilde{t}-(j+j
')\widetilde{d})/(2\widetilde{d}+r)}\right\}.
\]

Fix $(\lambda,\lambda')\in \supp(J,t,t')$; we show that $(\lambda,\lambda')\in \supp(\widetilde{J},\widetilde{t},\widetilde{t}')$. Since $J\le \widetilde{J}$, we immediately have
\[
(\lambda,\lambda')\in \Lambda_J\times \Lambda_J \quad \Longrightarrow \quad (\lambda,\lambda')\in \Lambda_{\widetilde{J}}\times \Lambda_{\widetilde{J}}.
\]
Next, note that 
\[
J(t+t'-r)-jt'-j't-(j+j')\widetilde{d}\le \widetilde{J}(\widetilde{t}+\widetilde{t}'-r)-j\widetilde{t}'-j'\widetilde{t}-(j+j')\widetilde{d},
\]
and hence $\tau_{jj'}\le \widetilde{\tau}_{jj'}$. Consequently,
\[
\mathrm{dist}(S_{\lambda},S_{\lambda'})\le \tau_{jj'}\quad \Longrightarrow \quad \mathrm{dist}(S_{\lambda},S_{\lambda'})\le \widetilde{\tau}_{jj'}.
\]
Moreover,
\[
j \le \frac{t+t'-r}{\sigma-\frac{n}{2}+t'-\frac{r}{2}}J+\frac{\sigma-\frac{n}{2}-(t-\frac{r}{2})}{\sigma-\frac{n}{2}+t'-\frac{r}{2}} j' \quad \Longrightarrow \quad j \le \frac{\widetilde{t}+\widetilde{t}'-r}{\sigma-\frac{n}{2}+\widetilde{t}'-\frac{r}{2}}\widetilde{J}+\frac{\sigma-\frac{n}{2}-(\widetilde{t}-\frac{r}{2})}{\sigma-\frac{n}{2}+\widetilde{t}'-\frac{r}{2}} j',
\]
which follows from
\begin{align*}
&\bigg(\frac{\widetilde{t}+\widetilde{t}'-r}{\sigma-\frac{n}{2}+\widetilde{t}'-\frac{r}{2}}\widetilde{J}+\frac{\sigma-\frac{n}{2}-(\widetilde{t}-\frac{r}{2})}{\sigma-\frac{n}{2}+\widetilde{t}'-\frac{r}{2}} j'\bigg)-\bigg(\frac{t+t'-r}{\sigma-\frac{n}{2}+t'-\frac{r}{2}}J+\frac{\sigma-\frac{n}{2}-(t-\frac{r}{2})}{\sigma-\frac{n}{2}+t'-\frac{r}{2}} j'\bigg)\\
&\ge \bigg(\frac{\widetilde{t}+\widetilde{t}'-r}{\sigma-\frac{n}{2}+\widetilde{t}'-\frac{r}{2}}J+\frac{\sigma-\frac{n}{2}-(\widetilde{t}-\frac{r}{2})}{\sigma-\frac{n}{2}+\widetilde{t}'-\frac{r}{2}} j'\bigg)-\bigg(\frac{t+t'-r}{\sigma-\frac{n}{2}+t'-\frac{r}{2}}J+\frac{\sigma-\frac{n}{2}-(t-\frac{r}{2})}{\sigma-\frac{n}{2}+t'-\frac{r}{2}} j'\bigg)\\
&= \bigg(\frac{\widetilde{t}+\widetilde{t}'-r}{\sigma-\frac{n}{2}+\widetilde{t}'-\frac{r}{2}}-\frac{t+t'-r}{\sigma-\frac{n}{2}+t'-\frac{r}{2}}\bigg)(J-j')\ge 0,
\end{align*}
since $t\le\widetilde{t}$, $t'\le\widetilde{t}', J\le \widetilde{J}$, and $\sigma-n/2>\max\{\widetilde{t},\widetilde{t}'\}-r/2$. A completely analogous argument yields
\[
j'\le \frac{t+t'-r}{\sigma-\frac{n}{2}+t-\frac{r}{2}}J+\frac{\sigma-\frac{n}{2}-(t'-\frac{r}{2})}{\sigma-\frac{n}{2}+t-\frac{r}{2}} j\quad \Longrightarrow \quad  j'\le \frac{\widetilde{t}+\widetilde{t}'-r}{\sigma-\frac{n}{2}+\widetilde{t}-\frac{r}{2}}\widetilde{J}+\frac{\sigma-\frac{n}{2}-(\widetilde{t}'-\frac{r}{2})}{\sigma-\frac{n}{2}+\widetilde{t}-\frac{r}{2}} j.
\]
Therefore, $\supp(J,t,t')\subset \supp(\widetilde{J},\widetilde{t},\widetilde{t}')$.
\end{proof}

\medskip

\begin{proof}[Proof of Lemma \ref{lemma:upper_contain_lower}]
  Let $(\lambda,\lambda')=((j,k),(j',k'))$ with $j>j'$ and suppose
$(\bfM_{(J,t,t')})_{\lambda,\lambda'}=1$.  
Then, by definition of the mask $\bfM_{(J,t,t')}$, the following conditions hold: 
    \begin{itemize}
    \item \emph{Distance condition:} \ $\mathrm{dist}(S_{j,k},S_{j',k'})\le \tau_{jj'}$.
     \item \emph{Scale constraint for $j'$:} \  $j'\le \frac{t+t'-r}{\sigma-\frac{n}{2}+t-\frac{r}{2}}J+\frac{\sigma-\frac{n}{2}-(t'-\frac{r}{2})}{\sigma-\frac{n}{2}+t-\frac{r}{2}} j$.
      \item \emph{Scale constraint for $j$:} \  $j \le \frac{t+t'-r}{\sigma-\frac{n}{2}+t'-\frac{r}{2}}J+\frac{\sigma-\frac{n}{2}-(t-\frac{r}{2})}{\sigma-\frac{n}{2}+t'-\frac{r}{2}} j'$.
      \end{itemize}
    To show $(\bfM_{(J,t,t')})_{\lambda',\lambda}=1$, we must verify that 
$(\lambda',\lambda)=((j',k'),(j,k))$ also satisfies the conditions defining 
$\supp(J,t,t')$.

\begin{itemize}
    \item \emph{Distance condition.}
    Since $\mathrm{dist}(S_{j,k},S_{j',k'})=\mathrm{dist}(S_{j',k'},S_{j,k})$, we have
    \[
        \mathrm{dist}(S_{j',k'},S_{j,k})
        \le \tau_{jj'}
        \overset{(\star)}{\le} \tau_{j'j},
    \]
    where $(\star)$ follows from
    \[
        \tau_{jj'} \le \tau_{j'j}
        \quad\Longleftrightarrow\quad
        (-jt'-j't) \le (-j't'-jt)
        \quad\Longleftrightarrow\quad
        (j'-j)(t'-t) \le 0.
    \]
    Since $j>j'$ and $t\le t'$, the last inequality holds.  Thus the distance 
    condition for $(\lambda',\lambda)$ is satisfied.

    \item \emph{Scale constraint for $j'$.}
    Since $j'<j$, it follows that
    \[
        j' < j 
        \le 
        \frac{t+t'-r}{\sigma-\frac{n}{2}+t'-\frac{r}{2}}\,J
        +\frac{\sigma-\frac{n}{2}-(t-\frac{r}{2})}{\sigma-\frac{n}{2}+t'-\frac{r}{2}}\, j,
    \]
    so the required upper bound for $j'$ holds.

    \item \emph{Scale constraint for $j$.}
    From the assumed inequality
    \[
    j \le 
    \frac{t+t'-r}{\sigma-\frac{n}{2}+t'-\frac{r}{2}}\,J
    + \frac{\sigma-\frac{n}{2}-(t-\frac{r}{2})}{\sigma-\frac{n}{2}+t'-\frac{r}{2}}\, j'
    \]
    and using the fact that $t\le t'$, we obtain
    \[
    j \le 
    \frac{t+t'-r}{\sigma-\frac{n}{2}+t-\frac{r}{2}}\,J
    + \frac{\sigma-\frac{n}{2}-(t'-\frac{r}{2})}{\sigma-\frac{n}{2}+t-\frac{r}{2}}\, j'.
    \]
\end{itemize}

Since all defining conditions of $\supp(J,t,t')$ hold for $(\lambda',\lambda)$, 
we conclude that $(\bfM_{(J,t,t')})_{\lambda',\lambda}=1$, completing the proof.
\end{proof}

\section{Auxiliary Materials Section \ref{sec:convergencerate}}\label{app:convergencerate}

\subsection{Sparsity Estimates}\label{app:sparsity_estimates}
 
In this subsection, we first present several basic decay estimates in Lemma~\ref{lemma:basic_estimates}, derived from Proposition~\ref{prop:property_UAW} \ref{approximate_sparsity}. These bounds will be used repeatedly in the analysis of our estimator. Although similar estimates appear in various places in the literature ---see, for instance, \cite{dahmen1997wavelet, dahmen2006compression, harbrecht2024multilevel}--- we provide here a concise and self-contained exposition for convenience. Proposition~\ref{prop:sparsity_count} below characterizes the sparsity pattern $\supp(J,t,t')$ introduced in Definition~\ref{def:supp_Jtt'} and counts the corresponding number of nonzero entries in the compressed matrix, which plays a central role in the variance analysis. In particular, this shows that the compression achieves optimal sparsity, since $\mathrm{nnz}\big(\bfM_{(J,t,t')}\big)\lesssim 2^{Jn}$, which matches the intrinsic number of degrees of freedom in $V_J$.

Recall that, under Definition \ref{def:supp_Jtt'}, the thresholding parameter is given by
\[
\tau_{jj'}\asymp \max \left\{2^{-\min\{j,j'\}}, 2^{(J(t+t'-r)-jt'-j't-(j+j
')\widetilde{d})/(2\widetilde{d}+r)}\right\}.
\]

\begin{lemma}[Basic estimates]\label{lemma:basic_estimates}

\begin{enumerate}[label=(\roman*)]
    \item For $j\ge j'\ge j_0$ and $\alpha>\frac{n}{n+r+2\widetilde{d}}$,
    \begin{align*}
    \sum_{k\in \nabla_j} |\langle \mcA\psi_{j,k},\psi_{j',k'}\rangle|^{\alpha} &\lesssim 2^{\alpha(j+j')r/2}\cdot 2^{-(j-j')\alpha\sigma} 2^{(j-j')n}, \\
     \sum_{k'\in \nabla_{j'}} |\langle \mcA\psi_{j,k},\psi_{j',k'}\rangle|^{\alpha}&\lesssim 2^{\alpha(j+j')r/2}\cdot 2^{-(j-j')\alpha\sigma}.
    \end{align*}
    \item For $j\ge j'\ge j_0$ and $\tau_{jj'}\gtrsim 2^{-j'
}$, 
    \[
    \big|\{k\in \nabla_j: \mathrm{dist}(S_{j,k},S_{j',k'})\le \tau_{jj'}\}\big|\lesssim 2^{jn}\tau_{jj'}^n,
    \]
    \[
     \big|\{k'\in \nabla_{j'}: \mathrm{dist}(S_{j,k},S_{j',k'})\le \tau_{jj'}\}\big|\lesssim 2^{j'n}\tau_{jj'}^n.
    \]
   \item For $j\ge j'\ge j_0$, $\tau_{jj'}\gtrsim 2^{-j'
}$ and $\alpha>\frac{n}{n+r+2\widetilde{d}}$,
    \[
     \sum_{k\in \nabla_j: \mathrm{dist}(S_{j,k},S_{j',k'})\ge \tau_{jj'}} |\langle \mcA\psi_{j,k},\psi_{j',k'}\rangle|^{\alpha}\lesssim 2^{-\alpha(j+j')(\widetilde{d}+n/2)}2^{jn}\tau_{jj'}^{-\alpha(n+r+2\widetilde{d})+n},
    \]
    \[
    \sum_{k'\in \nabla_{j'}: \mathrm{dist}(S_{j,k},S_{j',k'})\ge \tau_{jj'}} |\langle \mcA\psi_{j,k},\psi_{j',k'}\rangle|^{\alpha}\lesssim 2^{-\alpha(j+j')(\widetilde{d}+n/2)}2^{j'n}\tau_{jj'}^{-\alpha(n+r+2\widetilde{d})+n}.
    \]
\end{enumerate}
\end{lemma}

\begin{proof}
\begin{enumerate}[label=(\roman*)]
\item By Proposition \ref{prop:property_UAW} \ref{approximate_sparsity}, for $j\ge j'\ge j_0$,
\[
|\langle \mcA \psi_{j,k},\psi_{j',k'} \rangle|\lesssim 2^{(j+j')r/2}\cdot 2^{-(j-j')\sigma}\left(1+2^{j'}\mathrm{dist}(S_{j,k},S_{j',k'})\right)^{-(n+r+2\widetilde{d})}.
\]
where $0<\sigma <\min\left\{\gamma-r/2,\widetilde{\gamma}+r/2\right\}$. Therefore,
  \begin{align*}
       &\sum_{k\in \nabla_j} |\langle \mcA\psi_{j,k},\psi_{j',k'}\rangle|^{\alpha}\\
       &\lesssim 2^{\alpha(j+j')r/2}\cdot 2^{-(j-j')\alpha\sigma}\sum_{k\in \nabla_j}\left(1+2^{j'}\mathrm{dist}(S_{j,k},S_{j',k'})\right)^{-\alpha(n+r+2\widetilde{d})}\\
       &=2^{\alpha(j+j')r/2}\cdot 2^{-(j-j')\alpha\sigma} \sum_{k\in \nabla_j:\mathrm{dist}(S_{j,k},S_{j',k'})\le C2^{-j'}}\left(1+2^{j'}\mathrm{dist}(S_{j,k},S_{j',k'})\right)^{-\alpha(n+r+2\widetilde{d})}\\
  &\quad +2^{\alpha(j+j')r/2}\cdot 2^{-(j-j')\alpha\sigma}\sum_{k\in \nabla_j:\mathrm{dist}(S_{j,k},S_{j',k'})\ge C2^{-j'}}\left(1+2^{j'}\mathrm{dist}(S_{j,k},S_{j',k'})\right)^{-\alpha(n+r+2\widetilde{d})}.
  \end{align*}
 Since $j\ge j'$, there are at most $O(2^{(j-j')n})$ terms in the first sum: 
 \[
\sum_{k\in \nabla_j:\mathrm{dist}(S_{j,k},S_{j',k'})\le C2^{-j'}}\left(1+2^{j'}\mathrm{dist}(S_{j,k},S_{j',k'})\right)^{-\alpha(n+r+2\widetilde{d})}\lesssim 2^{(j-j')n}.
 \]
Moreover, since $\alpha(n+r+2\widetilde{d})>n$, we can estimate the second sum by an integral:
 \begin{align*}
 &\sum_{k\in \nabla_j:\mathrm{dist}(S_{j,k},S_{j',k'})\ge C2^{-j'}}\left(1+2^{j'}\mathrm{dist}(S_{j,k},S_{j',k'})\right)^{-\alpha(n+r+2\widetilde{d})}\\
 &\qquad \lesssim 2^{-\alpha(n+r+2\widetilde{d}) j'}\cdot 2^{jn}\int_{\|x\|\ge C2^{-j'} }\|x\|^{-\alpha(n+r+2\widetilde{d})}dx\\
 &\qquad \lesssim 2^{-\alpha(n+r+2\widetilde{d}) j'}\cdot 2^{jn}\cdot (C2^{-j'})^{-\alpha(n+r+2\widetilde{d})+n}\\
 &\qquad \asymp 2^{(j-j')n}.
 \end{align*}
Combining the above estimates yields that
\[
 \sum_{k\in \nabla_j} |\langle \mcA\psi_{j,k},\psi_{j',k'}\rangle|^{\alpha}\lesssim 2^{\alpha(j+j')r/2}\cdot 2^{-(j-j')\alpha\sigma} 2^{(j-j')n}.
\]

\item  By Lemma \ref{lemma:wavelet_property} (i), there exist fixed constants $C$ and $M$ such that 
\[
\mathrm{diam}(\mathrm{supp}(\psi_{j,k})) \le C 2^{-j},
\]
and, for each $k\in \nabla_j$, there are at most $M$ indices $k' \in \nabla_j$ satisfying \[
\mathrm{meas}(\mathrm{supp}(\psi_{j,k}) \cap \mathrm{supp}(\psi_{j,k'})) \ne 0.
\] 
Moreover, since $\tau_{jj'}\gtrsim 2^{-\min\{j,j'\}}$, it follows that
  \[
  \big|\{k\in \nabla_j: \mathrm{dist}(S_{j,k},S_{j',k'})\le \tau_{jj'}\}\big| \lesssim \left(\frac{\tau_{jj'}}{2^{-j}}\right)^n  = 2^{jn}\tau_{jj'}^n,
    \]
    and similarly,
    \[
     \big|\{k'\in \nabla_{j'}: \mathrm{dist}(S_{j,k},S_{j',k'})\le \tau_{jj'}\}\big|\lesssim \left(\frac{\tau_{jj'}}{2^{-j'}}\right)^n = 2^{j'n}\tau_{jj'}^n.
    \]
    
\item 
By Proposition \ref{prop:property_UAW} \ref{approximate_sparsity} (i),
    \begin{align*}
     &\sum_{k\in \nabla_j: \mathrm{dist}(S_{j,k},S_{j',k'})\ge \tau_{jj'}} |\langle \mcA\psi_{j,k},\psi_{j',k'}\rangle|^{\alpha}\\
     &\lesssim \sum_{k\in \nabla_j: \mathrm{dist}(S_{j,k},S_{j',k'})\ge \tau_{jj'}} 2^{-\alpha(j+j')(\widetilde{d}+n/2)}\mathrm{dist}(S_{j,k},S_{j',k'})^{-\alpha(n+r+2\widetilde{d})}.
    \end{align*}
    For $j\ge j'\ge j_0$, $\tau_{jj'}\gtrsim 2^{-j'
}=2^{-\min\{j,j'\}}$, we can estimate the sum by an integral:
\begin{align*}
    &\lesssim 2^{-\alpha(j+j')(\widetilde{d}+n/2)}2^{jn}\tau_{jj'}^{-\alpha(n+r+2\widetilde{d})+n},
\end{align*}
where we used $\alpha(n+r+2\widetilde{d})>n$ in the last step. Similarly,
 \[
    \sum_{k'\in \nabla_{j'}: \mathrm{dist}(S_{j,k},S_{j',k'})\ge \tau_{jj'}} |\langle \mcA\psi_{j,k},\psi_{j',k'}\rangle|^{\alpha}\lesssim 2^{-\alpha(j+j')(\widetilde{d}+n/2)}2^{j'n}\tau_{jj'}^{-\alpha(n+r+2\widetilde{d})+n}.
    \]
\end{enumerate}

\end{proof}

\begin{proposition}[Sparsity pattern of $\supp(J,t,t')$]\label{prop:sparsity_count}
For any matrix $B$, write $\mathrm{nnz}(B):=\#\{(i,j): B_{ij}\neq 0\}$, $\mathrm{nnz}_{\mathrm{row}}(B):=\max_i \#\{j: B_{ij}\neq 0\}$, and $
\mathrm{nnz}_{\mathrm{col}}(B):=\max_j \#\{i: B_{ij}\neq 0\}$. Let $\bfM_{(J,t,t')}\in \{0,1\}^{\Lambda_J\times \Lambda_J}$ be defined as in Definition \ref{def:supp_Jtt'}. Then the following estimates hold.
\begin{enumerate}[label=(\roman*)]
\item For $(j,j')\in D_1\cup D_2$, 
\[
\mathrm{nnz}\big( (\bfM_{(J,t,t')})_{j,j'} \big)=0, \qquad \mathrm{nnz}_{\mathrm{row}}\big( (\bfM_{(J,t,t')})_{j,j'} \big)=0 , \qquad \mathrm{nnz}_{\mathrm{col}}\big( (\bfM_{(J,t,t')})_{j,j'} \big)=0.
\]
\item For $(j,j')\in D_3$, $\tau_{jj'}\asymp 2^{-j'}$,
\[
\mathrm{nnz}\big( (\bfM_{(J,t,t')})_{j,j'} \big)\lesssim 2^{jn}, \qquad \mathrm{nnz}_{\mathrm{row}}\big( (\bfM_{(J,t,t')})_{j,j'} \big)\lesssim 1 , \qquad \mathrm{nnz}_{\mathrm{col}}\big( (\bfM_{(J,t,t')})_{j,j'} \big)\lesssim 2^{(j-j')n}.
\]
\item For $(j,j')\in D_4$, $\tau_{jj'}\asymp 2^{-j}$,
\[
\mathrm{nnz}\big( (\bfM_{(J,t,t')})_{j,j'} \big)\lesssim 2^{j'n}, \qquad \mathrm{nnz}_{\mathrm{row}}\big( (\bfM_{(J,t,t')})_{j,j'} \big)\lesssim 2^{(j'-j)n}, \qquad \mathrm{nnz}_{\mathrm{col}}\big( (\bfM_{(J,t,t')})_{j,j'} \big)\lesssim 1. 
\]

\item For $(j,j')\in D_5$, $\tau_{jj'}\asymp 2^{(J(t+t'-r)-jt'-j't-(j+j
')\widetilde{d})/(2\widetilde{d}+r)}$, 
\[
\mathrm{nnz}\big( (\bfM_{(J,t,t')})_{j,j'} \big)\lesssim 2^{(j+j')n}\tau_{jj'}^n, \quad \mathrm{nnz}_{\mathrm{row}}\big( (\bfM_{(J,t,t')})_{j,j'} \big)\lesssim 2^{j'n}\tau_{jj'}^n, \quad \mathrm{nnz}_{\mathrm{col}}\big( (\bfM_{(J,t,t')})_{j,j'} \big)\lesssim 2^{jn}\tau_{jj'}^n.
\]
\item For $(j,j')\in D_6$, $(A^{\varepsilon}_{\Lambda_J})_{j,j'}=(A_{\Lambda_J})_{j,j'}$,
\[
\mathrm{nnz}\big( (\bfM_{(J,t,t')})_{j,j'} \big)\lesssim 2^{(j+j')n}, \qquad \mathrm{nnz}_{\mathrm{row}}\big( (\bfM_{(J,t,t')})_{j,j'} \big)=2^{j'n} , \qquad \mathrm{nnz}_{\mathrm{col}}\big( (\bfM_{(J,t,t')})_{j,j'} \big)\lesssim 2^{jn}.
\]
\end{enumerate}
Moreover,
\[
\mathrm{nnz}\big(\bfM_{(J,t,t')}\big)\lesssim 2^{Jn},\quad
\mathrm{nnz}_{\mathrm{row}}\big(\bfM_{(J,t,t')}\big)\lesssim 2^{Jn\frac{t+t'-r}{\sigma-n/2+t-r/2}},
\quad
\mathrm{nnz}_{\mathrm{col}}\big(\bfM_{(J,t,t')}\big)\lesssim 2^{Jn\frac{t+t'-r}{\sigma-n/2+t'-r/2}}.
\]
All implicit constants may depend on $n,t,t',r,\sigma$ and the wavelet system $d,\widetilde{d},\gamma,\widetilde{\gamma}$, but are independent of $j,j',J$.

\end{proposition}

\begin{proof}
(i) follows directly from the construction of $\bfM_{(J,t,t')}$ in Definition \ref{def:supp_Jtt'}: the indicator matrix $(\bfM_{(J,t,t')})_{j,j'}=0$ for $(j,j')\in D_1\cup D_2$. For (ii)-(v), use Lemma \ref{lemma:basic_estimates}: for $j\ge j'$ and $\tau_{jj'}\gtrsim 2^{-j'}$, then
  \begin{align*}
   \big|\{k\in \nabla_j: \mathrm{dist}(S_{j,k},S_{j',k'})\le \tau_{jj'}\}\big| &\lesssim 2^{jn}\tau_{jj'}^n, \\
     \big|\{k'\in \nabla_{j'}: \mathrm{dist}(S_{j,k},S_{j',k'})\le \tau_{jj'}\}\big| &\lesssim 2^{j'n}\tau_{jj'}^n,
    \end{align*}
 together with
    \[
    \tau_{jj'}\asymp \begin{cases}
     2^{-j'}   &\quad \text{ if } (j,j')\in D_3, \\
      2^{-j}  &\quad \text{ if } (j,j')\in D_4,\\
     2^{(J(t+t'-r)-jt'-j't-(j+j
')\widetilde{d})/(2\widetilde{d}+r)} & \quad \text{ if } (j,j')\in D_5,\\
 \Theta(1) & \quad \text{ if } (j,j')\in D_6, \\
    \end{cases}
    \]
    which gives the stated bounds for each region. 

    To prove the global bounds, we sum the blockwise estimates over $(j,j')$ in all regions. Specifically, for the $\lambda'=(j',k')$-th column with $j'\le \frac{t+t'-r}{\widetilde{d}+t}J$, the number of nonzero entries is upper bounded by
    \begin{align*}
    &\sum_{0\le j\le \frac{t+t'-r}{\tilde{d}+t'}J-\frac{\tilde{d}+t}{\tilde{d}+t'}j'}2^{jn}+\sum_{\frac{t+t'-r}{\tilde{d}+t'}J-\frac{\tilde{d}+t}{\tilde{d}+t'}j'\le j\le \frac{t+t'-r}{\tilde{d}+t'}J+\frac{\tilde{d}+r-t}{\tilde{d}+t'}j'} 2^{jn}\cdot 2^{n(J(t+t'-r)-jt'-j't-(j+j
')\widetilde{d})/(2\widetilde{d}+r)}\\
&\qquad +\sum_{ \frac{t+t'-r}{\tilde{d}+t'}J+\frac{\tilde{d}+r-t}{\tilde{d}+t'}j'\le j\le \frac{t+t'-r}{\sigma-n/2+t'-r/2}J+\frac{\sigma-n/2-(t-r/2)}{\sigma-n/2+t'-r/2}j' }  2^{(j-j')n}\\
& \lesssim  2^{n\big(\frac{t+t'-r}{\tilde{d}+t'}J-\frac{\tilde{d}+t}{\tilde{d}+t'}j'\big)}+2^{n\frac{t+t'-r}{\tilde{d}+t'}(J-j')}+2^{n\frac{t+t'-r}{\sigma-n/2+t'-r/2}(J-j')}\\
&\lesssim 2^{Jn(t+t'-r)/(\sigma-\frac{n}{2}+t'-\frac{r}{2})}.
    \end{align*}
    For $(j',k')$-th column with $\frac{t+t'-r}{\widetilde{d}+t}J\le j'\le \frac{t+t'-r}{\sigma-\frac{n}{2}+t-\frac{r}{2}}J$, the number of nonzero entries is at most
    \begin{align*}
        &\sum_{0\le j\le \frac{\tilde{d}+t}{\tilde{d}+r-t'}j'-\frac{t+t'-r}{\tilde{d}+r-t'}J}1 +\sum_{\frac{\tilde{d}+t}{\tilde{d}+r-t'}j'-\frac{t+t'-r}{\tilde{d}+r-t'}J\le j\le \frac{t+t'-r}{\tilde{d}+t'}J+\frac{\tilde{d}+r-t}{\tilde{d}+t'}j'} 2^{jn}\cdot 2^{n(J(t+t'-r)-jt'-j't-(j+j
')\widetilde{d})/(2\widetilde{d}+r)}\\
&\qquad +\sum_{ \frac{t+t'-r}{\tilde{d}+t'}J+\frac{\tilde{d}+r-t}{\tilde{d}+t'}j'\le j\le \frac{t+t'-r}{\sigma-n/2+t'-r/2}J+\frac{\sigma-n/2-(t-r/2)}{\sigma-n/2+t'-r/2}j' }  2^{(j-j')n}\\
&\lesssim \frac{\widetilde{d}+t}{\widetilde{d}+r-t'}j'-\frac{t+t'-r}{\widetilde{d}+r-t'}J+2^{n\frac{t+t'-r}{\tilde{d}+t'}(J-j')}+2^{n\frac{t+t'-r}{\sigma-n/2+t'-r/2}(J-j')}\\
&\lesssim 2^{Jn(t+t'-r)/(\sigma-\frac{n}{2}+t'-\frac{r}{2})}.
    \end{align*}
For $(j',k')$-th column with $\frac{t+t'-r}{\sigma-\frac{n}{2}+t-\frac{r}{2}}J \le j'\le J$, the number of nonzero entries is at most
\begin{align*}
    &\sum_{\frac{\sigma-n/2+t-r/2}{\sigma-n/2-(t'-r/2)}j'-\frac{t+t'-r}{\sigma-n/2-(t'-r/2)}J\le j\le \frac{\tilde{d}+t}{\tilde{d}+r-t'}j'-\frac{t+t'-r}{\tilde{d}+r-t'}J} 1\\
    &\qquad +\sum_{\frac{\tilde{d}+t}{\tilde{d}+r-t'}j'-\frac{t+t'-r}{\tilde{d}+r-t'}J\le j\le \frac{t+t'-r}{\tilde{d}+t'}J+\frac{\tilde{d}+r-t}{\tilde{d}+t'}j'} 2^{jn}\cdot 2^{n(J(t+t'-r)-jt'-j't-(j+j
')\widetilde{d})/(2\widetilde{d}+r)}\\
&\qquad +\sum_{ \frac{t+t'-r}{\tilde{d}+t'}J+\frac{\tilde{d}+r-t}{\tilde{d}+t'}j'\le j\le \frac{t+t'-r}{\sigma-n/2+t'-r/2}J+\frac{\sigma-n/2-(t-r/2)}{\sigma-n/2+t'-r/2}j' }  2^{(j-j')n}\\
&\lesssim \frac{\widetilde{d}+t}{\widetilde{d}+r-t'}j'-\frac{t+t'-r}{\widetilde{d}+r-t'}J-\left(\frac{\sigma-\frac{n}{2}+t-\frac{r}{2}}{\sigma-\frac{n}{2}-(t'-\frac{r}{2})}j'-\frac{t+t'-r}{\sigma-\frac{n}{2}-(t'-\frac{r}{2})}J\right)\\
&\qquad +2^{n\frac{t+t'-r}{\tilde{d}+t'}(J-j')}+2^{n\frac{t+t'-r}{\sigma-n/2+t'-r/2}(J-j')}\\
&\lesssim 2^{Jn(t+t'-r)/(\sigma-\frac{n}{2}+t'-\frac{r}{2})}.
\end{align*}
Combining these three cases yields that
\[
\mathrm{nnz}_{\mathrm{col}}\big(\bfM_{(J,t,t')}\big)\lesssim 2^{Jn(t+t'-r)/(\sigma-\frac{n}{2}+t'-\frac{r}{2})}.
\]
Similarly,
\[
\mathrm{nnz}_{\mathrm{row}}\big(\bfM_{(J,t,t')}\big)\lesssim 2^{Jn(t+t'-r)/(\sigma-\frac{n}{2}+t-\frac{r}{2})}.
\]
Finally,
\begin{align*}
&\mathrm{nnz}\big(\bfM_{(J,t,t')}\big)\lesssim \sum_{j'=0}^{J}2^{j'n}\cdot \mathrm{nnz}((\bfM_{(J,t,t')})_{\cdot, \lambda'})\\
    &\lesssim \sum_{0\le j'\le \frac{t+t'-r}{\tilde{d}+t}J}2^{j'n}\left(2^{n\big(\frac{t+t'-r}{\tilde{d}+t'}J-\frac{\tilde{d}+t}{\tilde{d}+t'}j'\big)}+2^{n\frac{t+t'-r}{\sigma-n/2+t'-r/2}(J-j')}\right)\\
    &\quad +\sum_{\frac{t+t'-r}{\tilde{d}+t}J\le j'\le \frac{t+t'-r}{\sigma-n/2+t-r/2}J}2^{j'n}\left(\frac{\widetilde{d}+t}{\widetilde{d}+r-t'}j'-\frac{t+t'-r}{\widetilde{d}+r-t'}J+2^{n\frac{t+t'-r}{\sigma-n/2+t'-r/2}(J-j')}\right)\\
    &\quad+\sum_{\frac{t+t'-r}{\sigma-n/2+t-r/2}J\le j'\le J} 2^{j'n}\left(\frac{\widetilde{d}+t}{\widetilde{d}+r-t'}j'-\frac{t+t'-r}{\widetilde{d}+r-t'}J-\left(\frac{\sigma-\frac{n}{2}+t-\frac{r}{2}}{\sigma-\frac{n}{2}-(t'-\frac{r}{2})}j'-\frac{t+t'-r}{\sigma-\frac{n}{2}-(t'-\frac{r}{2})}J\right)\right)\\
    &\quad+ \sum_{\frac{t+t'-r}{\sigma-n/2+t-r/2}J\le j'\le J}2^{j'n}\cdot 2^{n\frac{t+t'-r}{\sigma-n/2+t'-r/2}(J-j')}\\
    &\lesssim 2^{Jn},
\end{align*}
where we have used the inequality $\sum_{0\le j'\le J}2^{j'n}(J-j')\lesssim 2^{Jn}$.
    \end{proof}

\subsection{Truncation}\label{app:truncation}
\begin{proof}[Proof of Proposition \ref{prop:Error-Truncation}]
We write $(\cdot)_{j,j'}$ for the $(j,j')$ wavelet block. By Lemma \ref{lemma:basic_estimates} (i) with $\alpha=1$, for any $j,j'\ge j_0$,
\begin{align*}
    \|\bfA_{j,j'}\|&\le \|\bfA_{j,j'}\|_1^{1/2}\; \|\bfA_{j,j'}\|_{\infty}^{1/2}\\
    &\le \sup_{k'\in \nabla_{j'}}\bigg( \sum_{k\in \nabla_j} |\langle \mcA\psi_{j,k},\psi_{j',k'}\rangle|\bigg)^{1/2}\cdot \sup_{k\in \nabla_j}\bigg( \sum_{k'\in \nabla_{j'}} |\langle \mcA\psi_{j,k},\psi_{j',k'}\rangle|\bigg)^{1/2}\\
    &\lesssim \left(2^{(j+j')r/2}\cdot 2^{-|j-j'|\sigma} 2^{|j-j'|n}  \cdot  2^{(j+j')r/2}\cdot 2^{-|j-j'|\sigma}\right)^{1/2}\\
    &\asymp 2^{(j+j')r/2-|j-j'|(\sigma-n/2)}.
\end{align*}

We now apply Lemma \ref{lemma:norm_compression} to $\bfD^{-t'} (\bfA_{\Lambda_J}^{\uparrow} - \bfA) \bfD^{-t}$ with the natural $(j,j')$-block partition. Let $\mathcal{N}(\cdot)$ denote the matrix of block norms.  Then
\begin{align*}
    \| \bfD^{-t'} (\bfA_{\Lambda_J}^{\uparrow} - \bfA) \bfD^{-t} \|&\le \|\mathcal{N}(\bfD^{-t'} (\bfA_{\Lambda_J}^{\uparrow} - \bfA) \bfD^{-t})\|\\
    &\le  \|\mathcal{N}(\bfD^{-t'} (\bfA_{\Lambda_J}^{\uparrow} - \bfA) \bfD^{-t})\|_1^{1/2} \;  \|\mathcal{N}(\bfD^{-t'} (\bfA_{\Lambda_J}^{\uparrow} - \bfA) \bfD^{-t})\|_{\infty}^{1/2}.
\end{align*}
For $j,j'\ge j_0$,
\begin{align*}\left(\mathcal{N}(\bfD^{-t'} (\bfA_{\Lambda_J}^{\uparrow} - \bfA) \bfD^{-t})\right)_{j,j'}= \|(\bfD^{-t'} (\bfA_{\Lambda_J}^{\uparrow} - \bfA) \bfD^{-t})_{j,j'}\|=\begin{cases}
    0 & \text{ if } j,j'\le J, \\
    2^{-jt'-j't}\|\bfA_{j,j'}\|  & \text{ otherwise}.
\end{cases}
\end{align*}

We first estimate $\|\mathcal{N}(\bfD^{-t'} (\bfA_{\Lambda_J}^{\uparrow} - \bfA) \bfD^{-t})\|_1$. Split into the cases $j'\le J$ and $j'>J$: 
\begin{align*}
\|\mathcal{N}(\bfD^{-t'} (\bfA_{\Lambda_J}^{\uparrow} - \bfA) \bfD^{-t})\|_1\le \bigg(\sup_{j'\le J}\sum_{j>J}   2^{-jt'-j't}\|\bfA_{j,j'}\|\bigg) +\bigg(\sup_{j'>J}\sum_{j\ge j_0}   2^{-jt'-j't}\|\bfA_{j,j'}\|\bigg).
\end{align*}

\noindent\textbf{Case 1: $j'\le J$.} \quad Using the bound for $\|\bfA_{j,j'}\|$,
\begin{align*}
   \sup_{j'\le J}\sum_{j>J}   2^{-jt'-j't}\|\bfA_{j,j'}\|
   &\lesssim  \sup_{j'\le J}\sum_{j>J}   2^{-jt'-j't}\cdot 2^{(j+j')r/2-|j-j'|(\sigma-n/2)}\\
   &=\sup_{j'\le J}2^{j'(\sigma-n/2-t+r/2)}\sum_{j>J}   2^{-j(\sigma-n/2+t'-r/2)}\\
    &\lesssim 2^{J(\sigma-n/2-t+r/2)}\cdot 2^{-J(\sigma-n/2+t'-r/2)}\\
    &=2^{-J(t+t'-r)}.
\end{align*}

\noindent\textbf{Case 2: $j'>J$.} \quad Using the estimate for $\|\bfA_{j,j'}\|$ and decomposing the sum into the two regions $j_0\le j\le j'$ and $j>j'$, we obtain
\begin{align*}
   &\sup_{j'>J}\sum_{j\ge j_0}   2^{-jt'-j't}\|\bfA_{j,j'}\| \lesssim \sup_{j'>J}\sum_{j\ge j_0}   2^{-jt'-j't}\cdot 2^{(j+j')r/2-|j-j'|(\sigma-n/2)}\\
    &=\sup_{j'>J}\bigg(\sum_{j_0\le j\le j'}+\sum_{j>j'}\bigg)   2^{-jt'-j't+(j+j')r/2-|j-j'|(\sigma-n/2)}\\
    &=\sup_{j'>J}\bigg(\sum_{j_0\le j\le j'}2^{j'(-t+r/2-\sigma+n/2)}\cdot 2^{j(\sigma-n/2-(t'-r/2))}+\sum_{j>j'}2^{j'(-t+r/2+\sigma-n/2)}\cdot 2^{-j(\sigma-n/2+t'-r/2)}\bigg)\\
    &\lesssim \sup_{j'>J} \bigg(2^{j'(-t+r/2-\sigma+n/2)}\cdot 2^{j'(\sigma-n/2-(t'-r/2))}+2^{j'(-t+r/2+\sigma-n/2)}\cdot 2^{-j'(\sigma-n/2+t'-r/2)}\bigg)\\
    &\asymp 2^{-J(t+t'-r)},
\end{align*}
where we used that $\sigma-n/2>\max\{t,t'\}-r/2\ge 0$.

Combining both cases, we conclude that
\[
\|\mathcal{N}(\bfD^{-t'} (\bfA_{\Lambda_J}^{\uparrow} - \bfA) \bfD^{-t})\|_{1}\lesssim 2^{-J(t+t'-r)}.
\]
The bound for $\|\mathcal{N}(\cdot)\|_{\infty}$ follows identically, yielding
\[
\|\mathcal{N}(\bfD^{-t'} (\bfA_{\Lambda_J}^{\uparrow} - \bfA) \bfD^{-t})\|_{\infty} \lesssim 2^{-J(t+t'-r)}.
\]
Therefore,
\begin{align*}
\ErrorTruncation &=\| \bfD^{-t'} (\bfA_{\Lambda_J}^{\uparrow} - \bfA) \bfD^{-t} \|\\
&\le  \|\mathcal{N}(\bfD^{-t'} (\bfA_{\Lambda_J}^{\uparrow} - \bfA) \bfD^{-t})\|_1^{1/2} \;  \|\mathcal{N}(\bfD^{-t'} (\bfA_{\Lambda_J}^{\uparrow} - \bfA) \bfD^{-t})\|_{\infty}^{1/2}\\
&\lesssim 2^{-J(t+t'-r)}.
\end{align*}
\end{proof}

\subsection{Compression}\label{app:compression}
\begin{proof}[Proof of Proposition \ref{prop:Error-Compression}]

We first consider the case $t_{\ell}=t', t_r=t$. Recall that
\[
\ErrorCompression=\|\bfD^{-t'}_{\Lambda_{J}}(\bfA^{\varepsilon}_{\Lambda_J}-\bfA_{\Lambda_J})\bfD^{-t}_{\Lambda_{J}}\|.
\]
We apply the norm-compression inequality (Lemma \ref{lemma:norm_compression}) with the natural $(j,j')$-block partition, writing $(\cdot)_{j,j'}$ for the $(j,j')$-block:
\[
\|\bfD^{-t'}_{\Lambda_{J}}(\bfA^{\varepsilon}_{\Lambda_J}-\bfA_{\Lambda_J})\bfD^{-t}_{\Lambda_{J}}\|\le \bigg(\sum_{j,j'= j_0}^{J} 2^{-2jt'-2j't} \|(\bfA^{\varepsilon}_{\Lambda_J})_{j,j'}-(\bfA_{\Lambda_J})_{j,j'}\|^2\bigg)^{1/2}.
\]

Next, we analyze $\|(\bfA^{\varepsilon}_{\Lambda_J})_{j,j'}-(\bfA_{\Lambda_J})_{j,j'}\|$ for $j_0\le j,j'\le J$.

\medskip
\noindent\textbf{Case 1.}
If 
\[
j> j' \text{  and  } j\ge \frac{t+t'-r}{\sigma-\frac{n}{2}+t'-\frac{r}{2}}J+\frac{\sigma-\frac{n}{2}-(t-\frac{r}{2})}{\sigma-\frac{n}{2}+t'-\frac{r}{2}} j',
\]
then $(\bfA^{\varepsilon}_{\Lambda_J})_{j,j'}=0$, hence $\|(\bfA^{\varepsilon}_{\Lambda_J})_{j,j'}-(\bfA_{\Lambda_J})_{j,j'}\|= \|(\bfA_{\Lambda_J})_{j,j'}\|$. Using the matrix norm inequality $\|B\|\le \|B\|_1^{1/2}\|B\|_{\infty}^{1/2}$ together with Lemma \ref{lemma:basic_estimates} (i) (with $\alpha=1$), we obtain
\begin{align*}
\|(\bfA_{\Lambda_J})_{j,j'}\|& \le \|(\bfA_{\Lambda_J})_{j,j'}\|_1^{1/2}\cdot \|(\bfA_{\Lambda_J})_{j,j'}\|_{\infty}^{1/2}\\
&=\sup_{k'\in \nabla_{j'}}\bigg( \sum_{k\in \nabla_j} |\langle \mcA\psi_{j,k},\psi_{j',k'}\rangle|\bigg)^{1/2}\cdot \sup_{k\in \nabla_j}\bigg( \sum_{k'\in \nabla_{j'}} |\langle \mcA\psi_{j,k},\psi_{j',k'}\rangle|\bigg)^{1/2}\\
&\lesssim \left(2^{(j+j')r/2}\cdot 2^{-(j-j')\sigma} 2^{(j-j')n}  \cdot  2^{(j+j')r/2}\cdot 2^{-(j-j')\sigma}\right)^{1/2}\\
&= 2^{(j+j')r/2-(j-j')(\sigma-n/2)}\\
&\le 2^{-J(t+t'-r)+jt'+j't},
\end{align*}
where the last inequality follows from the stated slope condition:
\[
j\ge \frac{t+t'-r}{\sigma-\frac{n}{2}+t'-\frac{r}{2}}J+\frac{\sigma-\frac{n}{2}-(t-\frac{r}{2})}{\sigma-\frac{n}{2}+t'-\frac{r}{2}} j'.
\]

\medskip
\noindent\textbf{Case 2.}
If 
\[
j'>j  \text{  and  } j'\ge \frac{t+t'-r}{\sigma-\frac{n}{2}+t-\frac{r}{2}}J+\frac{\sigma-\frac{n}{2}-(t'-\frac{r}{2})}{\sigma-\frac{n}{2}+t-\frac{r}{2}} j,
\]
then, by the same reasoning,
\[
\|(\bfA^{\varepsilon}_{\Lambda_J})_{j,j'}-(\bfA_{\Lambda_{J}})_{j,j'}\|=\|(\bfA_{\Lambda_{J}})_{j,j'}\|\lesssim 2^{(j+j')r/2-(j'-j)(\sigma-n/2)}\lesssim 2^{-J(t+t'-r)+jt'+j't}.
\]

\medskip
\noindent\textbf{Case 3.} In the remaining blocks, we set to zero all entries with $\mathrm{dist}(S_{j,k},S_{j',k'})>\tau_{jj'}$. Hence
\[
\|(\bfA^{\varepsilon}_{\Lambda_J})_{j,j'}-(\bfA_{\Lambda_{J}})_{j,j'}\|\le \|(\bfA^{\varepsilon}_{\Lambda_J})_{j,j'}-(\bfA_{\Lambda_{J}})_{j,j'}\|_1^{1/2}\cdot \|(\bfA^{\varepsilon}_{\Lambda_J})_{j,j'}-(\bfA_{\Lambda_{J}})_{j,j'}\|_{\infty}^{1/2},
\]
with
\begin{align*}
\|(\bfA^{\varepsilon}_{\Lambda_J})_{j,j'}-(\bfA_{\Lambda_{J}})_{j,j'}\|_1 &=\sup_{k'\in\nabla_{j'}}\sum_{k\in \nabla_j: \mathrm{dist}(S_{j,k},S_{j',k'})\ge \tau_{jj'}} |\langle \mcA\psi_{j,k},\psi_{j',k'}\rangle|\\
&\lesssim  2^{-(j+j')(\widetilde{d}+n/2)}2^{jn}\tau_{jj'}^{-(r+2\widetilde{d})},
\end{align*}
and 
\begin{align*}
\|(\bfA^{\varepsilon}_{\Lambda_J})_{j,j'}-(\bfA_{\Lambda_{J}})_{j,j'}\|_\infty &=\sup_{k\in\nabla_{j}}\sum_{k'\in \nabla_{j'}: \mathrm{dist}(S_{j,k},S_{j',k'})\ge \tau_{jj'}} |\langle \mcA\psi_{j,k},\psi_{j',k'}\rangle|\\
&\lesssim  2^{-(j+j')(\widetilde{d}+n/2)}2^{j'n}\tau_{jj'}^{-(r+2\widetilde{d})},
\end{align*}
by Lemma \ref{lemma:basic_estimates} (iii). Therefore,
\[
\|(\bfA^{\varepsilon}_{\Lambda_J})_{j,j'}-(\bfA_{\Lambda_{J}})_{j,j'}\|\lesssim 2^{-(j+j')\widetilde{d}}\cdot \tau_{jj'}^{-(2\widetilde{d}+r)}\lesssim  2^{-J(t+t'-r)+jt'+j't},
\]
where the last inequality uses $\tau_{jj'}\gtrsim 2^{(J(t+t'-r)-jt'-j't-(j+j
')\widetilde{d})/(2\widetilde{d}+r)}$.

Combining \textbf{Cases 1–3} yields,
\begin{align*}
\|\bfD^{-t'}_{\Lambda_{J}}(\bfA^{\varepsilon}_{\Lambda_J}-\bfA_{\Lambda_{J}})\bfD^{-t}_{\Lambda_{J}}\| &\le \bigg(\sum_{j,j'= j_0}^{J} 2^{-2jt'-2j't} \|(\bfA^{\varepsilon}_{\Lambda_J})_{j,j'}-(\bfA_{\Lambda_J})_{j,j'}\|^2\bigg)^{1/2}\\
& \lesssim  \bigg(\sum_{j,j'= j_0}^{J} 2^{-2jt'-2j't} 2^{-2J(t+t'-r)+2jt'+2j't}\bigg)^{1/2}\\
&\lesssim J\, 2^{-J(t+t'-r)},
\end{align*}
as claimed.

For general weights $t_{\ell}$ and $t_r$ satisfying $r/2\le t_{\ell}, t_{r}<\gamma$, we have
\begin{align*}
\|\bfD^{-t_{\ell}}_{\Lambda_{J}}(\bfA^{\varepsilon}_{\Lambda_J}-\bfA_{\Lambda_{J}})\bfD^{-t_{r}}_{\Lambda_{J}}\| &= \|\bfD^{t'-t_{\ell}}_{\Lambda_{J}}\,\bfD^{-t'}_{\Lambda_{J}}(\bfA^{\varepsilon}_{\Lambda_J}-\bfA_{\Lambda_{J}})\bfD^{-t}_{\Lambda_{J}}\, \bfD^{t-t_{r}}_{\Lambda_{J}}\| \\
&\le \|\bfD^{t'-t_{\ell}}_{\Lambda_{J}}\|\cdot \|\bfD^{-t'}_{\Lambda_{J}}(\bfA^{\varepsilon}_{\Lambda_J}-\bfA_{\Lambda_{J}})\bfD^{-t}_{\Lambda_{J}}\| \cdot \|\bfD^{t-t_{r}}_{\Lambda_{J}}\| \\
& \lesssim J\, 2^{-J(t+t'-r)}\cdot 2^{J\max\{0,t'-t_{\ell}\}}\cdot 2^{J\max\{0,t-t_{r}\}}\\
&\asymp J 2^{-J\left(\min\{t',t_{\ell}\}+\min\{t,t_{r}\}-r\right)}.
\end{align*}
\end{proof}

\subsection{Estimation}\label{app:estimation}

\subsubsection{Error Decomposition}

\begin{proof}[Proof of Lemma \ref{lemma:error_decomposition}]
We recall that the estimator $\widehat{\bfA} \in \mathbb{R}^{\Lambda_J \times \Lambda_J}$ in \eqref{eq:estimator_matrix_box} is defined entrywise by
\begin{align*}
\widehat{\bfA}_{\lambda,\lambda'}
:=
\begin{cases}
 (\widehat{\bfA}^{(1)})_{\lambda,\lambda'}  & \text{if } (\lambda,\lambda')\in \supp(J,t,t') \text{ and }  j \le j', \\[2mm]
 (\widehat{\bfA}^{(1)})_{\lambda',\lambda} & \text{if } (\lambda,\lambda')\in \supp(J,t,t') \text{ and } j > j',\\[2mm]
 0 & \text{if } (\lambda,\lambda')\notin  \supp(J,t,t').
\end{cases}
\end{align*}
Since $\widehat{\bfA}^{(1)}\in \R^{\Lambda_{\widetilde{J}}\,\times \Lambda_J}$ is rectangular, we introduce an intermediate square matrix 
\[
\widehat{\bfA}^{(2)}:=(\widehat{\bfA}^{(1)})_{\Lambda_J,\cdot}\in \R^{\Lambda_J\times \Lambda_J},
\]
obtained by restricting $\widehat{\bfA}^{(1)}$ to the rows indexed by $\Lambda_J$.

To facilitate the analysis, we decompose any matrix $\bfB\in \R^{\Lambda_J\times \Lambda_J}$ into its \emph{upper triangular} and \emph{lower triangular} parts, according to the ordering of the scale indices:
\begin{align*}
(\bfB_{\mathrm{up}})_{\lambda,\lambda'}:=\begin{cases}
 \bfB_{\lambda,\lambda'} &  j\le j',\\
 0&  j> j',
    \end{cases}\qquad (\bfB_{\mathrm{low}})_{\lambda,\lambda'}:=\begin{cases}
 0 & j\le j',\\
 \bfB_{\lambda,\lambda'} & j> j'.
    \end{cases}
\end{align*}
Thus, $\bfB=\bfB_{\mathrm{up}}+\bfB_{\mathrm{low}}$. We apply this decomposition to $\widehat{\bfA}, \widehat{\bfA}^{(2)}, \bfA_{\Lambda_J},\bfA_{\Lambda_J}^{\varepsilon}$, and the mask matrix $\bfM_{(J,t,t')}$.  

By construction in \eqref{eq:estimator_matrix_box}, the upper and lower triangular parts of $\widehat{\bfA}$ satisfy
\[
\widehat{\bfA}_{\mathrm{up}}=(\bfM_{(J,t,t')})_{\mathrm{up}}\odot \widehat{\bfA}^{(2)}, \qquad \widehat{\bfA}_{\mathrm{low}}=(\bfM_{(J,t,t')})_{\mathrm{low}}\odot (\widehat{\bfA}^{(2)})^{\top}.
\]
Since the compressed matrix is given by $\bfA_{\Lambda_J}^{\varepsilon}=\bfM_{(J,t,t')}\odot \bfA_{\Lambda_J}$, we likewise have
\[
(\bfA_{\Lambda_J}^{\varepsilon})_{\mathrm{up}}=(\bfM_{(J,t,t')})_{\mathrm{up}}\odot \bfA_{\Lambda_J},\qquad (\bfA_{\Lambda_J}^{\varepsilon})_{\mathrm{low}}=(\bfM_{(J,t,t')})_{\mathrm{low}}\odot \bfA_{\Lambda_J}.
\]
We therefore obtain
\begin{align}\label{eq:erro_decomposition_aux0}
   & \| \bfD^{-t'}_{\Lambda_J} (\widehat{\bfA} - \bfA_{\Lambda_J}^{\varepsilon}) \bfD^{-t}_{\Lambda_J} \| \nonumber\\
   &=\left\| \bfD^{-t'}_{\Lambda_J} \left(\widehat{\bfA}_{\mathrm{up}}+\widehat{\bfA}_{\mathrm{low}} - (\bfA_{\Lambda_J}^{\varepsilon})_{\mathrm{up}}-(\bfA_{\Lambda_J}^{\varepsilon})_{\mathrm{low}}\right) \bfD^{-t}_{\Lambda_J} \right\| \nonumber\\
    &\le \left\| \bfD^{-t'}_{\Lambda_J} \left(\widehat{\bfA}_{\mathrm{up}}- (\bfA_{\Lambda_J}^{\varepsilon})_{\mathrm{up}}\right) \bfD^{-t}_{\Lambda_J} \right\|+\left\| \bfD^{-t'}_{\Lambda_J} \left(\widehat{\bfA}_{\mathrm{low}}-(\bfA_{\Lambda_J}^{\varepsilon})_{\mathrm{low}}\right) \bfD^{-t}_{\Lambda_J} \right\| \nonumber\\
    &=\left\| \bfD^{-t'}_{\Lambda_J} (\bfM_{(J,t,t')})_{\mathrm{up}}\odot \left(\widehat{\bfA}^{(2)}- \bfA_{\Lambda_J}\right) \bfD^{-t}_{\Lambda_J} \right\|+\left\| \bfD^{-t'}_{\Lambda_J} (\bfM_{(J,t,t')})_{\mathrm{low}}\odot \left((\widehat{\bfA}^{(2)})^{\top}-\bfA_{\Lambda_J}\right) \bfD^{-t}_{\Lambda_J} \right\| \nonumber\\
    &\overset{(\star)}{=}\left\| \bfD^{-t'}_{\Lambda_J} (\bfM_{(J,t,t')})_{\mathrm{up}}\odot \left(\widehat{\bfA}^{(2)}- \bfA_{\Lambda_J}\right) \bfD^{-t}_{\Lambda_J} \right\|+\left\| \bfD^{-t}_{\Lambda_J} ((\bfM_{(J,t,t')})_{\mathrm{low}})^{\top}\odot \left(\widehat{\bfA}^{(2)}-\bfA_{\Lambda_J}\right) \bfD^{-t'}_{\Lambda_J} \right\|,
\end{align}
where $(\star)$ follows by transposing the second term and using $(\bfA_{\Lambda_J})^{\top}=\bfA_{\Lambda_J}$. Notice that both terms reduce to a similar structural error involving $\widehat{\bfA}^{(2)}-\bfA_{\Lambda_J}$.

Recall that the regression support $\Omega_{\lambda'}=\left\{\lambda\in \mcJ: (\lambda,\lambda')\in \supp(\widetilde{J},\widetilde{t},\widetilde{t}') \right\}$, $\Omega_{\lambda'}^c=\mcJ\setminus \Omega_{\lambda'}$, and \eqref{eq:model_column_partitioned}:
\begin{align*}
\bfF_{\cdot,\lambda'}=\bfU \bfA_{\cdot,\lambda'}+\bfW_{\cdot,\lambda'}=\bfU_{\cdot,\Omega_{\lambda'}}\bfA_{\Omega_{\lambda'},\lambda'}+\bfU_{\cdot,\Omega_{\lambda'}^c} \bfA_{\Omega_{\lambda'}^c,\lambda'}+\bfW_{\cdot,\lambda'},
\end{align*}
and that our preliminary estimator in \textbf{Step 1} \eqref{eq:estimator_column_box} is defined columnwise by
\begin{align*}
(\widehat{\bfA}^{(1)})_{\cdot,\lambda'}
:=
\begin{bmatrix}
(\bfU_{\cdot,\Omega_{\lambda'}}^{\top} \bfU_{\cdot,\Omega_{\lambda'}})^{-1}
\bfU_{\cdot,\Omega_{\lambda'}}^{\top} \bfF_{\cdot,\lambda'} \\
\boldsymbol{0}
\end{bmatrix}
\in \mathbb{R}^{\Lambda_{\widetilde{J}}},
\qquad
\lambda' \in \Lambda_{J}.
\end{align*}
Observe that
\begin{align*}
&(\bfU_{\cdot,\Omega_{\lambda'}}^{\top} \bfU_{\cdot,\Omega_{\lambda'}})^{-1}
\bfU_{\cdot,\Omega_{\lambda'}}^{\top} \bfF_{\cdot,\lambda'}=(\bfU_{\cdot,\Omega_{\lambda'}}^{\top} \bfU_{\cdot,\Omega_{\lambda'}})^{-1}
\bfU_{\cdot,\Omega_{\lambda'}}^{\top}\left(\bfU_{\cdot,\Omega_{\lambda'}}\bfA_{\Omega_{\lambda'},\lambda'}+\bfU_{\cdot,\Omega_{\lambda'}^c} \bfA_{\Omega_{\lambda'}^c,\lambda'}+\bfW_{\cdot,\lambda'}\right)\\
&=\bfA_{\Omega_{\lambda'},\lambda'}+(\bfU_{\cdot,\Omega_{\lambda'}}^{\top} \bfU_{\cdot,\Omega_{\lambda'}})^{-1}
\bfU_{\cdot,\Omega_{\lambda'}}^{\top}\bfU_{\cdot,\Omega_{\lambda'}^c} \bfA_{\Omega_{\lambda'}^c,\lambda'}+(\bfU_{\cdot,\Omega_{\lambda'}}^{\top} \bfU_{\cdot,\Omega_{\lambda'}})^{-1}
\bfU_{\cdot,\Omega_{\lambda'}}^{\top}\bfW_{\cdot,\lambda'}.
\end{align*}
Hence, for each column $\lambda'\in \Lambda_J$,
\begin{align*}
    (\widehat{\bfA}^{(1)})_{\cdot,\lambda'}
=\begin{bmatrix}
\bfA_{\Omega_{\lambda'},\lambda'} \\
\boldsymbol{0}
\end{bmatrix}+\begin{bmatrix}
(\bfU_{\cdot,\Omega_{\lambda'}}^{\top} \bfU_{\cdot,\Omega_{\lambda'}})^{-1}
\bfU_{\cdot,\Omega_{\lambda'}}^{\top}\bfU_{\cdot,\Omega_{\lambda'}^c} \bfA_{\Omega_{\lambda'}^c,\lambda'} \\
\boldsymbol{0}
\end{bmatrix}+\begin{bmatrix}
   (\bfU_{\cdot,\Omega_{\lambda'}}^{\top} \bfU_{\cdot,\Omega_{\lambda'}})^{-1}
\bfU_{\cdot,\Omega_{\lambda'}}^{\top}\bfW_{\cdot,\lambda'}  \\
\boldsymbol{0}
\end{bmatrix}.
\end{align*}

Let $E_{\Omega_{\lambda'}}:\R^{\Omega_{\lambda'}}\to \R^{\Lambda_{\widetilde{J}}}$ be the row embedding operator (pads zeros outside $\Omega_{\lambda'}$). We define
\[
\mathsf{A}:=\sum_{\lambda'\in \Lambda_J} \left(E_{\Omega_{\lambda'}}\bfA_{\Omega_{\lambda'},\lambda'}\right) e_{\lambda'}^{\top}\in \R^{\Lambda_{\widetilde{J}}\,\times \Lambda_J },
\]
\[
\OVB:=\sum_{\lambda'\in \Lambda_J} \left(E_{\Omega_{\lambda'}}(\bfU_{\cdot,\Omega_{\lambda'}}^{\top} \bfU_{\cdot,\Omega_{\lambda'}})^{-1}
\bfU_{\cdot,\Omega_{\lambda'}}^{\top}\bfU_{\cdot,\Omega_{\lambda'}^c} \bf\bfA_{\Omega^c_{\lambda'},\lambda'}\right) e_{\lambda'}^{\top}\in \R^{\Lambda_{\widetilde{J}}\,\times \Lambda_J },
\]
and
\[
\Var:=\sum_{\lambda'\in \Lambda_J} \left(E_{\Omega_{\lambda'}}(\bfU_{\cdot,\Omega_{\lambda'}}^{\top} \bfU_{\cdot,\Omega_{\lambda'}})^{-1} \bfU_{\cdot,\Omega_{\lambda'}}^{\top}\bfW_{\cdot,\lambda'}\right) e_{\lambda'}^{\top}\in \R^{\Lambda_{\widetilde{J}}\,\times \Lambda_J },
\]
where $e_{\lambda'}\in \R^{\Lambda_J}$ is the standard basis vector whose only nonzero entry is a $1$ at coordinate  $\lambda'$. Therefore,
\[
\widehat{\bfA}^{(1)}=\mathsf{A}+\OVB+\Var\in \R^{\Lambda_{\widetilde{J}}\,\times \Lambda_J }.
\]
Taking the rows indexed by $\Lambda_J$, we have
\[
\widehat{\bfA}^{(2)}=(\widehat{\bfA}^{(1)})_{\Lambda_J,\cdot}=\mathsf{A}_{\Lambda_J,\cdot}+\OVB_{\Lambda_J,\cdot}+\Var_{\Lambda_J,\cdot} \in \R^{\Lambda_J\times \Lambda_J}.
\]

Because for each $\lambda'\in \Lambda_J$, the support of the $\lambda'$-th column of  $M_{J,t,t'}$ is contained in $\Omega_{\lambda'}$, and because $\mathsf{A}_{\cdot,\lambda'}$ coincides with $(\bfA_{\Lambda_J})_{\cdot,\lambda'}$ on $\Omega_{\lambda'}$, we have 
\[
\bfM_{(J,t,t')}\odot (\mathsf{A}_{\Lambda_J,\cdot}-\bfA_{\Lambda_J})=0.
\]
Therefore,
\begin{align}\label{eq:erro_decomposition_aux1}
(\bfM_{(J,t,t')})_{\mathrm{up}}\odot \left(\widehat{\bfA}^{(2)}- \bfA_{\Lambda_J}\right) &= (\bfM_{(J,t,t')})_{\mathrm{up}}\odot \left(\mathsf{A}_{\Lambda_J,\cdot}+\OVB_{\Lambda_J,\cdot}+\Var_{\Lambda_J,\cdot}-\bfA_{\Lambda_J}\right)\nonumber\\
&=(\bfM_{(J,t,t')})_{\mathrm{up}}\odot \left(\OVB_{\Lambda_J,\cdot}+\Var_{\Lambda_J,\cdot}\right).
\end{align}

Moreover, by Lemma \ref{lemma:upper_contain_lower}, the support of $((\bfM_{(J,t,t')})_{\mathrm{low}})^{\top}$ is contained in the support of $(\bfM_{(J,t,t')})_{\mathrm{up}}$. Since we already have $(\bfM_{(J,t,t')})_{\mathrm{up}}\odot (\mathsf{A}_{\Lambda_J,\cdot}-\bfA_{\Lambda_J})=0$, it follows that
\[
((\bfM_{(J,t,t')})_{\mathrm{low}})^{\top}\odot (\mathsf{A}_{\Lambda_J,\cdot}-\bfA_{\Lambda_J})=0.
\]
Hence,
\begin{align}\label{eq:erro_decomposition_aux2}
 ((\bfM_{(J,t,t')})_{\mathrm{low}})^{\top}\odot \left(\widehat{\bfA}^{(2)}-\bfA_{\Lambda_J}\right)&= ((\bfM_{(J,t,t')})_{\mathrm{low}})^{\top}\odot\left(\mathsf{A}_{\Lambda_J,\cdot}+\OVB_{\Lambda_J,\cdot}+\Var_{\Lambda_J,\cdot}-\bfA_{\Lambda_J}\right)\nonumber\\
 &=((\bfM_{(J,t,t')})_{\mathrm{low}})^{\top}\odot\left(\OVB_{\Lambda_J,\cdot}+\Var_{\Lambda_J,\cdot}\right).
\end{align}

Combining \eqref{eq:erro_decomposition_aux0}, \eqref{eq:erro_decomposition_aux1}, and \eqref{eq:erro_decomposition_aux2}, we obtain the following upper bound on the  estimation error:
\begin{align}\label{eq:estimator_error_decomp}
    &\| \bfD^{-t'}_{\Lambda_J} (\widehat{\bfA} - \bfA_{\Lambda_J}^{\varepsilon}) \bfD^{-t}_{\Lambda_J} \| \nonumber\\
    &\le  \left\| \bfD^{-t'}_{\Lambda_J} (\bfM_{(J,t,t')})_{\mathrm{up}}\odot \left(\widehat{\bfA}^{(2)}- \bfA_{\Lambda_J}\right) \bfD^{-t}_{\Lambda_J} \right\|+\left\| \bfD^{-t}_{\Lambda_J} ((\bfM_{(J,t,t')})_{\mathrm{low}})^{\top}\odot \left(\widehat{\bfA}^{(2)}-\bfA_{\Lambda_J}\right) \bfD^{-t'}_{\Lambda_J} \right\|\nonumber\\
    &= \left\| \bfD^{-t'}_{\Lambda_J} (\bfM_{(J,t,t')})_{\mathrm{up}}\odot \left(\OVB_{\Lambda_J,\cdot}+\Var_{\Lambda_J,\cdot}\right) \bfD^{-t}_{\Lambda_J} \right\| + \left\| \bfD^{-t}_{\Lambda_J} ((\bfM_{(J,t,t')})_{\mathrm{low}})^{\top}\odot \left(\OVB_{\Lambda_J,\cdot}+\Var_{\Lambda_J,\cdot}\right) \bfD^{-t'}_{\Lambda_J} \right\| \nonumber\\
    &\le \underbrace{\left\| \bfD^{-t'}_{\Lambda_J} \left((\bfM_{(J,t,t')})_{\mathrm{up}}\odot \left(\OVB_{\Lambda_J,\cdot}\right)\right) \bfD^{-t}_{\Lambda_J} \right\| + \left\| \bfD^{-t}_{\Lambda_J} \left(((\bfM_{(J,t,t')})_{\mathrm{low}})^{\top}\odot \left(\OVB_{\Lambda_J,\cdot}\right)\right) \bfD^{-t'}_{\Lambda_J} \right\|}_{=:\ErrorOVB  } \nonumber\\
    &\quad + \underbrace{\left\| \bfD^{-t'}_{\Lambda_J} \left((\bfM_{(J,t,t')})_{\mathrm{up}}\odot \left(\Var_{\Lambda_J,\cdot}\right)\right) \bfD^{-t}_{\Lambda_J} \right\| + \left\| \bfD^{-t}_{\Lambda_J} \left(((\bfM_{(J,t,t')})_{\mathrm{low}})^{\top}\odot \left(\Var_{\Lambda_J,\cdot}\right)\right) \bfD^{-t'}_{\Lambda_J} \right\|}_{=:\ErrorVar}.
\end{align}
This completes the proof.
\end{proof}

\subsubsection{Omitted-Variable Bias}

\begin{proof}[Proof of Proposition \ref{prop:Error-OVB}]
Recall from \eqref{eq:def_parameter} that the parameters $\widetilde{J},\widetilde{t},\widetilde{t}'$ are defined by
\begin{align}\label{eq:def_tilde_Jtt'}
\widetilde{J}:=\frac{t+t'-r+\varepsilon_1}{\min\{t',r_1\}+t-r}J, \quad \widetilde{t}:=\max\{t,t'\}=t',\quad \widetilde{t}':=\max\{r_1,t'\}.
\end{align}

We further recall the definition of $\ErrorOVB$ from \eqref{eq:estimator_error_decomp}:
\begin{align*}
\ErrorOVB&=\left\| \bfD^{-t'}_{\Lambda_J} \left((\bfM_{(J,t,t')})_{\mathrm{up}}\odot \left(\OVB_{\Lambda_J,\cdot}\right)\right) \bfD^{-t}_{\Lambda_J} \right\| + \left\| \bfD^{-t}_{\Lambda_J} \left(((\bfM_{(J,t,t')})_{\mathrm{low}})^{\top}\odot \left(\OVB_{\Lambda_J,\cdot}\right)\right) \bfD^{-t'}_{\Lambda_J} \right\| \\
&=\left\|(\bfM_{(J,t,t')})_{\mathrm{up}}\odot \left(\bfD^{-t'}_{\Lambda_J} \OVB_{\Lambda_J,\cdot} \bfD^{-t}_{\Lambda_J}\right)  \right\| + \left\|((\bfM_{(J,t,t')})_{\mathrm{low}})^{\top}\odot \left(\bfD^{-t}_{\Lambda_J} \OVB_{\Lambda_J,\cdot} \bfD^{-t'}_{\Lambda_J}\right)  \right\| ,
\end{align*}
where \[
\OVB=\sum_{\lambda'\in \Lambda_J} \left(E_{\Omega_{\lambda'}}(\bfU_{\cdot,\Omega_{\lambda'}}^{\top} \bfU_{\cdot,\Omega_{\lambda'}})^{-1}
\bfU_{\cdot,\Omega_{\lambda'}}^{\top}\bfU_{\cdot,\Omega_{\lambda'}^c} \bf\bfA_{\Omega^c_{\lambda'},\lambda'}\right) e_{\lambda'}^{\top}\in \R^{\Lambda_{\widetilde{J}}\,\times \Lambda_J }.
\]
Define $B^{t',t}:= \bfD_{\Lambda_{\widetilde{J}}}^{-t'} \,\OVB\, \bfD_{\Lambda_J}^{-t}$ and $ B^{t,t'}:= \bfD_{\Lambda_{\widetilde{J}}}^{-t} \,\OVB\, \bfD_{\Lambda_J}^{-t'}$, then
\begin{align}\label{eq:ovb_proof}
\ErrorOVB=\left\|(\bfM_{(J,t,t')})_{\mathrm{up}}\odot (B^{t',t})_{\Lambda_J,\cdot} \right\| + \left\|((\bfM_{(J,t,t')})_{\mathrm{low}})^{\top}\odot (B^{t,t'})_{\Lambda_J,\cdot}  \right\|.
\end{align}

To simply notation, we write
\[
\widehat{\Sigma}:= \frac{1}{N}\bfU^{\top}\bfU,\quad \widehat{\Sigma}_{\Omega_{\lambda'},\Omega_{\lambda'}}:=\frac{1}{N}\bfU_{\cdot,\Omega_{\lambda'}}^{\top} \bfU_{\cdot,\Omega_{\lambda'}},\quad \widehat{\Sigma}_{\Omega_{\lambda'},\Omega^c_{\lambda'}}:=\frac{1}{N}\bfU_{\cdot,\Omega_{\lambda'}}^{\top} \bfU_{\cdot,\Omega^c_{\lambda'}}.
\]
We also define the corresponding population quantities\footnote{Here we write $\Sigma$ in place of $\widetilde{\bfC}_u$ defined in Proposition \ref{prop:property_UAW} \ref{diagonal_preconditioning} (ii) to streamline the notation.}
\[
\Sigma:= \E \left[\frac{1}{N}\bfU^{\top}\bfU\right] ,\quad  \Sigma_{\Omega_{\lambda'},\Omega_{\lambda'}}:=\E \left[\frac{1}{N}\bfU_{\cdot,\Omega_{\lambda'}}^{\top}\bfU_{\cdot,\Omega_{\lambda'}}\right],\quad \Sigma_{\Omega_{\lambda'},\Omega^c_{\lambda'}}:=\E \left[\frac{1}{N}\bfU_{\cdot,\Omega_{\lambda'}}^{\top}\bfU_{\cdot,\Omega_{\lambda'}^c}\right].
\]
Moreover, by Proposition \ref{prop:property_UAW} \ref{diagonal_preconditioning} (ii), let $\bar{\Sigma}=\bfD^{r_1}\Sigma \bfD^{r_1}$ be the well-conditioned design covariance, which satisfies $c_{-}\le \sigma_{\min}(\bar{\Sigma})\le \sigma_{\max}(\bar{\Sigma})\le c_{+}$ for absolute constants $c_{-},c_{+}>0$; and its empirical version $\widehat{\bar{\Sigma}}=\bfD^{r_1}\widehat{\Sigma} \bfD^{r_1}$. 

The $\lambda'$-th column of $B^{t',t}$ can be written as 
\begin{align}\label{eq:ovb_proof_aux1}
(B^{t',t})_{\cdot,\lambda'}=E_{\Omega_{\lambda'}} \bfD^{-(t'-r_1)}_{\Omega_{\lambda'}}\big(\widehat{\bar{\Sigma}}_{\Omega_{\lambda'},\Omega_{\lambda'}}\big)^{-1}\widehat{\bar{\Sigma}}_{\Omega_{\lambda'},\Omega^c_{\lambda'}} \big(\bfD^{-r_1}_{\Omega^c_{\lambda'}}\bfA_{\Omega^c_{\lambda'},\lambda'}\big) (\bfD_{\Lambda_J}^{-t})_{\lambda',\lambda'}\in \R^{\Lambda_{\widetilde{J}}},
\end{align}
where $(\bfD_{\Lambda_J}^{-t})_{\lambda',\lambda'}=2^{-j't}\in \R$. 

Fix $\delta\in (0,1)$. By the first part of Lemma \ref{lemma:tech1}, for every $\lambda'\in \Lambda_J$, if  $N\gtrsim |\Omega_{\lambda'}|+\log(1/\delta)$, then with probability at least $1-\delta/2$, the matrix $\widehat{\bar{\Sigma}}_{\Omega_{\lambda'},\Omega_{\lambda'}}$ is invertible and satisfies
\begin{align}\label{eq:ovb_proof_aux2}
\|\big(\widehat{\bar{\Sigma}}_{\Omega_{\lambda'},\Omega_{\lambda'}}\big)^{-1}\|=1/\sigma_{\min}(\widehat{\bar{\Sigma}}_{\Omega_{\lambda'},\Omega_{\lambda'}})\lesssim 1.
\end{align}
We write
\[
\widehat{\bar{\Sigma}}_{\Omega_{\lambda'},\Omega^c_{\lambda'}}=\bfD^{r_1}_{\Omega_{\lambda'}}\widehat{\Sigma}_{\Omega_{\lambda'},\Omega^c_{\lambda'}}\bfD^{r_1}_{\Omega_{\lambda'}^c}=\frac{1}{N}\sum_{i=1}^{N} \bar{\bfU}_{i,\Omega_{\lambda'}}^{\top} \bar{\bfU}_{i,\Omega_{\lambda'}^c},
\]
where $\bar{\bfU}_{i,\Omega_{\lambda'}}:=\bfU_{i,\Omega_{\lambda'}}\bfD^{r_1}_{\Omega_{\lambda'}}, \bar{\bfU}_{i,\Omega^c_{\lambda'}}:=\bfU_{i,\Omega^c_{\lambda'}}\bfD^{r_1}_{\Omega^c_{\lambda'}}$. Thus,
\begin{align}\label{eq:ovb_aux1}
\left\|\widehat{\bar{\Sigma}}_{\Omega_{\lambda'},\Omega^c_{\lambda'}}\big(\bfD^{-r_1}_{\Omega^c_{\lambda'}} \bfA_{\Omega^c_{\lambda'},\lambda'}\big)(\bfD_{\Lambda_J}^{-t'})_{\lambda',\lambda'}\right\|=\left\|\frac{1}{N}\sum_{i=1}^{N}\bar{\bfU}_{i,\Omega_{\lambda'}}^{\top}\left\langle \bar{\bfU}_{i,\Omega_{\lambda'}^c} ,\big(\bfD^{-r_1}_{\Omega^c_{\lambda'}} \bfA_{\Omega^c_{\lambda'},\lambda'}\big)(\bfD_{\Lambda_J}^{-t'})_{\lambda',\lambda'}\right\rangle \right\|.
\end{align}

To control \eqref{eq:ovb_aux1}, we apply Lemma \ref{lemma:empirical_cross_term} with 
\[
X_i=\bar{\bfU}_{i,\Omega_{\lambda'}}^{\top}\in \R^{|\Omega_{\lambda'}|},\qquad Z_i=\left\langle \bar{\bfU}_{i,\Omega_{\lambda'}^c} ,\big(\bfD^{-r_1}_{\Omega^c_{\lambda'}} \bfA_{\Omega^c_{\lambda'},\lambda'}\big)(\bfD_{\Lambda_J}^{-t'})_{\lambda',\lambda'}\right\rangle,
\]
and note that
\[
\mu=\E [X_i Z_i]=\bar{\Sigma}_{\Omega_{\lambda'},\Omega^c_{\lambda'}}\big(\bfD^{-r_1}_{\Omega^c_{\lambda'}} \bfA_{\Omega^c_{\lambda'},\lambda'}\big)(\bfD_{\Lambda_J}^{-t'})_{\lambda',\lambda'}, 
\]
\[
\sigma_{Z}^2=\E [Z_i^2]= \left\langle\bar{\Sigma}_{\Omega^c_{\lambda'},\Omega^c_{\lambda'}} \big(\bfD^{-r_1}_{\Omega^c_{\lambda'}} \bfA_{\Omega^c_{\lambda'},\lambda'}\big)(\bfD_{\Lambda_J}^{-t'})_{\lambda',\lambda'}, \big(\bfD^{-r_1}_{\Omega^c_{\lambda'}} \bfA_{\Omega^c_{\lambda'},\lambda'}\big)(\bfD_{\Lambda_J}^{-t'})_{\lambda',\lambda'}\right\rangle,
\]
\[
\Sigma_{XX}= \E [X_iX_i^{\top}] =\bar{\Sigma}_{\Omega_{\lambda'},\Omega_{\lambda'}}.
\]
It follows that, with probability at least $1-\delta/2$,
\begin{align}\label{eq:ovb_proof_aux3}
&\left\|\widehat{\bar{\Sigma}}_{\Omega_{\lambda'},\Omega^c_{\lambda'}}\big(\bfD^{-r_1}_{\Omega^c_{\lambda'}} \bfA_{\Omega^c_{\lambda'},\lambda'}\big)(\bfD_{\Lambda_J}^{-t'})_{\lambda',\lambda'}\right\|_2 \nonumber\\
&=\left\|\frac{1}{N}\sum_{i=1}^{N}\bar{\bfU}_{i,\Omega_{\lambda'}}\left\langle \bar{\bfU}_{i,\Omega_{\lambda'}^c} ,\big(\bfD^{-r_1}_{\Omega^c_{\lambda'}} \bfA_{\Omega^c_{\lambda'},\lambda'}\big)(\bfD_{\Lambda_J}^{-t'})_{\lambda',\lambda'}\right\rangle \right\|_2 \nonumber\\
&\lesssim \big(\sigma_Z \|\Sigma_{XX}\|^{1/2} + \|\mu\|_2\big)\!\left(
\sqrt{\frac{|\Omega_{\lambda'}|+\log(4/\delta)}{N}}
+
\frac{|\Omega_{\lambda'}|+\log(4/\delta)}{N}
\right) \nonumber\\
&\overset{(\star)}{\lesssim} \|\big(\bfD^{-r_1}_{\Omega^c_{\lambda'}} \bfA_{\Omega^c_{\lambda'},\lambda'}\big)(\bfD_{\Lambda_J}^{-t'})_{\lambda',\lambda'}\|_2 \left(
\sqrt{\frac{|\Omega_{\lambda'}|+\log(4/\delta)}{N}}
+
\frac{|\Omega_{\lambda'}|+\log(4/\delta)}{N}
\right) \nonumber\\
&\lesssim \|\big(\bfD^{-r_1}_{\Omega^c_{\lambda'}} \bfA_{\Omega^c_{\lambda'},\lambda'}\big)(\bfD_{\Lambda_J}^{-t'})_{\lambda',\lambda'}\|_2,
\end{align}
provided that $N\gtrsim |\Omega_{\lambda'}|+\log (1/\delta)$. The inequality $(\star)$ follows from the bounds
\begin{align*}
\sigma_Z =(\E [Z_i^2])^{1/2}&=\left\langle\bar{\Sigma}_{\Omega^c_{\lambda'},\Omega^c_{\lambda'}} \big(\bfD^{-r_1}_{\Omega^c_{\lambda'}} \bfA_{\Omega^c_{\lambda'},\lambda'}\big)(\bfD_{\Lambda_J}^{-t'})_{\lambda',\lambda'}, \big(\bfD^{-r_1}_{\Omega^c_{\lambda'}} \bfA_{\Omega^c_{\lambda'},\lambda'}\big)(\bfD_{\Lambda_J}^{-t'})_{\lambda',\lambda'}\right\rangle^{1/2}\\
&\lesssim \|\big(\bfD^{-r_1}_{\Omega^c_{\lambda'}} \bfA_{\Omega^c_{\lambda'},\lambda'}\big)(\bfD_{\Lambda_J}^{-t'})_{\lambda',\lambda'}\|_2, \qquad\text{since }\|\bar{\Sigma}_{\Omega^c_{\lambda'},\Omega^c_{\lambda'}}\|\lesssim 1,
\end{align*}
and 
\[
\|\mu\|_2=\left\| \bar{\Sigma}_{\Omega_{\lambda'},\Omega^c_{\lambda'}}\big(\bfD^{-r_1}_{\Omega^c_{\lambda'}} \bfA_{\Omega^c_{\lambda'},\lambda'}\big)(\bfD_{\Lambda_J}^{-t'})_{\lambda',\lambda'}\right\|_2 \lesssim \|\big(\bfD^{-r_1}_{\Omega^c_{\lambda'}} \bfA_{\Omega^c_{\lambda'},\lambda'}\big)(\bfD_{\Lambda_J}^{-t'})_{\lambda',\lambda'}\|_2,\quad \text{since }\|\bar{\Sigma}_{\Omega_{\lambda'},\Omega^c_{\lambda'}}\|\lesssim 1.
\]

Therefore, combining \eqref{eq:ovb_proof_aux1}, \eqref{eq:ovb_proof_aux2}, and \eqref{eq:ovb_proof_aux3}, we conclude that if 
\[
N\gtrsim |\Omega_{\lambda'}|+\log (1/\delta),
\]
then with probability at least $1-\delta$,
\begin{align}\label{eq:ovb_proof_aux4}
   \| (B^{t',t})_{\cdot,\lambda'}\|_2
   &=\left\| E_{\Omega_{\lambda'}} \bfD^{-(t'-r_1)}_{\Omega_{\lambda'}}\big(\widehat{\bar{\Sigma}}_{\Omega_{\lambda'},\Omega_{\lambda'}}\big)^{-1}\widehat{\bar{\Sigma}}_{\Omega_{\lambda'},\Omega^c_{\lambda'}} \big(\bfD^{-r_1}_{\Omega^c_{\lambda'}}\bfA_{\Omega^c_{\lambda'},\lambda'}\big) (\bfD_{\Lambda_J}^{-t})_{\lambda',\lambda'}\right\|_2 \nonumber\\
   &\le  \|E_{\Omega_{\lambda'}} \|\cdot \|\bfD^{-(t'-r_1)}_{\Omega_{\lambda'}}\|\cdot \|\big(\widehat{\bar{\Sigma}}_{\Omega_{\lambda'},\Omega_{\lambda'}}\big)^{-1}\|\cdot \left\|\widehat{\bar{\Sigma}}_{\Omega_{\lambda'},\Omega^c_{\lambda'}}\big(\bfD^{-r_1}_{\Omega^c_{\lambda'}} \bfA_{\Omega^c_{\lambda'},\lambda'}\big)(\bfD_{\Lambda_J}^{-t})_{\lambda',\lambda'}\right\|_2 \nonumber\\
   &\lesssim 2^{\widetilde{J}\max\{r_1-t',0\}} \cdot \|\big(\bfD^{-r_1}_{\Omega^c_{\lambda'}} \bfA_{\Omega^c_{\lambda'},\lambda'}\big)(\bfD_{\Lambda_J}^{-t})_{\lambda',\lambda'}\|_2,
\end{align}
where the last inequality uses $\|E_{\Omega_{\lambda'}} \|= 1, \|\bfD^{-(t'-r_1)}_{\Omega_{\lambda'}}\|\le 2^{\widetilde{J}\max\{r_1-t',0\}}$.

Let $\bfA^{\varepsilon}_{\Lambda_{\widetilde{J}}}$ be the compressed matrix as in Definition \ref{def:supp_Jtt'} with parameters $(\widetilde{J},\,\widetilde{t},\,\widetilde{t}')$, i.e.
\[
\bfA^{\varepsilon}_{\Lambda_{\widetilde{J}}}:= M_{(\widetilde{J},\,\widetilde{t},\,\widetilde{t}')} \odot \bfA_{\Lambda_{\widetilde{J}}}\,.
\]
Applying Proposition~\ref{prop:Error-Compression} with parameters $(\widetilde{J},\,\widetilde{t},\,\widetilde{t}')$, we obtain that, for $r_1, t\in [r/2,\gamma)$,
\[
\|\bfD^{-r_1}_{\Lambda_{\widetilde{J}}}(\bfA^{\varepsilon}_{\Lambda_{\widetilde{J}}}-\bfA_{\Lambda_{\widetilde{J}}})\bfD^{-t}_{\Lambda_{\widetilde{J}}}\|\lesssim  \widetilde{J} 2^{-\widetilde{J}\left(\min\{\widetilde{t}',r_1\}+\min\{\widetilde{t},t\}-r\right)}= \widetilde{J} 2^{-\widetilde{J}\left(r_1+t-r\right)},
\]
where the last equality follows since $\widetilde{t}=\max\{t,t'\}\ge t$ and $\widetilde{t}'=\max\{r_1,t'\}\ge r_1$.

On the other hand, since $t> r/2$ and $r_1> n/2+\max\{0,r\}\ge r/2$ by Assumption \ref{assumption:operator_data_noise} (ii), by Proposition \ref{prop:Error-Truncation},
\[
\| \bfD^{-r_1} (\bfA_{\Lambda_{\widetilde{J}}}^{\uparrow} - \bfA) \bfD^{-t} \|\lesssim 2^{-\widetilde{J}\left(r_1+t-r\right)}.
\]

 Recall that the regression support is defined by \[
 \Omega_{\lambda'}:=\left\{\lambda\in \mcJ: (\lambda,\lambda')\in \supp(\widetilde{J},\widetilde{t},\widetilde{t}') \right\}, \qquad \Omega_{\lambda'}^c:=\mcJ\setminus \Omega_{\lambda'}.
 \]
For every $\lambda'\in \Lambda_J$, the vector $ \bfA_{\Omega^c_{\lambda'},\lambda'}$ collects exactly those entries of the $\lambda'$-th column of $\bfA$ that are omitted from the regression. By construction of the truncation $\bfA_{\Lambda_{\widetilde{J}}}^{\uparrow}$ and the compression mask $M_{(\widetilde{J},\,\widetilde{t},\,\widetilde{t}')}$, these omitted entries coincide with the entries of the $\lambda'$-th column of $\bfA_{\Lambda_{\widetilde{J}}}^{\uparrow} - \bfA$ and $\bfA^{\varepsilon}_{\Lambda_{\widetilde{J}}}-\bfA_{\Lambda_{\widetilde{J}}}$\,. Therefore, after weighting by the diagonal matrices on both sides,
\begin{align}\label{eq:ovb_proof_aux5}
\|\big(\bfD^{-r_1}_{\Omega^c_{\lambda'}} \bfA_{\Omega^c_{\lambda'},\lambda'}\big)(\bfD_{\Lambda_J}^{-t})_{\lambda',\lambda'}\|_2\le \| \bfD^{-r_1} (\bfA_{\Lambda_{\widetilde{J}}}^{\uparrow} - \bfA) \bfD^{-t} \|+\|\bfD^{-r_1}_{\Lambda_{\widetilde{J}}}(\bfA^{\varepsilon}_{\Lambda_{\widetilde{J}}}-\bfA_{\Lambda_{\widetilde{J}}})\bfD^{-t}_{\Lambda_{\widetilde{J}}}\| \lesssim \widetilde{J} 2^{-\widetilde{J}(r_1+t-r)}.
\end{align}

Combining \eqref{eq:ovb_proof_aux4} and \eqref{eq:ovb_proof_aux5} yields
\begin{align*}
\|(B^{t',t})_{\cdot,\lambda'}\|_2&\lesssim 2^{\widetilde{J}\max\{r_1-t',0\}} \cdot \|\big(\bfD^{-r_1}_{\Omega^c_{\lambda'}} \bfA_{\Omega^c_{\lambda'},\lambda'}\big)(\bfD_{\Lambda_J}^{-t})_{\lambda',\lambda'}\|_2\\
&\lesssim 2^{\widetilde{J}\max\{r_1-t',0\}}\cdot \widetilde{J} 2^{-\widetilde{J}(r_1+t-r)}\\
&=\widetilde{J} 2^{-\widetilde{J}(\min\{t',r_1\}+t-r)}.
\end{align*}

Taking a union bound over all columns indexed by $\lambda'\in \Lambda_{J}$, we obtain that, if
\[
N\gtrsim \max_{\lambda'\in \Lambda_{J}}|\Omega_{\lambda'}|+J+\log (1/\delta),
\]
then with probability at least $1-\delta$, the following bound holds simultaneously for all $\lambda'\in \Lambda_{J}$:
\[
\|(B^{t',t})_{\cdot,\lambda'}\|_2\lesssim \widetilde{J} 2^{-\widetilde{J}(\min\{t',r_1\}+t-r)}.
\]

For any matrix $B$, write $\mathrm{nnz}(B):=\#\{(i,j): B_{ij}\neq 0\}$, $\mathrm{nnz}_{\mathrm{row}}(B):=\max_i \#\{j: B_{ij}\neq 0\}$, and $
\mathrm{nnz}_{\mathrm{col}}(B):=\max_j \#\{i: B_{ij}\neq 0\}$. By Proposition \ref{prop:sparsity_count},
\[
\mathrm{nnz}_{\mathrm{row}}(\bfM_{(J,t,t')})\lesssim 2^{Jn(t+t'-r)/(\sigma-\frac{n}{2}+t-\frac{r}{2})},\quad \mathrm{nnz}_{\mathrm{col}}(\bfM_{(J,t,t')})\lesssim 2^{Jn(t+t'-r)/(\sigma-\frac{n}{2}+t'-\frac{r}{2})}.
\]

Hence,
\begin{align*}
    \left\|(\bfM_{(J,t,t')})_{\mathrm{up}}\odot (B^{t',t})_{\Lambda_J,\cdot}  \right\|_1&=\max_{\lambda'\in\Lambda_J} \left\|\left((\bfM_{(J,t,t')})_{\mathrm{up}}\odot (B^{t',t})_{\Lambda_J,\cdot}\right)_{\cdot,\lambda'}\right\|_1\\
    &\le \max_{\lambda'\in\Lambda_J} \sqrt{\mathrm{nnz}_{\mathrm{col}}((\bfM_{(J,t,t')})_{\mathrm{up}})}\left\|\left((\bfM_{(J,t,t')})_{\mathrm{up}}\odot (B^{t',t})_{\Lambda_J,\cdot}\right)_{\cdot,\lambda'}\right\|_2\\
    &\le \max_{\lambda'\in\Lambda_J} \sqrt{\mathrm{nnz}_{\mathrm{col}}(\bfM_{(J,t,t')})}\cdot \,\|(B^{t',t})_{\cdot,\lambda'}\|_2\\
    &\lesssim 2^{Jn(t+t'-r)/2(\sigma-\frac{n}{2}+t'-\frac{r}{2})}\cdot \widetilde{J} 2^{-\widetilde{J}(\min\{t',r_1\}+t-r)}.
\end{align*}

On the other hand,
\begin{align*}
     \left\|(\bfM_{(J,t,t')})_{\mathrm{up}}\odot (B^{t',t})_{\Lambda_J,\cdot}  \right\|_{\infty}&=\max_{\lambda\in\Lambda_J} \left\|\left((\bfM_{(J,t,t')})_{\mathrm{up}}\odot (B^{t',t})_{\Lambda_J,\cdot}\right)_{\lambda,\cdot}\right\|_1\\
     &\le \mathrm{nnz}_{\mathrm{row}}((\bfM_{(J,t,t')})_{\mathrm{up}})\cdot \left\|(\bfM_{(J,t,t')})_{\mathrm{up}}\odot (B^{t',t})_{\Lambda_J,\cdot} \right\|_{\max}\\
     &\le \mathrm{nnz}_{\mathrm{row}}(\bfM_{(J,t,t')})\cdot \max_{\lambda'\in\Lambda_J} \|(B^{t',t})_{\cdot,\lambda'}\|_2\\
     &\lesssim 2^{Jn(t+t'-r)/(\sigma-\frac{n}{2}+t-\frac{r}{2})}\cdot \widetilde{J} 2^{-\widetilde{J}(\min\{t',r_1\}+t-r)}.
\end{align*}
Therefore,
\begin{align}\label{eq:ovb_proof_aux6}
\left\|(\bfM_{(J,t,t')})_{\mathrm{up}}\odot (B^{t',t})_{\Lambda_J,\cdot}  \right\|&\le \left(\,\left\|(\bfM_{(J,t,t')})_{\mathrm{up}}\odot (B^{t',t})_{\Lambda_J,\cdot}  \right\|_1 \left\|(\bfM_{(J,t,t')})_{\mathrm{up}}\odot (B^{t',t})_{\Lambda_J,\cdot}  \right\|_{\infty}\right)^{1/2}\nonumber\\
&\lesssim  \widetilde{J} 2^{-\widetilde{J}(\min\{t',r_1\}+t-r)}\cdot 2^{\varepsilon_1 J},
\end{align}
where
\[
\varepsilon_1:= \frac{n(t+t'-r)}{\sigma-n/2+t-r/2}\ge n(t+t'-r)\left(\frac{1}{4(\sigma-n/2+t'-r/2)}+\frac{1}{2(\sigma-n/2+t-r/2)}\right).
\]

Similarly, if $N\gtrsim \max_{\lambda'\in \Lambda_{J}}|\Omega_{\lambda'}|+J+\log (1/\delta)$, then with probability at least $1-\delta$, the following bound holds simultaneously for all $\lambda'\in \Lambda_{J}$:
\[
\|(B^{t,t'})_{\cdot,\lambda'}\|\lesssim \widetilde{J} 2^{-\widetilde{J}(\min\{t,r_1\}+t'-r)}.
\]
Using the same argument yields
\begin{align}\label{eq:ovb_proof_aux7}
\left\|((\bfM_{(J,t,t')})_{\mathrm{low}})^{\top}\odot (B^{t,t'})_{\Lambda_J,\cdot}  \right\|\lesssim \widetilde{J} 2^{-\widetilde{J}(\min\{t,r_1\}+t'-r)}\cdot 2^{\varepsilon_1 J},
\end{align}
where we used 
\[
\varepsilon_1=\frac{n(t+t'-r)}{\sigma-n/2+t-r/2}\ge n(t+t'-r)\left(\frac{1}{4(\sigma-n/2+t-r/2)}+\frac{1}{2(\sigma-n/2+t'-r/2)}\right).
\]

Combining \eqref{eq:ovb_proof}, \eqref{eq:ovb_proof_aux6}, and \eqref{eq:ovb_proof_aux7} yields that, if $N\gtrsim  \max_{\lambda'\in \Lambda_{J}}|\Omega_{\lambda'}|+J+\log (1/\delta)$, then with probability at least $1-\delta$, 
\begin{align*}
\ErrorOVB&=\left\|(\bfM_{(J,t,t')})_{\mathrm{up}}\odot (B^{t',t})_{\Lambda_J,\cdot} \right\| + \left\|((\bfM_{(J,t,t')})_{\mathrm{low}})^{\top}\odot (B^{t,t'})_{\Lambda_J,\cdot}  \right\|\\
&\lesssim \widetilde{J} 2^{-\widetilde{J}(\min\{t',r_1\}+t-r)}\cdot 2^{\varepsilon_1 J}+\widetilde{J} 2^{-\widetilde{J}(\min\{t,r_1\}+t'-r)}\cdot 2^{\varepsilon_1 J}\\
&\lesssim J\, 2^{-J(t+t'-r)},
\end{align*}
where the last line follows from the definition of $\widetilde{J}$
\[
\widetilde{J}=\left\lceil\frac{t+t'-r+\varepsilon_1}{\min\{t',r_1\}+t-r}\,J\right\rceil\ge \max\left\{\frac{t+t'-r+\varepsilon_1}{\min\{t',r_1\}+t-r}J,\, \frac{t+t'-r+\varepsilon_1}{\min\{t,r_1\}+t'-r}J\right\}.
\]
The final equality holds because $t\le t'$ implies
\[
\min\{t',r_1\}+t-r\le \min\{t,r_1\}+t'-r.
\]

In summary, we use the set $\supp(\widetilde{J},\widetilde{t},\widetilde{t}')$ to perform the regression. Then, we conclude that, if 
\[
N\gtrsim \max_{\lambda'\in \Lambda_{J}}|\Omega_{\lambda'}|+J+\log (1/\delta)\asymp 2^{\widetilde{J}n(\widetilde{t}+\widetilde{t}'-r)/(\sigma-\frac{n}{2}+\widetilde{t}'-\frac{r}{2})}+\log (1/\delta),
\]
then with probability at least $1-\delta$,
\begin{align}\label{eq:ovb_final_error}
\ErrorOVB\lesssim J\, 2^{-J(t+t'-r)}.
\end{align}

\end{proof}

\subsubsection{Variance}

\begin{proof}[Proof of Proposition \ref{prop:Error-Var}]

Recall that
\begin{align*}
\ErrorVar&=\left\| \bfD^{-t'}_{\Lambda_J} \left((\bfM_{(J,t,t')})_{\mathrm{up}}\odot \left(\Var_{\Lambda_J,\cdot}\right)\right) \bfD^{-t}_{\Lambda_J} \right\| + \left\| \bfD^{-t}_{\Lambda_J} \left(((\bfM_{(J,t,t')})_{\mathrm{low}})^{\top}\odot \left(\Var_{\Lambda_J,\cdot}\right)\right) \bfD^{-t'}_{\Lambda_J} \right\|\\
&=\left\|(\bfM_{(J,t,t')})_{\mathrm{up}}\odot \left(\bfD^{-t'}_{\Lambda_J} \Var_{\Lambda_J,\cdot} \bfD^{-t}_{\Lambda_J}\right)  \right\| + \left\|((\bfM_{(J,t,t')})_{\mathrm{low}})^{\top}\odot \left(\bfD^{-t}_{\Lambda_J} \Var_{\Lambda_J,\cdot} \bfD^{-t'}_{\Lambda_J}\right)  \right\| ,
\end{align*}
where 
\[
\Var:=\sum_{\lambda'\in \Lambda_J} \left(E_{\Omega_{\lambda'}}(\bfU_{\cdot,\Omega_{\lambda'}}^{\top} \bfU_{\cdot,\Omega_{\lambda'}})^{-1} \bfU_{\cdot,\Omega_{\lambda'}}^{\top}\bfW_{\cdot,\lambda'}\right) e_{\lambda'}^{\top}\in \R^{\Lambda_{\widetilde{J}}\,\times \Lambda_J }.
\]
Define $V^{t',t}:= \bfD_{\Lambda_{\widetilde{J}}}^{-t'} \,\Var\, \bfD_{\Lambda_J}^{-t}$ and $ V^{t,t'}:= \bfD_{\Lambda_{\widetilde{J}}}^{-t} \,\Var\, \bfD_{\Lambda_J}^{-t'}$, then
\begin{align}\label{eq:var_proof}
\ErrorVar=\left\|(\bfM_{(J,t,t')})_{\mathrm{up}}\odot (V^{t',t})_{\Lambda_J,\cdot} \right\| + \left\|((\bfM_{(J,t,t')})_{\mathrm{low}})^{\top}\odot (V^{t,t'})_{\Lambda_J,\cdot}  \right\|.
\end{align}

To control the first term in \eqref{eq:var_proof}, we apply the norm-compression inequality (Lemma \ref{lemma:norm_compression}) with the natural $(j,j')$-block partition, writing $(\cdot)_{j,j'}$ for the $(j,j')$-block. We obtain that
\begin{align}\label{eq:var_aux1}
&\left\|(\bfM_{(J,t,t')})_{\mathrm{up}}\odot (V^{t',t})_{\Lambda_J,\cdot} \right\| \le \bigg(\sum_{j_0\le j\le j' \le J}\bigg\|\left(\bfM_{(J,t,t')})_{\mathrm{up}}\odot (V^{t',t})_{\Lambda_J,\cdot}\right)_{j,j'}\bigg\|^2\bigg)^{1/2} \nonumber\\
&\le \bigg(\sum_{j_0\le j\le j' \le J}\bigg\|\left(\bfM_{(J,t,t')})_{\mathrm{up}}\odot (V^{t',t})_{\Lambda_J,\cdot}\right)_{j,j'}\bigg\|_{1}\cdot \bigg\|\left(\bfM_{(J,t,t')})_{\mathrm{up}}\odot (V^{t',t})_{\Lambda_J,\cdot}\right)_{j,j'}\bigg\|_{\infty}\bigg)^{1/2} \nonumber\\
&\le\bigg(\sum_{j_0\le j\le j' \le J} \mathrm{nnz}_{\mathrm{col}}\big(\big(\bfM_{(J,t,t')}\big)_{j,j'}\big)\cdot \mathrm{nnz}_{\mathrm{row}}\big(\big(\bfM_{(J,t,t')}\big)_{j,j'}\big)\cdot  \big\|\big(V^{t',t}\big)_{j,j'}\big\|^2_{\max}\bigg)^{1/2}\nonumber\\
&=\bigg(\sum_{j_0\le j\le j' \le J}  2^{-2jt'-2j't} \cdot (\mathrm{count}(j,j'))^2\cdot \| (\Var)_{j,j'}\|^2_{\max}\bigg)^{1/2},
\end{align}

where in the last line we used $(V^{t',t})_{j,j'}= 2^{-jt'-j't} (\Var)_{j,j'}$ and 
\[
\mathrm{count}(j,j'):=\left(\mathrm{nnz}_{\mathrm{col}}\big(\big(\bfM_{(J,t,t')}\big)_{j,j'}\big)\cdot \mathrm{nnz}_{\mathrm{row}}\big(\big(\bfM_{(J,t,t')}\big)_{j,j'}\big)\right)^{1/2}.
\]
By Proposition \ref{prop:sparsity_count}, we have
\begin{align}\label{eq:var_aux2}
    \mathrm{count}(j,j')\lesssim \begin{cases}
      0  & \text{if } (j,j')\in D_1\cup D_2, \\
      2^{(j-j')n/2}  & \text{if } (j,j')\in D_3, \\
      2^{(j'-j)n/2}  & \text{if } (j,j')\in D_4, \\
      2^{(j+j')n/2} 2^{n(J(t+t'-r)-jt'-j't-(j+j
')\widetilde{d})/(2\widetilde{d}+r)}  &\text{if } (j,j')\in D_5, \\
       2^{(j+j')n/2} & \text{if } (j,j')\in D_6.
    \end{cases}
\end{align}

Next, we derive a high-probability upper bound on $\|(\Var)_{j,j'}\|_{\max}$ for all $j_0\le j,j'\le J$. 

For $j_0\le j,j'\le J$ and $\lambda'=(j',k')$ with $k'\in \nabla_{j'}$, let $(\Var)_{j,\lambda'}=(\Var)_{j,(j',k')}\in \R^{\nabla_j}$ denote the $\lambda'$-th column of $(\Var)_{j,j'}$. Then,
\[
\|(\Var)_{j,j'}\|_{\max}=\max_{k'\in \nabla_{j'}}\|(\Var)_{j,\lambda'}\|_{\max}.
\]

Recall that the $\lambda'$-th column of the variance term $\Var$ is given by 
\[
(\Var)_{\cdot,\lambda'}= E_{\Omega_{\lambda'}}(\bfU_{\cdot,\Omega_{\lambda'}}^{\top} \bfU_{\cdot,\Omega_{\lambda'}})^{-1} \bfU_{\cdot,\Omega_{\lambda'}}^{\top}\bfW_{\cdot,\lambda'}\in \R^{\Lambda_{\widetilde{J}}},
\]
where $\bfW_{\cdot,\lambda'}\sim \mcN(0,\sigma^2_{\lambda'} I_{N})$ is independent of $\bfU$, and $\sigma_{\lambda'}\asymp 2^{-j'r_2}$ by Proposition \ref{prop:property_UAW} \ref{diagonal_preconditioning} (iii). Conditioning on $\bfU$, 
\[
(\bfU_{\cdot,\Omega_{\lambda'}}^{\top} \bfU_{\cdot,\Omega_{\lambda'}})^{-1} \bfU_{\cdot,\Omega_{\lambda'}}^{\top}\bfW_{\cdot,\lambda'} \,\mid\, \bfU \sim \mcN\left(0,\sigma^2_{\lambda'}(\bfU_{\cdot,\Omega_{\lambda'}}^{\top} \bfU_{\cdot,\Omega_{\lambda'}})^{-1}\right).
\]
Write $\widehat{\Sigma}_{\Omega_{\lambda'},\Omega_{\lambda'}}:=\frac{1}{N}\bfU_{\cdot,\Omega_{\lambda'}}^{\top} \bfU_{\cdot,\Omega_{\lambda'}}$ and $\widehat{\bar{\Sigma}}_{\Omega_{\lambda'},\Omega_{\lambda'}}:=\bfD^{r_1}_{\Omega_{\lambda'}}\widehat{\Sigma}_{\Omega_{\lambda'},\Omega_{\lambda'}} \bfD^{r_1}_{\Omega_{\lambda'}}$. Hence,
\[
(\bfU_{\cdot,\Omega_{\lambda'}}^{\top} \bfU_{\cdot,\Omega_{\lambda'}})^{-1} \bfU_{\cdot,\Omega_{\lambda'}}^{\top}\bfW_{\cdot,\lambda'} \,\mid\, \bfU \sim \mcN\left(0,\frac{\sigma^2_{\lambda'}}{N}\,\widehat{\Sigma}_{\Omega_{\lambda'},\Omega_{\lambda'}}^{-1}\right)=\mcN\left(0,\frac{\sigma^2_{\lambda'}}{N}\,\bfD^{r_1}_{\Omega_{\lambda'}}\widehat{\bar{\Sigma}}_{\Omega_{\lambda'},\Omega_{\lambda'}}^{-1}\bfD^{r_1}_{\Omega_{\lambda'}}\right).
\]

Observe that the vector $(\Var)_{j,\lambda'}\in \R^{\nabla_j}$ is the subvector of $(\Var)_{\cdot,\lambda'}$ obtained by restricting the row index to the $j$-th block. Therefore, conditioning on $\bfU$, the nonzero entries in $(\Var)_{j,\lambda'}$ has the same distribution as $v_{j,\lambda'}$, where
\[
v_{j,\lambda'}\sim \mathcal{N}\bigg(0,\frac{2^{2jr_1}\sigma^2_{\lambda'}}{N}\big(\widehat{\bar{\Sigma}}_{\Omega_{\lambda'},\Omega_{\lambda'}}^{-1}\big)_{j,j}\bigg).
\]
Let $ \bar{v}_{j,\lambda'}:=v_{j,\lambda'}\big(\frac{2^{jr_1}\sigma_{\lambda'}}{\sqrt{N}}\big)^{-1}$, then $\bar{v}_{j,\lambda'}\sim \mcN\big(0,\big(\widehat{\bar{\Sigma}}_{\Omega_{\lambda'},\Omega_{\lambda'}}^{-1}\big)_{j,j}\big)$. 

Fix $\delta\in (0,1)$ and set 
\[\delta_{j,\lambda'}:=2^{-j'n}J^{-2}\delta,\quad \lambda'=(j',k').
\]
Since $\bar{\Sigma}_{\Omega_{\lambda'},\Omega_{\lambda'}}$ is well-conditioned with eigenvalues of order one, Lemma \ref{lemma:tech1} implies that, provided $N\gtrsim |\Omega_{\lambda'}|+\log(1/\delta_{j,\lambda'})$, we have with probability at least $1-\delta_{j,\lambda'}$,
\[
\|\bar{v}_{j,\lambda'}\|_{\max}\lesssim \sqrt{\log\big(\mathrm{dim}(\bar{v}_{j,\lambda'})/\delta_{j,\lambda'}\big)} \lesssim  \sqrt{\log\big(2^{jn}/\delta_{j,\lambda'}\big)}\lesssim \sqrt{j}+\sqrt{\log(1/\delta_{j,\lambda'})}.
\]
Taking a union bound over $k'\in \nabla_{j'}$, if $N\gtrsim \max_{k'\in \nabla_{j'}}\big( |\Omega_{\lambda'}|+\log(1/\delta_{j,\lambda'})\big)$, then  with probability at least $1-\sum_{k'\in\nabla_{j'}}\delta_{j,\lambda'}$,
\[
\max_{k'\in \nabla_{j'}}\|\bar{v}_{j,\lambda'}\|_{\max}\lesssim \sqrt{j}+\sqrt{\log (1/\min_{k'\in \nabla_{j'}}\delta_{j,\lambda'})}.
\]
Finally, applying a union bound over all $(j,j')$–blocks, if \[
N\gtrsim \max_{j,j'} \left( \max_{k'\in \nabla_{j'}} \big(|\Omega_{\lambda'}|+\log(1/\delta_{j,\lambda'})\big)\right)\asymp \max_{\lambda'}|\Omega_{\lambda'}|+J+\log(1/\delta),
\]
then with probability at least $1-\sum_{j,j'}\sum_{k'\in\nabla_{j'}}\delta_{j,\lambda'}\ge 1-\delta$, we have simultaneously for all $j,j'$,
\[
\max_{k'\in \nabla_{j'}}\|\bar{v}_{j,\lambda'}\|_{\max}\lesssim \sqrt{j}+\sqrt{\log (1/\min_{k'\in \nabla_{j'}}\delta_{j,\lambda'})}\lesssim \sqrt{J+\log(1/\delta)}.
\]
Recall that $\max_{\lambda'}|\Omega_{\lambda'}|$ denotes the maximal column-wise regression support size over $\lambda'\in \Lambda_J$. Since the regression support is chosen as $\supp(\widetilde{J},\widetilde{t},\widetilde{t}')$ defined in \eqref{eq:supp_Jtt'}, we have
\[
\max_{\lambda'}|\Omega_{\lambda'}|=\mathrm{nnz}_{\mathrm{col}}\left(\supp(\widetilde{J},\widetilde{t},\widetilde{t}')\right)\lesssim 2^{\widetilde{J}n(\widetilde{t}+\widetilde{t}'-r)/(\sigma-\frac{n}{2}+\widetilde{t}'-\frac{r}{2})},
\]
as stated in Proposition~\ref{prop:sparsity_count}.

Consequently, if 
\[
N\gtrsim 2^{\widetilde{J}n(\widetilde{t}+\widetilde{t}'-r)/(\sigma-n/2+\widetilde{t}'-r/2)}+\log (1/\delta),
\]
then with probability at least $1-\delta$, simultaneously for all $j,j'$,
\begin{align}\label{eq:var_aux3}
\|(\Var)_{j,j'}\|_{\max}&=\max_{k'\in \nabla_{j'}}\|(\Var)_{j,\lambda'}\|_{\max}\nonumber\\
&\le \max_{k'\in\nabla_{j'}} \frac{2^{jr_1}\sigma_{\lambda'}}{\sqrt{N}}\|\bar{v}_{j,\lambda'}\|_{\max}\nonumber\\
&\lesssim \frac{2^{jr_1-j'r_2}}{\sqrt{N}}\max_{k'\in \nabla_{j'}}\|\bar{v}_{j,\lambda'}\|_{\max}\nonumber\\
&\lesssim  \frac{2^{jr_1-j'r_2}}{\sqrt{N}}\cdot \sqrt{J+\log(1/\delta)}.
\end{align}

Let $D_{\mathrm{up}}:=\{(j,j'):0\le j\le j'\le J\}$ and $D_{\mathrm{low}}:=\{(j,j'):0\le j'< j\le J\}$. Combining the estimates \eqref{eq:var_aux1}, \eqref{eq:var_aux2}, and \eqref{eq:var_aux3} yields that, if $N\gtrsim 2^{\widetilde{J}n(\widetilde{t}+\widetilde{t}'-r)/(\sigma-n/2+\widetilde{t}'-r/2)}+\log (1/\delta)$, then with probability at least $1-\delta$,
\begin{align}\label{eq:var_proof_1}
    &\left\|(\bfM_{(J,t,t')})_{\mathrm{up}}\odot (V^{t',t})_{\Lambda_J,\cdot} \right\| \le \bigg(\sum_{j_0\le j\le j' \le J}  2^{-2jt'-2j't} \cdot (\mathrm{count}(j,j'))^2\cdot \| (\Var)_{j,j'}\|^2_{\max}\bigg)^{1/2} \nonumber\\
    &\le  \sqrt{\frac{J+\log (1/\delta)}{N}}\bigg(\sum_{(j,j')\in D_{\mathrm{up}}}2^{2j(r_1-t')-2j'(r_2+t)}\cdot(\mathrm{count}(j,j'))^2\bigg)^{1/2} \nonumber\\
    &\lesssim \sqrt{\frac{J+\log(1/\delta)}{N}} \bigg(\sum_{(j,j')\in D_4}2^{2j(r_1-t')-2j'(r_2+t)}\cdot 2^{(j'-j)n}+ \sum_{(j,j')\in D_{\mathrm{up}}\cap D_6}2^{2j(r_1-t')-2j'(r_2+t)}\cdot 2^{(j+j')n} \nonumber\\
    &\quad + \sum_{(j,j')\in D_{\mathrm{up}}\cap D_5} 2^{2j(r_1-t')-2j'(r_2+t)}\cdot 2^{(j+j')n} 2^{2n(J(t+t'-r)-jt'-j't-(j+j
')\widetilde{d})/(2\widetilde{d}+r)}\bigg)^{1/2} \nonumber\\
&\lesssim \sqrt{\frac{J+\log(1/\delta)}{N}}\cdot J\cdot\left(1+2^{J\frac{t+t'-r}{\sigma-n/2+t-r/2}(-t-r_2+n/2)}+2^{J\frac{t+t'-r}{2\widetilde{d}+t+t'}(-t-t'+r_1-r_2+n)}+2^{J(-t-t'+r_1-r_2)}\right),
\end{align}
where, in the last step, we repeatedly used the standard exponential–sum estimate
\[
\sum_{j=k_1}^{k_2} 2^{j\alpha} \asymp \begin{cases}
 2^{k_2 \alpha}  & \text{ if } \alpha>0,\\
  k_2-k_1 &  \text{ if } \alpha=0,\\
 2^{k_1}\alpha  & \text{ if } \alpha<0,\\
\end{cases}
\]
and, moreover, since the regions $D_1$-$D_6$ are delineated by straight lines, the dominant contributions to the sum arise from the pairs $(j,j')$ lying on the boundary lines and at the corner points. This observation reduces the extraction of the leading-order terms to checking finitely many boundary/corner contributions. Indeed, the four terms inside the parentheses in the last line of \eqref{eq:var_proof_1} correspond to the contributions from the following four blocks in the $(j,j')$--plane: 
\[
(0,0),\quad \bigg(0,\ \frac{t+t'-r}{\sigma-\frac{n}{2}+t-\frac{r}{2}}J\bigg), \quad \bigg(\frac{t+t'-r}{2\widetilde{d}+t+t'}J,\ \frac{t+t'-r}{2\widetilde{d}+t+t'}J\bigg), \quad  (J,J).
\]

Similarly, for the second term in \eqref{eq:var_proof}, we have
\begin{align}\label{eq:var_proof_2}
   & \left\|((\bfM_{(J,t,t')})_{\mathrm{low}})^{\top}\odot (V^{t,t'})_{\Lambda_J,\cdot}  \right\|=  \left\|(\bfM_{(J,t,t')})_{\mathrm{low}}\odot (V^{t,t'})^{\top}_{\Lambda_J,\cdot}  \right\| \nonumber\\
   &\le \bigg(\sum_{j_0\le j'< j \le J}  2^{-2jt'-2j't} \cdot (\mathrm{count}(j,j'))^2\cdot \| (\Var)_{j',j}\|^2_{\max}\bigg)^{1/2}\nonumber\\
   &\lesssim  \sqrt{\frac{J+\log (1/\delta)}{N}}\bigg(\sum_{(j,j')\in D_{\mathrm{low}}}2^{2j'(r_1-t)-2j(r_2+t')}\cdot(\mathrm{count}(j,j'))^2\bigg)^{1/2} \nonumber\\
   &\lesssim \sqrt{\frac{J+\log(1/\delta)}{N}} \bigg(\sum_{(j,j')\in D_3}2^{2j'(r_1-t)-2j(r_2+t')}\cdot 2^{(j-j')n}+ \sum_{(j,j')\in D_{\mathrm{low}}\cap D_6}2^{2j'(r_1-t)-2j(r_2+t')}\cdot 2^{(j+j')n} \nonumber\\
    &\quad + \sum_{(j,j')\in D_{\mathrm{low}}\cap D_5} 2^{2j'(r_1-t)-2j(r_2+t')}\cdot 2^{(j+j')n} 2^{2n(J(t+t'-r)-jt'-j't-(j+j
')\widetilde{d})/(2\widetilde{d}+r)}\bigg)^{1/2} \nonumber\\
&\lesssim \sqrt{\frac{J+\log(1/\delta)}{N}}\cdot J \cdot\left(1+2^{J\frac{t+t'-r}{\sigma-n/2+t'-r/2}(-t'-r_2+n/2)}+2^{J\frac{t+t'-r}{2\widetilde{d}+t+t'}(-t-t'+r_1-r_2+n)}+2^{J(-t-t'+r_1-r_2)}\right).
\end{align}

Therefore, combining \eqref{eq:var_proof}, \eqref{eq:var_proof_1}, and \eqref{eq:var_proof_2}, we conclude that, if 
\[
N\gtrsim 2^{\widetilde{J}n(\widetilde{t}+\widetilde{t}'-r)/(\sigma-n/2+\widetilde{t}'-r/2)}+\log (1/\delta),
\]
then with probability at least $1-\delta$,
\begin{align}
    &\ErrorVar=\left\|(\bfM_{(J,t,t')})_{\mathrm{up}}\odot (V^{t',t})_{\Lambda_J,\cdot} \right\| + \left\|((\bfM_{(J,t,t')})_{\mathrm{low}})^{\top}\odot (V^{t,t'})_{\Lambda_J,\cdot}  \right\| \nonumber\\
    &\lesssim  \sqrt{\frac{J+\log(1/\delta)}{N}}\cdot J\cdot\Bigg(1+2^{J\frac{t+t'-r}{\sigma-n/2+t'-r/2}(-t'-r_2+n/2)}+2^{J\frac{t+t'-r}{\sigma-n/2+t-r/2}(-t-r_2+n/2)} \nonumber\\
    &\quad +2^{J\frac{t+t'-r}{2\widetilde{d}+t+t'}(-t-t'+r_1-r_2+n)}+2^{J(-t-t'+r_1-r_2)}\Bigg) \nonumber\\
    &\asymp \sqrt{\frac{J+\log(1/\delta)}{N}}\cdot J\cdot 2^{\rho J(t+t'-r)/2},
\end{align}
where
\[
\rho=2\max\left\{ \frac{-t-r_2+n/2}{\sigma-n/2+t-r/2},\, \frac{-t'-r_2+n/2}{\sigma-n/2+t'-r/2},\,\frac{-t-t'+r_1-r_2+n}{t+t'+2\widetilde{d}} ,\,\frac{-t-t'+r_1-r_2}{t+t'-r},0\right\}.
\]
 The leading contributions to the final variance arise from the following five blocks in the $(j,j')$--plane:
\[
(0,0),\ \ \bigg(0,\ \frac{t+t'-r}{\sigma-\frac{n}{2}+t-\frac{r}{2}}J\bigg), \ \ \bigg(\frac{t+t'-r}{\sigma-\frac{n}{2}+t'-\frac{r}{2}}J,\ 0\bigg), \ \ \bigg(\frac{t+t'-r}{2\widetilde{d}+t+t'}J,\ \frac{t+t'-r}{2\widetilde{d}+t+t'}J\bigg), \ \  (J,J).
\]
\end{proof}

\subsection{Technical Lemmas}\label{subsec:technical_lemmas}

\begin{lemma}[Norm compression inequality]\label{lemma:norm_compression}
    Let $A\in\mathbb{R}^{p\times p}$ and let $(I_1,\dots,I_G)$ be a partition of $\{1,\dots,p\}$ with $|I_j|=p_j$ and $p_1+\cdots+p_G=p$.
For $1\le i,j\le G$ write $A_{ij}:=A[I_i,I_j]\in\mathbb{R}^{p_i\times p_j}$ for the $(i,j)$ block.
Define the \emph{norm compression} of $A$ to be the matrix
\[
\mathcal{N}(A)\in\mathbb{R}^{G\times G},\qquad
\big(\mathcal{N}(A)\big)_{ij}:=\|A_{ij}\|,
\]
where $\|\cdot\|$ denotes the operator norm. For any $A\in\mathbb{R}^{p\times p}$ and any partition $(I_1,\dots,I_G)$ as above,
\[
\|A\| \le \|\mathcal{N}(A)\| \le  \bigg(\sum_{i,j=1}^{G}\|A_{ij}\|^2\bigg)^{1/2}.
\]
\end{lemma}
\begin{proof}[Proof of Lemma \ref{lemma:norm_compression}]
Let $x\in\mathbb{R}^p$ with $\|x\|_2=1$ and decompose $x=(x_1;\dots;x_G)$ with $x_j\in\mathbb{R}^{p_j}$.
Then the $i$th block of $Ax$ is $\sum_{j=1}^G A_{ij}x_j$, so by the triangle inequality and the definition of the spectral norm,
\[
\| (Ax)_i \|_2 \;\le\; \sum_{j=1}^G \|A_{ij}\|\,\|x_j\|_2
\;=\; \sum_{j=1}^G \big(\mathcal{N}(A)\big)_{ij}\, y_j,
\]
where $y=(y_1,\dots,y_G)^\top$ with $y_j:=\|x_j\|_2$.
Therefore,
\[
\|Ax\|_2^2=\sum_{i=1}^G \|(Ax)_i\|_2^2
\;\le\; \sum_{i=1}^G \Big(\sum_{j=1}^G \big(\mathcal{N}(A)\big)_{ij} y_j\Big)^2
= \|\mathcal{N}(A)\,y\|_2^2
\;\le\; \|\mathcal{N}(A)\|^2 \,\|y\|_2^2
\;\le\; \|\mathcal{N}(A)\|^2,
\]
since $\|y\|_2^2=\sum_j \|x_j\|_2^2=\|x\|_2^2=1$.
Taking the supremum over all unit vectors $x$ gives the first inequality. The second inequality in the claim follows from
\[
\|\mathcal{N}(A)\| \le \|\mathcal{N}(A)\|_{F}=   \bigg(\sum_{i,j=1}^{G}\|A_{ij}\|^2\bigg)^{1/2}.
\]
\end{proof}

\begin{lemma}\label{lemma:empirical_cross_term}
Let $(X,Z)$ be jointly mean-zero Gaussian with $X\in\mathbb{R}^d$ and $Z\in\mathbb{R}$. 
Write
\[
\mu := \E[XZ]\in\R^d, \qquad 
\sigma_Z^2 := \E[Z^2], \qquad
\Sigma_{XX}:=\E[XX^\top].
\]
Let $(X_i,Z_i)_{i=1}^{N}$ be i.i.d.\ copies of $(X,Z)$. 
Then, for any $\delta\in(0,1)$, with probability at least $1-\delta$,
\[
\bigg\|\frac{1}{N}\sum_{i=1}^N X_i Z_i - \mu \bigg\|_2
\le
C \big(\sigma_Z \|\Sigma_{XX}\|^{1/2} + \|\mu\|_2\big)\!\left(
\sqrt{\frac{d+\log(2/\delta)}{N}}
+
\frac{d+\log(2/\delta)}{N}
\right),
\]
for a universal constant $C>0$.
\end{lemma}
\begin{proof}[Proof of Lemma \ref{lemma:empirical_cross_term}]
Fix a unit vector $u\in \mathbb{S}^{d-1}$. Since both $\langle u,X\rangle$ and $Z$ are centered Gaussian, with
$\|\langle u,X\rangle\|_{\psi_2} \lesssim \|\Sigma_{XX}\|^{1/2}$,
$\|Z\|_{\psi_2} \lesssim \sigma_Z$, hence their product is sub-exponential
(see, e.g., \cite[Lemma 2.7.7]{vershynin2018high}) and
\begin{align*}
\|\langle u, XZ - \mu\rangle\|_{\psi_1} 
\lesssim \|\langle u,X\rangle Z\|_{\psi_1} + |\langle u,\mu\rangle|
\lesssim \sigma_Z \|\Sigma_{XX}\|^{1/2} + \|\mu\|_2.
\end{align*}
Let $W_i(u):=\langle u, X_i Z_i - \mu\rangle$. By the scalar Bernstein inequality for independent sub-exponential variables
(\cite[Theorem 2.8.1]{vershynin2018high}), for any $t>0$,
\begin{equation}\label{eq:bernstein-1d}
\mathbb{P}\left[\bigg|\frac1N \sum_{i=1}^N W_i(u)\bigg|
\ge
C(\sigma_Z \|\Sigma_{XX}\|^{1/2} + \|\mu\|_2)\bigg(\sqrt{\frac{t}{N}} + \frac{t}{N}\bigg)
\right]
\le 2 e^{-t}.
\end{equation}
Let $\mathcal{N}$ be a $\frac{1}{2}$-net of $\mathbb{S}^{d-1}$ with $|\mathcal{N}|\le 5^d$. 
By a union bound over $\mathcal{N}$ and \eqref{eq:bernstein-1d} with $t=d+\log(2/\delta)$,
we obtain with probability at least $1-\delta$,
\[
\max_{u\in\mathcal{N}} \bigg|\frac{1}{N}\sum_{i=1}^N W_i(u)\bigg|
\le C \big(\sigma_Z \|\Sigma_{XX}\|^{1/2} + \|\mu\|_2\big)\left(
\sqrt{\frac{d+\log(2/\delta)}{N}}
+
\frac{d+\log(2/\delta)}{N}
\right).
\]
Finally, a standard covering argument shows that for any $v\in\mathbb{R}^d$,
\[
\|v\|_2 \le 2 \max_{u\in\mathcal{N}} \langle u, v\rangle.
\]
Applying this with 
\(v = \frac{1}{N}\sum_{i=1}^N X_i Z_i - \mu\)
yields the desired Euclidean-norm bound and completes the proof.
\end{proof}

\begin{lemma}\label{lemma:tech1}
Let $X\sim\mcN(0,\Sigma)$ in $\R^d$ with $c I_d \preceq \Sigma \preceq C I_d$ for some constants $0<c\le C<\infty$, and let $X_1,\ldots,X_N$ be i.i.d.\ copies of $X$. Define $\widehat{\Sigma}:=\frac{1}{N}\sum_{i=1}^N X_i X_i^\top.$ Fix $\delta\in(0,1)$. If $N\gtrsim (\frac{C}{c})^2(d+\log(1/\delta))$, then with probability at least $1-\delta/2$,
\[
\sigma_{\min}(\widehat{\Sigma})\ge c/2,
\]
and in particular $\widehat{\Sigma}$ is invertible. Let $\Omega\subset \{1,\ldots,d\}$, and conditioned on $\widehat{\Sigma}$, let $v\sim \mcN(0,(\widehat{\Sigma}^{-1})_{\Omega,\Omega})$. Then, with probability at least $1-\delta$ jointly over $(X_1,\ldots, X_N,v)$,
\[
\|v\|_{\max}\lesssim \sqrt{\frac{\log(|\Omega|/\delta)}{c}}.
\]
\end{lemma}

\begin{proof}[Proof of Lemma \ref{lemma:tech1}]

By the standard operator-norm deviation bound for Gaussian sample covariances \cite{koltchinskii2017concentration},
\[
\mathbb{P}\bigg[\,\|\hat\Sigma-\Sigma\|
\lesssim \|\Sigma\|\Big(\sqrt{\tfrac{d+\log(2/\delta)}{N}}
+\tfrac{d+\log(2/\delta)}{N}\Big)\bigg] \ge 1-\tfrac{\delta}{2}.
\]
If $N\gtrsim (\frac{C}{c})^2 (d+\log(1/\delta))$, the deviation term is at most $\frac{1}{2}\sigma_{\min}(\Sigma)$. By Weyl's inequality,
\[
\sigma_{\min}(\widehat{\Sigma}) \ge \sigma_{\min}(\Sigma)-\|\widehat{\Sigma}-\Sigma\|
 \ge \frac{1}{2}\sigma_{\min}(\Sigma) \ge c/2
\]
with probability at least $1-\delta/2$. Denote this event by $\mathcal{E}$; in particular, $\widehat{\Sigma}$ is invertible on~$\mathcal{E}$.

Condition on $\hat\Sigma$ and let $v\sim\mathcal N(0,(\hat\Sigma^{-1})_{\Omega,\Omega})$.
Since $v=(v_j)_{j\in \Omega}$ has coordinates $v_j\sim \mcN(0,(\widehat{\Sigma}^{-1})_{jj})$ for $j\in \Omega$, we have
\[
\mathbb{P}\left[\|v\|_{\max}>t\,\big|\,\hat\Sigma\right] \le \sum_{j\in \Omega} 2\exp\!\Big(-\tfrac{t^2}{2(\hat\Sigma^{-1})_{jj}}\Big)
\le 2|\Omega|\exp\!\Big(-\tfrac{t^2\,\sigma_{\min}(\hat\Sigma)}{2}\Big).
\]
Taking $t=\sqrt{\frac{2\log(4|\Omega|/\delta)}{\sigma_{\min}(\widehat{\Sigma})}}$ gives conditional probability at most $\delta/2$. On~$\mathcal{E}$, $\sigma_{\min}(\widehat{\Sigma})\ge c/2$, hence
\[
\|v\|_{\max} \le \sqrt{\frac{4\log(4|\Omega|/\delta)}{c}}\lesssim \sqrt{\frac{\log(|\Omega|/\delta)}{c}}
\]
with conditional probability at least $1-\delta/2$.

Intersecting the two events yields the result with probability at least $1-\delta$ jointly over $(X_1,\ldots,X_N,v)$.
\end{proof}

\section{Auxiliary Materials Section \ref{sec:stability}}\label{app:stability}

We explain why the modified estimator defined through the updated support $\supp^{\mathrm{new}}(J,t,t')$ in \eqref{eq:supp_Jtt'_new} can yield the enhanced bound \eqref{eq:enhanced_estimate_0}.
  
  A direct modification of the proof of Proposition~\ref{prop:Error-Compression} gives
\[
\ErrorCompression\lesssim \varepsilon 2^{-J(t+t'-r)}.
\]
Indeed, recall that
\[
\ErrorCompression=\|\bfD^{-t'}_{\Lambda_{J}}(\bfA^{\varepsilon}_{\Lambda_J}-\bfA_{\Lambda_J})\bfD^{-t}_{\Lambda_{J}}\|\le \bigg(\sum_{j,j'= j_0}^{J} 2^{-2jt'-2j't} \|(\bfA^{\varepsilon}_{\Lambda_J})_{j,j'}-(\bfA_{\Lambda_J})_{j,j'}\|^2\bigg)^{1/2}.
\]
For indices $(j,j')$ where the new thresholding rule is performed, we have
\begin{align*}
\|(\bfA^{\varepsilon}_{\Lambda_J})_{j,j'}-(\bfA_{\Lambda_{J}})_{j,j'}\|\lesssim 2^{-(j+j')\widetilde{d}} \cdot (\tau_{jj'}^{\mathrm{new}})^{-(2\widetilde{d}+r)} \le  \frac{\varepsilon}{J}2^{-J(t+t'-r)+jt'+j't},
\end{align*}
and for discarded indices
\begin{align*}
\|(\bfA^{\varepsilon}_{\Lambda_J})_{j,j'}-(\bfA_{\Lambda_{J}})_{j,j'}\|=\|(\bfA_{\Lambda_{J}})_{j,j'}\|\lesssim 2^{(j+j')r/2-|j-j'|(\sigma-n/2)}  \le  \frac{\varepsilon}{J}2^{-J(t+t'-r)+jt'+j't}.
\end{align*}
Substituting these bounds above yields the claimed estimate for $\ErrorCompression$.

Next, observe that the new thresholding parameters and slope conditions differ from those in~\eqref{eq:supp_Jtt'} only by polynomial factors in $J$ and by at most $O(\log J)$ shifts in scale. Consequently, the arguments in Proposition~\ref{prop:Error-OVB} and Proposition~\ref{prop:Error-Var} carry over with minor changes.

For Proposition~\ref{prop:Error-OVB}, the modifications consist of verifying the support-inclusion relations (as in Lemma~\ref{lemma:supp_monotonicity}) under the new thresholding rule and slope conditions, and choosing an appropriately adjusted $\widetilde J$. This gives
\[
\ErrorOVB \lesssim \varepsilon 2^{-J(t+t'-r)}.
\]
For the variance term, one adapts the support-pattern analysis in Proposition \ref{prop:sparsity_count} to the new support $\supp^{\mathrm{new}}(J,t,t')$. With this in place, the proof of Proposition~\ref{prop:Error-Var} extends directly and yields
\[
\ErrorVar\lesssim \sqrt{\frac{\log(1/\delta)}{N}}\cdot (J/\varepsilon)^{\kappa}\cdot 2^{\rho J(t+t'-r)/2},
\]
for some constant $\kappa>0$, where $\rho$ is the same exponent as in~\eqref{eq:rho}.
We omit the technical details for brevity. 

Finally, combining the bounds for $\ErrorCompression,\ErrorOVB$, and $\ErrorVar$ yields the estimate \eqref{eq:enhanced_estimate_0}.

\end{document}